\documentclass[a4paper]{article}

\usepackage{latexsym}
\usepackage{a4wide}
\usepackage{amscd}
\usepackage{graphics}
\usepackage{amsmath}
\usepackage{amssymb}
\usepackage{mathrsfs}
\usepackage{amsthm}
\usepackage{bbm}
\input xy
\xyoption{all}

\newcommand{\N}{\mathbb{N}}                     
\newcommand{\Z}{\mathbb{Z}}                     
\newcommand{\R}{\mathbb{R}}                     
\newcommand{\C}{\mathbb{C}}                     
\newcommand{\T}{\mathbb{T}}                     
\newcommand{\set}[2]{\left\{{#1}\mid{#2}\right\}}       
\newcommand{\im}{\mathrm{Im\,}}                 
\newcommand{\re}{\mathrm{Re\,}}                 
\newcommand{\Ker}{\mathrm{Ker\,}}               
\newcommand{\coker}{\mathrm{coker\,}}           
\newcommand{\sgn}{\mathrm{sgn\,}}               
\newcommand{\ind}{\mathrm{ind\,}}               
\newcommand{\codim}{\mathrm{codim}}           
\newcommand{\crit}{\mathrm{crit}}		
\newcommand{\Det}{\mathrm{Det}}                 
\newcommand{\delbar}{\overline{\partial}_J\,}  
\DeclareMathOperator*{\esssup}{ess\,sup}

\newtheorem{thm}{\bf Theorem}[section]      
\newtheorem*{thm*}{\bf Theorem}			
\newtheorem{cor}[thm]{\bf Corollary}        
\newtheorem{lem}[thm]{\bf Lemma}            
\newtheorem{prop}[thm]{\bf Proposition}     
\newtheorem{defn}[thm]{\bf Definition}      
\newtheorem{rem}[thm]{\bf Remark}	    

\title{The role of the Legendre transform in the study of the Floer complex of cotangent bundles} 

\author{Alberto Abbondandolo and Matthias Schwarz}

\date{June 18, 2013}

\begin{document}

\maketitle

\begin{abstract}
Consider a classical Hamiltonian $H$ on the cotangent bundle $T^*M$ of a closed orientable manifold $M$, and let $L:TM \rightarrow \R$ be its Legendre-dual Lagrangian. In a previous paper we constructed an isomorphism $\Phi$ from the Morse complex of the Lagrangian action functional which is associated to $L$ to the Floer complex which is determined by $H$. In this paper we give an explicit construction of a homotopy inverse $\Psi$ of $\Phi$. Contrary to other previously defined maps going in the same direction, $\Psi$ is an isomorphism at the chain level and preserves the action filtration. Its definition is based on counting Floer trajectories on the negative half-cylinder which on the boundary satisfy ``half'' of the Hamilton equations. Albeit not of Lagrangian type, such a boundary condition defines Fredholm operators with good compactness properties.
We also present a heuristic argument which, independently on any Fredholm and compactness analysis, explains why the spaces of maps which are used in the definition of $\Phi$ and $\Psi$ are the natural ones. The Legendre transform plays a crucial role both in our rigorous and in our heuristic arguments. We treat with some detail the delicate issue of orientations and show that the homology of the Floer complex is isomorphic to the singular homology of the loop space of $M$ with a system of local coefficients, which is defined by the  pull-back of the second Stiefel-Whitney class of $TM$ on 2-tori in $M$. 
\end{abstract}

\renewcommand{\theenumi}{\roman{enumi}}
\renewcommand{\labelenumi}{(\theenumi)}

\section*{Introduction}

This paper is the natural continuation of our previous paper \cite{as06}. Let $M$ be a closed manifold, which for sake of simplicity we assume to be orientable. The cotangent bundle $T^*M$ of $M$ carries a canonical symplectic structure $\omega$.  A time-periodic $\omega$-compatible almost complex structure $J$ on $(T^*M,\omega)$ and a time-periodic Hamiltonian $H\in C^{\infty}(\T \times T^*M)$ define a chain complex, which is called the Floer complex of $H$. 
Here the Hamiltonian should be non-degenerate, meaning that all the one-periodic orbits of the corresponding Hamiltonian vector field $X_H$ are non-degenerate as fixed points of the time-one flow. The set of such periodic orbits is denoted by $\mathscr{P}(H)$. Moreover, the almost complex structure $J$ should be generic and both $H$ and $J$ should have a suitable behavior at infinity. For instance, $H$ should be quadratic at infinity (see Definition \ref{qai} below) and $J$ should be $C^0$-close to a Levi-Civita almost complex structure. Alternatively, $H$ should be fiberwise radial and superlinear and $J$ should be of contact-type outside of a compact set.

Under these assumptions, one defines $F_*(H)$ to be the free Abelian group which is generated by the elements of $\mathscr{P}(H)$, graded by the Conley-Zehnder index
\[
\mu_{CZ} :  \mathscr{P}(H) \rightarrow \Z.
\]
The boundary operator of the Floer chain complex
\[
\partial^F : F_*(H) \rightarrow F_{*-1}(H)
\]
is defined by counting solutions $u\in C^{\infty}(\T\times \R,T^*M)$ of the perturbed Cauchy-Riemann equation
\begin{equation}
\label{floereq}
\partial_s u + J(t,u) \bigl( \partial_t u - X_H(t,u) \bigr) = 0, \qquad \forall (s,t)\in \T\times \R,
\end{equation}
which have finite energy
\[
E(u) = \int_{\R \times \T} |\partial_s u|^2\, ds\, dt,
\]
and are asymptotic to pairs of periodic orbits for $s\rightarrow \pm \infty$. The homology of such a complex does not depend on the choice of $(H,J)$ and it can be therefore called the {\em Floer homology of $T^*M$} and denoted by $HF_*(T^*M)$. 

This construction is originally due to A.\ Floer for closed symplectic manifolds (see \cite{flo88b,flo88a,flo88d,flo89b,flo89a}). For cotangent bundles, it can be seen as a particular case of symplectic homology for Liouville domains (see \cite{fh94,cfh95,vit99,oan04,sei06b,bo09b}). More precisely, $HF_*(T^*M)$ can be identified with the symplectic homology of the unit cotangent disk bundle $D^*M$. 

In our previous paper \cite{as06} we considered Hamiltonians which are quadratic at infinity and uniformly fiberwise convex (see condition (\ref{uncon}) below), and we constructed an explicit chain isomorphism 
\[
\Phi: M_*(\mathbb{S}) \rightarrow F_*(H)
\]
from the Morse complex $M_*(\mathbb{S})$ which is associated to the Lagrangian action functional
\[
\mathbb{S}(q) := \int_{\T} L\bigl(t,q(t),q'(t)\bigr)\, dt, \qquad \forall q\in H^1(\T,M),
\]
where $L\in C^{\infty}(\T\times TM)$ is the Lagrangian which is Legendre-dual to $H$, to the Floer complex $F_*(H)$. Here $H^1(\T,M)$ is the Hilbert manifold of free loops in $M$ of Sobolev class $H^1$ and the Morse complex of $\mathbb{S}$ is well-defined because $\mathbb{S}$ admits a smooth Morse-Smale negative gradient flow whose unstable manifolds are finite dimensional and which satisfies the Palais-Smale compactness condition (see \cite{as09}). Since the Morse homology of $\mathbb{S}$ is isomorphic to the singular homology of $H^1(\T,M)$, hence also to the singular homology of the free loop space $\Lambda M:= C^0(\T,M)$, the above result implies that the Floer homology of the cotangent bundle of $M$ is isomorphic to the singular homology of $\Lambda M$, a theorem which had been previously proven by C.\ Viterbo in \cite{vit03} and by D.\ Salamon and J.\ Weber in \cite{sw06} using generating functions and the heat flow for loops on manifolds, respectively (see also \cite{web05} for a comparison of the three approaches).

Actually, there is a subtle issue involving orientations which was overlooked in \cite{as06} and later corrected in \cite{as13}. This issue was highlighted by T.~Kragh in \cite{kra07}, and clarified in \cite{abo11}, \cite{kra11} and \cite{as13}. Namely, if one defines the boundary operator of the Floer complex using standard conventions, the Floer complex of $H$ is isomorphic to the Morse complex of $\mathbb{S}$ {\em with a system of local coefficients on $H^1(\T,M)$}. Such a system of local coefficients is induced by the following $\Z$-representation $\rho$ of the fundamental group of $\Lambda M$: if every continuous loop $\gamma:\T \rightarrow \Lambda M$ is identified with a continuous map $\tilde{\gamma}: \T^2 \rightarrow M$, then the representation $\rho$ is defined as
\[
\rho: \pi_1(\Lambda M) \rightarrow \mathrm{Aut}(\Z), \qquad \rho [\gamma] := \left\{ \begin{array}{cl} \mathrm{id} & \mbox{if } w_2(\tilde{\gamma}^* (TM)) = 0, \\ -\mathrm{id} & \mbox{if } w_2(\tilde{\gamma}^* (TM)) \neq 0, \end{array} \right.
\]
where $w_2$ denotes the second Stiefel-Whitney class. In particular, the standard Floer homology of $T^*M$ is isomorphic to the singular homology of $\Lambda M$ with the above system of local coefficients. This local system is trivial - and hence the homology of $\Lambda M$ is the standard one - when the second Whitney class of the oriented manifold $M$ vanishes on tori, and in particular when $M$ is spin. It is also possible to define a twisted version of the Floer complex of $H$ which is isomorphic to the standard Morse complex of $\mathbb{S}$ (this is the approach adopted in \cite{as13}, see also Remark \ref{twistrem} below). Here we prefer to work with the standard Floer complex and to use local coefficients on the loop space. In Sections \ref{linorsec}, \ref{defcompsec} and \ref{defisosec} below we treat orientations and local coefficients with some detail: in doing this, we follow A.~Floer's and H.~Hofer's approach from \cite{fh93}, but we adopt the point of view of P.~Seidel \cite[Sections II.11 and II.12]{sei08b} and M.~Abouzaid \cite{abo11}  of including orientations in the definition of the generators of the Floer and Morse complexes. In the remaining part of this introduction we prefer to ignore these orientation issues and we use periodic orbits and critical points as generators of the Floer and of the Morse complex.

The definition of the chain isomorphism $\Phi$ is based on the study of the following spaces of maps on the positive half-cylinder:
\begin{eqnarray*}
\mathscr{M}_{\Phi} (q,x) := \Bigl\{ u: \R^+ \times \T \rightarrow T^*M  & \Big| &  u \mbox{ is a finite energy solution of (\ref{floereq}), } \pi\circ u(0,\cdot) \in W^u(q;G_{\mathbb{S}})\\ && \mbox{and } 
\lim_{s\rightarrow +\infty} u(s,\cdot)=x\Bigr\},
\end{eqnarray*}
where $q$ belongs to $\mathrm{crit} \, \mathbb{S}$, the set of critical points of $\mathbb{S}$, and $x$ is in $\mathscr{P}(H)$. Here $\pi:T^*M \rightarrow M$ is the bundle projection and $W^u(q;G_{\mathbb{S}})$ is the unstable manifold of the critical point $q$ with respect to the negative pseudo-gradient vector field $G_{\mathbb{S}}$ which defines the Morse complex of $\mathbb{S}$ (hence $W^u(q;G_{\mathbb{S}})$ is a finite dimensional manifold of $M$-valued loops). The isomorphism $\Phi$ preserves the action filtrations of the two chain complexes, which on the Morse side is given by the Lagrangian action functional $\mathbb{S}$ and on the Floer side is given by the Hamiltonian action functional
\[
\mathbb{A}(x) := \int_{\T} x^*(\lambda) - \int_{\T} H(t,x(t))\, dt, \quad \forall x\in C^{\infty}(\T,T^*M),
\]
where $\lambda$ denotes the standard Liouville form on $T^*M$.
This construction turns out to be quite flexible and has been applied by several authors to various situations, both in symplectic geometry and in gauge theory (see \cite{aps08,as09,abo12,rit09,ase10,jan10,mer10,mer11,fmp12}). 

From the beginning, a natural question has been how to define a chain isomorphism going in the opposite direction
\[
\Psi : F_*(H) \rightarrow M_*(\mathbb{S}).
\]
Of course $\Phi^{-1}$ is such an isomorphism, but one would like to have a definition of $\Psi$ which is based on counting spaces of maps. The naif candidate which is obtained by replacing the positive half-cylinder by the negative one and the unstable manifolds by the stable manifolds obviously does not work: one one hand, stable manifolds are infinite dimensional, so one would not get a finite dimensional space of maps, on the other hand the crucial inequality
\begin{equation}
\label{fenineq}
\mathbb{A}(x) \leq \mathbb{S}(\pi\circ x), \qquad \forall x\in C^{\infty}(\T,T^*M),
\end{equation}
which follows from Legendre duality and allows us to get the energy estimates which lead to good compactness properties of the spaces $\mathscr{M}_{\Phi}(q,x)$, would be of no use here.

In the literature there is a construction of a chain map 
\[
\tilde{\Psi} :  F_*(H) \rightarrow M_*(\mathbb{S}),
\]
which at the level of homology induces an inverse of $\Phi_*$. Its definition  is based on the study of the space of maps
\begin{eqnarray*}
\mathscr{M}_{\tilde{\Psi}} (x,q) := \Bigl\{ u: \R^- \times \T \rightarrow T^*M  & \Big| &  u \mbox{ is a finite-energy solution of (\ref{floereq}), } \lim_{s\rightarrow -\infty} u(s,\cdot)=x, \\ && u(0,t)\in \mathbb{O}_{T^*M} \mbox{ and } 
\pi\circ u(0,\cdot) \in W^s(q;G_{\mathbb{S}})\Bigr\},
\end{eqnarray*}
where $\mathbb{O}_{T^*M}$ denotes the zero-section of $T^*M$. This construction has been used in \cite{cl09,abo11,as12}. However $\tilde{\Psi}$ need not be an isomorphism at the chain level and it does not preserve the action filtrations of the two complexes.

The first aim of this paper is to define a map $\Psi:F_*(H) \rightarrow M_*(\mathbb{S})$ which is a chain isomorphism, preserves the action filtrations and is a chain homotopy inverse of $\Phi$. The definition of $\Psi$ is based on a space of maps which we call the {\em reduced unstable manifold} of the periodic orbit $x\in \mathscr{P}(H)$:
\begin{eqnarray*}
\mathscr{U}(x) := \Bigl\{ u: \R^- \times \T \rightarrow T^*M  & \Big| &  u \mbox{ is a finite-energy solution of (\ref{floereq}), } \lim_{s\rightarrow -\infty} u(s,\cdot)=x \\ &&\mbox{and } \partial_t (\pi \circ u) (0,t) = d_p H\bigl(t,u(0,t)\bigr) \; \forall t\in \T\Bigr\}.
\end{eqnarray*}
The name {\em unstable manifold} is motivated by the fact that (\ref{floereq}) is the $L^2$-negative gradient equation of the Hamiltonian action functional $\mathbb{A}$  and by the asymptotic condition for $s\rightarrow -\infty$. The adjective {\em reduced} refers to the fact that we are adding to these conditions, which would define an infinite dimensional object, the boundary condition
\begin{equation}
\label{bdryzero}
\partial_t (\pi \circ u) (0,t) = d_p H\bigl(t,u(0,t)\bigr), \qquad \forall t\in \T,
\end{equation}
which says that the loop $u(0,\cdot)$ solves ``half of the Hamiltonian equations'' (here $d_p H(t,x)$ denotes the fiberwise differential of $H(t,\cdot)$ at $x$, which is an element of $(T_{\pi(x)}^* M )^* = T_{\pi(x)} M$, and can therefore be  compared to the derivative of the loop $\pi\circ x:\T \rightarrow M$). Notice however that the set $\mathscr{U}(x)$ is not invariant with respect to the semi-flow
\[
(\sigma,u) \mapsto u(\cdot - \sigma,\cdot), \qquad \sigma\in \R^+,
\]
a property which one would instead expect from a true unstable manifold, and that its topology could be arbitrarily complicated.

Unlike those appearing in the definition of the spaces $\mathscr{M}_{\Phi}(q,x)$ and $\mathscr{M}_{\tilde{\Psi}}(x,q)$, the boundary condition (\ref{bdry}) is not of Lagrangian type.
In the first section of this paper we develop the linear Fredholm analysis which allows to see $\mathscr{U}(x)$ as the set of zeroes of a Fredholm section of a suitable Hilbert bundle. This linear analysis is ultimately based on the elementary identity of Lemma \ref{baside}. Thanks to this linear analysis and to the standard regularity theory for the Cauchy-Riemann operator, we can prove that if $J$ is generic and 
\begin{equation}
\label{positivo}
d_{pp} H\bigl(t,x(t)\bigr) >0 , \qquad \forall t\in \T, 
\end{equation}
then $\mathscr{U}(x)$ is a non-empty smooth orientable manifold of dimension $\mu_{CZ}(x)$ (notice that by (\ref{positivo}), the Conley-Zehnder index $\mu_{CZ}(x)$ is non-negative). The stationary solution $u_0(s,t):=x(t)$ always belongs to $\mathscr{U}(x)$, so transversality is impossible to achieve at $u_0$ when $\mu_{CZ}(x)$ is negative. Instead, when (\ref{positivo}) holds, there is an automatic transversality argument which shows that for every choice of $J$ the stationary solution $u_0$ is a regular point of the Fredholm section whose set of zeroes is $\mathscr{U}(x)$. Interestingly, such an argument is again based on the Legendre transform, which by (\ref{positivo}) can be defined locally around $x$ (see Lemma \ref{auttrans} below).

The next step is to show that that, under the above mentioned  asymptotic quadraticity assumption on $H$ (see Definition \ref{qai}) and assuming $J$ to be $C^0$-close to a Levi-Civita almost complex structure, all the maps in $\mathscr{U}(x)$ take values in a compact subset of $T^*M$ (see Proposition \ref{komp1} below). This fact follows from the energy estimate coming from the boundary condition (\ref{bdryzero}) and from an argument in \cite{as06} which involves elliptic estimates. Although in this paper we mainly deal with the above mentioned set of assumptions on $H$ and $J$, we also prove the analogous compactness statement under the assumptions which are more common in symplectic homology (see Proposition \ref{komp2} below). The proof of the latter result uses the maximum principle and the Hopf Lemma. Standard arguments involving bubbling-off of $J$-holomorhic spheres and disks and an elliptic bootstrap imply that under both set of assumptions the reduced unstable manifold $\mathscr{U}(x)$ is relatively compact in $C^{\infty}_{\mathrm{loc}} (\R^- \times \T,\T^*M)$.

Assume now that $H$ is both quadratic at infinity and uniformly fiberwise convex and let $L$ be its Legendre-dual Lagrangian. Given $x\in \mathscr{P}(H)$ and  $q\in \mathrm{crit}\, \mathbb{S}$, we consider the set of maps
\[
\mathscr{M}_{\Psi}(x,q) := \mathscr{Q}_x^{-1} \bigl(W^s(q;G_{\mathbb{S}}) \bigr),
\]
where $\mathscr{Q}_x$ is the map 
\[
\mathscr{Q}_x : \mathscr{U}(x) \rightarrow H^1(\T,T^*M), \qquad \mathscr{Q}_x(u) = \pi\circ u(0,\cdot).
\]
In other words, $\mathscr{M}_{\Psi}(x,q)$ is the set of elements $u$ of the reduced unstable manifold of $x$ such that $\pi\circ u(0,\cdot)$ belongs to the stable manifold of $q$ with respect to the negative pseudo-gradient flow of $G_{\mathbb{S}}$. Up to a generic perturbation of $G_{\mathbb{S}}$, the map $\mathscr{Q}_x$ is transverse to the stable manifold $W^s(q;G_{\mathbb{S}})$, so $\mathscr{M}_{\Psi}(x,q)$ is a manifold of dimension $\mu_{CZ}(x)-\mathrm{ind}\, (q;\mathbb{S})$, where $\mathrm{ind}\, (q;\mathbb{S})$ denotes the Morse index of the critical point $q$. Moreover, an orientation of $\mathscr{U}(x)$ and a co-orientation of $W^s(q;\mathbb{S})$ (that is, an orientation of its normal bundle in $H^1(\T,M)$) induce an orientation of $\mathscr{M}_{\Psi}(x,q)$. 

These facts imply that, in the case $\mu_{CZ}(x) = \mathrm{ind}\, (q;\mathbb{S})$, $\mathscr{M}_{\Psi}(x,q)$ is a compact oriented zero-dimensional manifold, that is a finite set of points each of which carries an orientation sign. The algebraic sum of this signs defines the coefficient $n_{\Psi}(x,q)$ of the chain map
\begin{equation}
\label{defpsi0}
\Psi: F_*(H) \rightarrow M_*(\mathbb{S}), \qquad \Psi x := \sum_{\substack{q\in \crit\, \mathbb{S} \\ \ind(q;\mathbb{S}) = \mu_{CZ}(x)}} n_{\Psi}(x,q)\, q, \qquad \forall x\in \mathscr{P}(H).
\end{equation}
Again by Legendre duality, the equality holds in (\ref{fenineq}) if and only if the loop $x:\T \rightarrow T^*M$ solves ``half of the Hamiltonian equations'',
\[
x'(t) = d_p H(t,x(t)), \qquad \forall t\in \T.
\]
This fact implies that $\Psi$ has the form
\[
\Psi x = \epsilon(x) \pi\circ x + \sum_{\mathbb{S}(q)< \mathbb{A}(x)} n_{\Psi}(x,q) q, \qquad \forall x\in \mathscr{P}(H),
\]
where $\epsilon(x)$ is either $1$ or $-1$. This leads to the following theorem, which is the main result of this paper:

\begin{thm*}
Let $M$ be a closed orientable manifold.
Assume that $L$ is the Lagrangian which is Legendre-dual to the fiberwise uniformly convex and quadratic at infinity non-degenerate Hamiltonian $H\in C^{\infty}(\T\times T^* M)$. Then:
\begin{enumerate} 
\item The formula (\ref{defpsi0}) defines a chain isomorphism $\Psi$ from the Floer complex of $H$ to the Morse complex of the Lagrangian action functional $\mathbb{S}$ associated to  $L$. Such an isomorphism preserves the action filtrations and the splittings of $F_*(H)$ and $M_*(\mathbb{S})$ determined by the free homotopy classes of the generators.
\item The chain isomorphisms $\Phi$ and $\Psi$ are homotopy inverses one of the other through chain homoopies which preserve the action filtrations and the splitting of the Floer and the Morse complexes determined by the free homotopy classes of the generators.
\end{enumerate}
\end{thm*}

We also remark that the arguments of \cite{as10} could be used to show that $\Psi$ induces an isomorphism in homology which intertwines the pair-of-pants product on $HF_*(T^*M)$ with the Morse-theoretical representation of the Chas-Sullivan loop product. Furthermore, the construction of $\Psi$ can be easily generalized to treat more general boundary conditions than the periodic one, see \cite{aps08}. 

The second aim of this paper is to present a heuristic argument which, independently on any Fredholm or compactness analysis, explains why 
$\mathscr{M}_{\Phi}(q,x)$ and $\mathscr{M}_{\Psi}(x,q)$ are the natural spaces of maps to be considered in order to define chain isomorphisms between the Floer complex of $H$ and the Morse complex of $\mathbb{S}$. We briefly describe this argument here, referring to Section \ref{heursec} for more details.

The starting point is again the Legendre transform, which induces the following diffeomorphism
\[
\mathcal{L} : C^{\infty}(\T,TM) \rightarrow C^{\infty}(\T,T^*M), \qquad \mathcal{L}(q,v) := \bigl( q, d_v L(\cdot,q,q'+v) \bigr),
\]
between the free loop spaces of $TM$ and $T^*M$. The Taylor formula with integral remainder implies the identity
\[
\mathbb{A}\bigl( \mathcal{L}(q,v) \bigr) = \mathbb{S}(q) - \mathbb{U}(q,v),
\]
where $\mathbb{U}$ is the function
\[
\mathbb{U} : C^{\infty}(\T,TM) \rightarrow \R, \qquad \mathbb{U}(q,v) :=  \int_{\T} \int_0^1 s\, d_{vv} L(t,q,q'+sv)[v]^2\, ds \, dt.
\]
In the particular case $L(t,q,v)=g_q(v,v)/2$, with $g$ a Riemannian metric on $M$, this formula has been used by M.\ Lipyanskiy in \cite{lip09}.
The function $\mathbb{U}$ has minimum 0, which is achieved at the loops which take values into the zero section $\mathbb{O}_{TM}$ of $TM$, and has no other critical point. It is therefore easy to construct a negative pseudo-gradient vector field $G_{\mathbb{A}\circ \mathcal{L}}$ for $\mathbb{A}\circ \mathcal{L}$ on $C^{\infty}(\T,TM)$ (or rather on a suitable Hilbert completion of this space, but we are ignoring this point here, see Section \ref{heursec} for a more precise discussion) whose Morse complex is precisely the Morse complex of $\mathbb{S}$: the space of loops  which take values into the zero section $\mathbb{O}_{TM}$ is invariant with respect to the corresponding flow and the orbit of every other loop tends  to such a space for $s\rightarrow -\infty$ and to infinity for $s\rightarrow +\infty$. By pushing this vector field by means of the map $\mathcal{L}$ we obtain a negative pseudo-gradient vector field $G_{\mathbb{A}}$ for $\mathbb{A}$ on  $C^{\infty}(\T,T^*M)$.

This means that on $C^{\infty}(\T,T^*M)$ we have two negative pseudo-gradient vector fields for $\mathbb{A}$: the first one is minus the $L^2$-gradient vector field
\[
- \nabla_{L^2} \mathbb{A} (u) = -J(t,u) \bigl( \partial_t u - X_H(t,u) \bigr),
\]
which does not induce a flow and whose associated chain complex is the Floer complex of $H$; the second one
is $G_{\mathbb{A}}$, which does induce a flow and whose Morse complex is precisely the Morse complex of $\mathbb{S}$. When a functional has two different negative pseudo-gradient vector fields $G_1$ and $G_2$, the natural way of defining a chain isomorphism between the two Morse complexes is to count the number of intersections between the unstable manifold with respect of $G_1$ and the stable manifold with respect to $G_2$ of pairs of critical points of the same index. Equivalently, one has to look at continuous curves $u$ which are asymptotic to the two critical points and satisfy
\[
u'(s) = G_1(u(s)) \quad \mbox{for } s\leq 0, \qquad u'(s) = G_2(u(s)) \quad \mbox{for } s\geq 0.
\]
If we apply this observation to the negative pseudo-gradient vector fields $G_{\mathbb{A}}$ and $- \nabla_{L^2} \mathbb{A}$, we find exactly the set $\mathscr{M}_{\Phi}(q,x)$ or $\mathscr{M}_{\Psi}(x,q)$, depending on which order we decide to follow. See Section \ref{heursec} for more details, and in particular Remark \ref{thom} for the interpretation of all this as an infinite dimensional version of the Thom isomorphism. 

\section{Linear results}

\paragraph{Sign conventions.} The sign convention we adopt here differ from those in \cite{as06} and they are more standard.
If $\omega$ is a symplectic form on the manifold $M$, the almost complex structure $J$ on $M$ is said to be $\omega$-compatible if the bilinear form
\[
g_J( \xi,\eta ) := \omega(J \xi,\eta), \qquad \forall \xi,\eta\in T_x M, \; \forall x\in M,
\]
is a Riemannian structure on $M$. The associated norm is denoted by
\[
|\xi|_J := \sqrt{g_J( \xi,\xi )}, \qquad \forall \xi\in TM.
\]
The vector space $T^* \R^n = \R^n \times (\R^n)^*$, whose elements are denoted by $(q,p)$, $q\in \R^n$, $p\in (\R^n)^*$, is endowed with the Liouville form
\[
\lambda_0 = p\, dq, \quad \mbox{that is } \lambda_0(q,p)[(\xi,\eta)] = \langle p, \xi \rangle, \quad \forall (q,p), (\xi,\eta) \in T^* \R^n,
\]
where $\langle\cdot,\cdot \rangle$ denotes the duality pairing. The exterior differential of $\lambda_0$ is the canonical symplectic form $\omega_0$ on $T^*M$,
\[
\omega_0 = d\lambda_0 = dp\wedge dq, \quad \mbox{that is } \omega_0 [ (\xi_1,\eta_1),(\xi_2,\eta_2) ] = \langle \eta_1,\xi_2 \rangle - \langle \eta_2,\xi_1 \rangle, \quad \forall (\xi_1,\eta_1), (\xi_2,\eta_2) \in T^* \R^n.
\]
The dual space $(\R^n)^*$ can be identified with $\R^n$ by the standard Euclidean product, and $T^*\R^n \cong \R^n \times \R^n=\R^{2n}$ is also endowed with the complex structure 
\[
J_0 (q,p) = (-p,q),
\]
which corresponds to the identification $\R^{2n} \cong \C^n$ given by $(q,p)\mapsto q+ip$. Such a complex structure $J_0$ is $\omega_0$-compatible and the corresponding scalar product is the standard Euclidean structure of $\R^{2n}$:
\[
g_{J_0}( \xi,\eta ) = \omega_0 (J_0 \xi,\eta) = \xi \cdot \eta, \qquad \forall \xi,\eta\in T^*\R^n \cong \R^{2n}.
\]

\paragraph{The basic identity.} Set $\R^-:=]-\infty,0]$ and denote the elements of the half-cylinder $\R^- \times \T$ as $(s,t)$. Given $J=J(s,t)$ a family of complex structures on $\R^{2n}$ parametrized on $\R^-\times \T$, denote by
\[
\delbar u := \partial_s  + J \partial_t
\] 
the corresponding Cauchy-Riemann operator on $\R^{2n}$-valued maps on the half-cylinder. We denote by $\mathrm{L}(\R^{2n})$ the vector space of linear endomorphisms of $\R^{2n}$.

\begin{lem}
\label{baside}
Let $J\in L^{\infty}(\R^- \times \T, \mathrm{L}(\R^{2n}))$ be such that $J(s,t)$ is an $\omega_0$-compatible complex structure on $\R^{2n}$, for a.e.\ $(s,t)\in \R^- \times \T$ . Then for every $u$ in the Sobolev space $H^1(\R^- \times \T,\R^{2n})$ there holds
\begin{equation}
\label{ide}
\int_{\R^-\times \T} \bigl( |\partial_s u|_J^2 + |\partial_t u|_J^2 \bigr)\, ds\,dt = \int_{\R^-\times \T} |\overline{\partial}_J u|_J^2 \, ds \,dt - 2 \int_{\T} u(0,\cdot)^* \lambda_0.
\end{equation}
\end{lem}

\begin{proof}
Let $u\in C^{\infty}_c(\R^-\times \T,\R^{2n})$ (since $\R^-\times \T$ is the closed cylinder, the support of $u$ may intersect the boundary $\{0\}\times \T$).
Since $J$ preserves the inner product $g_J$,
\[
|\partial_s u|_J^2 + |\partial_t u|_J^2 = |\partial_s u|_J^2 + |J \partial_t u|_J^2 = |\partial_s + J \partial_t u|_J^2 - 2 g_J( \partial_s u, J \partial_t u)= |\delbar u|_J^2 - 2 \omega_0 (\partial_s u, \partial_t u).
\]
Integrating the above identity over $\R^-\times \T$ and applying Stokes theorem, we obtain
\[
\begin{split}
\int_{\R^- \times \T} \bigl( |\partial_s u|_J^2 + |\partial_t u|_J^2 \bigr)\, ds\,dt = \int_{\R^- \times \T}  |\overline{\partial}_J u|_J^2 \, ds \,dt - 2 \int_{\R^- \times \T} d\lambda_0 (\partial_s u, \partial_t u) \, ds\, dt\\ = 
\int_{\R^- \times \T}  |\overline{\partial}_J u|_J^2 \, ds\, dt - 2\int_{\R^- \times \T} u^* d\lambda_0 = 
\int_{\R^- \times \T}  |\overline{\partial}_J u|_J^2 \, ds \,dt - 2\int_{\T} u(0,\cdot)^* \lambda_0,
\end{split}
\]
so the identity (\ref{ide}) holds for compactly supported smooth maps. It is well-known that the quadratic form
\[
x\mapsto \int_{\T} x^* \lambda_0, \quad x\in C^{\infty}(\T,\R^{2n}),
\]
extends continuously to $H^{1/2}(\T,\R^{2n})$ (see e.g.\ \cite[Chapter 6]{rab86} or \cite[Section 3.3]{hz94}). Since the trace operator maps
$H^1(\R^- \times \T,\R^{2n})$ continuously into $H^{1/2}(\T,\R^{2n})$, the identity (\ref{ide}) extends to every $u\in H^1(\R^- \times \T,\R^{2n})$.
\end{proof} 

\paragraph{An index theorem.} Let $J:\R^-\times \T \rightarrow \mathrm{L}(\R^{2n})$ be a Lipschitz map which extends continuously to the compactification $[-\infty,0]\times \T$, such that
\begin{equation}
\label{regJ}
\lim_{s_0\rightarrow -\infty} \esssup_{(s,t)\in ]-\infty,s_0[ \times \T} \Bigl( \bigl|\partial_s J(s,t)\bigr| + \bigl| \partial_t J(s,t) - \partial_t J(-\infty,t) \bigr| \Bigr) = 0,
\end{equation}
and $J(s,t)$ is an $\omega_0$-compatible complex structure on $\R^{2n}$, for every $(s,t)\in [-\infty,0]\times \T$. The fact that both $J$ and $\nabla J$ are in $L^{\infty}$ implies that the linear operator
\[
H^2(\R^-\times \T, \R^{2n}) \rightarrow H^1(\R^-\times \T, \R^{2n}), \quad u \mapsto \delbar u,
\]
is continuous.

If $A: \R^- \times \T\rightarrow \mathrm{L}(\R^{2n})$ is a bounded Lipschitz map, then the multiplication operator
\[
H^1(\R^-\times \T,\R^{2n}) \rightarrow H^1(\R^-\times \T,\R^{2n}), \quad u \mapsto JAu,
\]
is continuous. We also assume that $A$ extends continuously to the compactification $[-\infty,0]\times \T$, that
\begin{equation}
\label{regA}
\lim_{s_0\rightarrow -\infty} \esssup_{(s,t)\in ]-\infty,s_0[ \times \T} \Bigl( \bigl|\partial_s  A(s,t)\bigr| + \bigl| \partial_t A(s,t) - \partial_t A(-\infty,t) \bigr|\Bigr) = 0,
\end{equation}
and that for every $t\in \T$ the linear endomorphism
$A(-\infty,t)$ belongs to $\mathrm{sp}(\R^{2n},\omega_0)$, the Lie algebra of the symplectic group $\mathrm{Sp}(\R^{2n},\omega_0)$. We denote by 
$W:[0,1] \rightarrow \mathrm{Sp}(\R^{2n},\omega_0)$ the fundamental solution of the time-periodic linear Hamiltonian systems defined by $A(-\infty,\cdot)$, that is
the solution of the linear Cauchy problem
\begin{equation}
\label{linham}
W'(t) = A(-\infty,t) W(t), \quad W(0) = I,
\end{equation}
and we assume that the symplectic path $W$ is non-degenerate, meaning that 1 is not an eigenvalue of $W(1)$. Under this assumption, the Conley-Zehnder index of $W$ is a well-defined integer, that we denote by $\mu_{CZ}(W)$ (see \cite{sz92,rs95}). By our sign choices, the Conley-Zehnder index of the symplectic path 
\[
W_a : [0,1]\rightarrow \mathrm{Sp}(\R^2,\omega_0), \quad W_a (t) = e^{-atJ_0},
\]
which is non-degenerate if and only if the real number $a$ does not belong to $2\pi \Z$, is
\begin{equation}
\label{czind}
\mu_{CZ}(W_a) = 2\left\lfloor \frac{a}{2\pi} \right\rfloor + 1,\qquad \forall a\in \R\setminus 2\pi \Z,
\end{equation}
that is the unique odd integer which is closest to $a/\pi \in \R\setminus 2\Z$.

Let 
\[
\pi: \R^{2n} = \R^n \times \R^n \rightarrow \R^n , \qquad (q,p) \mapsto q,
\]
be the projection onto the first factor. Since the trace operator
\[
H^1(\R^-\times \T,\R^n) \rightarrow H^{1/2} (\T,\R^n), \quad v\mapsto v(0,\cdot),
\]
is continuous, so is the linear operator
\[
H^2(\R^-\times \T,\R^{2n}) \rightarrow H^{1/2} (\T,\R^n), \quad u \mapsto \partial_t (\pi\circ u) (0,\cdot).
\]
Finally, if $\alpha:\T\rightarrow \mathrm{L}(\R^{2n},\R^n)$ is a Lipschitz map, the linear operator
\begin{equation}
\label{comp}
H^2(\R^-\times \T,\R^{2n}) \rightarrow H^{1/2}(\T,\R^n), \quad u \mapsto \alpha u(0,\cdot),
\end{equation}
is continuous, and actually compact, as it factorizes through
\[
H^2(\R^-\times \T,\R^{2n}) \stackrel{\mathrm{tr}}{\rightarrow} H^{3/2}(\T,\R^{2n}) \hookrightarrow H^1(\T,\R^{2n}) \stackrel{\alpha}{\rightarrow} H^1(\T,\R^n) \hookrightarrow H^{1/2}(\T,\R^n),
\]
where the two inclusions are compact.

By the above considerations, the linear operator
\[
\begin{split}
T : H^2(\R^-\times \T, \R^{2n}) &\rightarrow H^1(\R^-\times \T, \R^{2n}) \times H^{1/2}(\T, \R^n), \\ u & \mapsto \bigl(\delbar u - JA u, \partial_t (\pi\circ u) (0,\cdot) - \alpha u(0,\cdot)\bigr), \end{split}
\] 
is continuous.

\begin{thm}
\label{fredholm}
Under the above assumptions, the operator $T$ is Fredholm of index $\mu_{CZ}(W)$.
\end{thm}

The proof of this theorem takes the remaining part of this section. 

\paragraph{The semi-Fredholm property.} The proof of the next lemma uses the basic identity of Lemma \ref{baside}.

\begin{lem}
\label{semi}
The operator $T$ has finite-dimensional kernel and closed image.
\end{lem}

\begin{proof}
The thesis is that $T$ is semi-Fredholm with index different from $+\infty$. The semi-Fredholm property and the semi-Fredholm index are stable with respect to compact perturbations. Therefore, the fact that the operator (\ref{comp}) is compact implies that we may assume that $\alpha=0$.

It is well-known that, due to the non-degeneracy of $W$, the translation-invariant operator
\begin{equation}
\label{iso}
H^2(\R \times \T,\R^{2n}) \rightarrow H^1(\R\times \T,\R^{2n}), \quad u \mapsto \overline{\partial}_{J(-\infty,\cdot)} u - J(-\infty,\cdot) A(-\infty,\cdot) u,
\end{equation}
is an isomorphism (see \cite{sz92}). By (\ref{regJ}) and (\ref{regA}), the norm of the operator
\[
\begin{split}
H^2(]-\infty,s_0[\times \T,\R^{2n})&\rightarrow  H^1(]-\infty,s_0[\times \T,\R^{2n}), \\ u &\mapsto \bigl( J(s,t)-J(-\infty,t) \bigr) \partial_t u - \bigl( J(s,t)A(s,t) - J(-\infty,t) A(-\infty,t) \bigr) u, \end{split}
\]
tends to zero for $s_0\rightarrow -\infty$. Then the fact that (\ref{iso}) is an isomorphism and the fact that the set of isomorphisms is open in the operator norm topology implies that, if $s_0<0$ is small enough, there exists a number $c_0$ such that
\begin{equation}
\label{allinf}
\|u\|_{H^2} \leq c_0 \|\delbar u - JAu\|_{H^1}, \quad \forall u\in H^2(]-\infty,s_0[\times \T,\R^{2n}).
\end{equation}
Let $u$ be in $H^2(\R^-\times \T,\R^{2n})$. Using the fact that the inner products $g_{J(s,t)}$ are equivalent to the Euclidean one, uniformly in $(s,t)$, and applying Lemma \ref{baside} to $\partial_t u$ we obtain
\begin{equation}
\label{eq1}
\begin{split}
\| \nabla \partial_t u \|_{L^2}^2 \leq c_1 \int_{\R^- \times \T} \bigl( |\partial_s \partial_t u|^2_J + |\partial_t \partial_t u|_J^2 \bigr)\, ds\, dt \\
= c_1 \int_{\R^- \times \T} | \delbar \partial_t u |_J^2\, ds \, dt - 2 c_1 \int_{\T} \bigl( \partial_t u(0,\cdot) \bigr)^* \lambda_0.
\end{split}
\end{equation}
Here and in the following lines $c_1,c_2,\dots$ denote positive constants.
Using also the fact that $J$ and $A$ are globally Lipschitz and bounded, the first integral on the right-hand side of (\ref{eq1}) can be bounded by
\begin{equation}
\label{eq2}
\begin{split}
\int_{\R^- \times \T} | \delbar \partial_t u |_J^2\, ds \, dt = \int_{\R^- \times \T} | \partial_t \delbar u - (\partial_t J)( \partial_t u)|_J^2\, ds \, dt \\ \leq 2 \int_{\R^- \times \T} | \partial_t (\delbar u - JA u) |_J^2\, ds \, dt + 2 \int_{\R^- \times \T} | \partial_t (JAu) - (\partial_t J)( \partial_t u) |_J^2\, ds \, dt  \\ \leq c_2 \| \delbar u -JAu\|_{H^1}^2 + c_3 \|u\|_{H^1}^2.
\end{split}
\end{equation}
On the other hand, setting $u(t,0)=(q(t),p(t))$, the second integral on the right-hand side of (\ref{eq1}), can be estimated by
\begin{equation}
\label{eq3}
\int_{\T} \bigl(\partial_t u(0,\cdot) \bigr)^* \lambda_0 = \int_{\T} p'(t) \cdot q''(t)\, dt \leq \|p'\|_{H^{1/2}(\T)} \|q''\|_{H^{-1/2}(\T)} \leq c_4 \|u\|_{H^2(\R^- \times \T)} \|q'\|_{H^{1/2}(\T)},
\end{equation}
From (\ref{eq1}), (\ref{eq2}) and (\ref{eq3}) we obtain the estimate
\begin{equation}
\label{eq3.5}
\| \nabla \partial_t u \|_{L^2} \leq c_5 \| \delbar u - JA u\|_{H^1} + c_6 \|u\|_{H^2}^{1/2} \|q'\|_{H^{1/2}}^{1/2} + c_7 \|u\|_{H^1}.
\end{equation}
From the identity
\[
\partial_{ss} u = \partial_s (\delbar u - JA u) - J \partial_{st} u - (\partial_s J)( \partial_t u) +  \partial_s (JAu),
\]
and from (\ref{eq3.5}) we obtain the estimate
\begin{equation}
\label{eq4}
\begin{split}
\|\partial_{ss} u\|_{L^2} &\leq \|\delbar u - JA u\|_{H^1} + c_8 \|\partial_{st} u\|_{L^2} + c_9 \|u\|_{H^1} \\
&\leq c_{10} \|\delbar u - JA u\|_{H^1} + c_{11} \|u\|_{H^2}^{1/2} \|q'\|_{H^{1/2}}^{1/2} + c_{12} \|u\|_{H^1}.
\end{split}
\end{equation}
The estimates (\ref{eq3.5}) and (\ref{eq4}) imply the bound
\begin{equation}
\label{eq5}
\|u\|_{H^2} \leq c_{13} \| \delbar u - JA u\|_{H^1} + c_{14} \|q'\|_{H^{1/2}}
+ c_{15} \|u\|_{H^1}, \qquad \forall u\in H^2(\R^-\times \T,\R^{2n}).
\end{equation}
Let $\chi\in C^{\infty}(\R^-)$ be a cut-off function such that $\chi=1$ on $]-\infty,s_0-1]$ and $\chi=0$ on $[s_0,0]$. Since
\[
Tu = (\delbar u -JA u,q'),
\]
writing $u=\chi u + (1-\chi) u$ and applying (\ref{allinf}) to $\chi u$ and (\ref{eq5}) to $(1-\chi)u$, we obtain the estimate
\begin{equation}
\label{final}
\|u\|_{H^2} \leq c_{16} \bigl( \|T u\|_{H^1(\R^-\times \T) \times H^{1/2}(\T)} + \|u\|_{H^1(]s_0-1,0[\times \T)} \bigr), \qquad \forall u\in H^2(\R^-\times \T,\R^{2n}).
\end{equation}
Since the operator
\[
H^2(\R^-\times \T) \rightarrow H^1(]s_0-1,0[\times \T), \quad u \mapsto u|_{]s_0-1,0[\times \T},
\]
is compact, the estimate (\ref{final}) implies that $T$ has finite dimensional kernel and closed range (see e.g.\ \cite{}).
\end{proof}

\paragraph{Some explicit index computations.}
In the next two lemmas we compute explicitly the kernel and cokernel of $T$ for particular choices of $J$, $A$ and $\alpha$. In both cases we assume that $n=1$, that $J = J_0$ is the standard complex structure on $\R^2$ and that $\alpha= 0$. 

\begin{lem}
\label{compu1}
Under the above assumptions, let $A(s,t) \equiv -aJ_0$ with $a\in \R\setminus 2\pi \Z$. Then
\[
\dim \ker T = \left\{ \begin{array}{ll} 2 \lceil \frac{a}{2\pi} \rceil & \mbox{if } a>0 \\
0 & \mbox{if } a<0 \end{array} \right. , \quad
\dim \coker T = \left\{ \begin{array}{ll}  1 & \mbox{if } a>0 \\
2 \lfloor - \frac{a}{2\pi} \rfloor + 1 & \mbox{if } a<0 \end{array} \right. .
\]
In particular,
\[
\ind T = 2 \left\lfloor \frac{a}{2\pi} \right\rfloor + 1 = \mu_{CZ}(W),
\]
for every $a\in \R\setminus 2\pi \Z$.
\end{lem}

\begin{proof}
Here it is convenient to identify $\R^2$ with $\C$ by the standard complex structure $J_0$, which is then denoted by $i$. The solution $W(t)=e^{-ait}$ of (\ref{linham})
is non-degenerate precisely because $a\notin 2\pi \Z$, and its Conley-Zehnder index is given by formula (\ref{czind}). Let $u\in H^2(\R^-\times \T,\C)$ be an element of the kernel of $T$. Since $\partial_t u$ solves a linear Cauchy-Riemann equation on $\R^-\times \T$ with real boundary conditions, $u$ is smooth up to the boundary and it solves the problem
\begin{eqnarray}
\label{q1}
& \partial_s u + i \partial_t u - a u =0, \qquad &\forall (s,t) \in \R^- \times \T,\\
\label{q2}
& \re \partial_t  u(0,t) = 0, \qquad &\forall t\in \T. 
\end{eqnarray}
By taking a Fourier expansion in the $t$ variable, one sees that the general solution of (\ref{q1}) is 
\[
u(s,t) = \sum_{k\in\Z} e^{s(2\pi k + a)} e^{2\pi i k t} u_k, 
\]
where the $u_k$'s are complex numbers. Such a function is in $H^2(\R^- \times \T)$ if and only if
\begin{equation}
\label{p1} 
u_k = 0, \qquad \forall k < - \frac{a}{2\pi}.
\end{equation}
So (\ref{q2}) takes the form
\begin{equation}
\label{q3}
\sum_{k> -\frac{a}{2\pi}} 2\pi k \bigl( \re u_k \sin (2\pi k t) + \im u_k \cos  (2\pi k t) \bigr) = 0, \quad \forall t\in \T.
\end{equation}
When $a<0$, the above sum runs over positive numbers $k$, so the above identity implies that $u_k=0$ for every $k>-a/(2\pi)$. Together with (\ref{p1}), we deduce that
\[
\ker T = (0), \qquad \mbox{if } a<0.
\]
When $a>0$, we can rewrite (\ref{q3}) as
\[
\begin{array}{r}
\displaystyle{\sum_{1\leq k < \frac{a}{2\pi}} 2\pi k \bigl( \re (u_k+u_{-k})  \sin (2\pi k t) + \im (u_k-u_{-k}) \cos  (2\pi k t) \bigr)} \\ \displaystyle{+  \sum_{k> \frac{a}{2\pi}} 2\pi k \bigl( \re u_k \sin (2\pi k t) + \im u_k \cos  (2\pi k t) \bigr) = 0,} \end{array} \qquad  \forall t\in \T,
\]
which holds if and only if
\begin{equation}
\label{p2}
\begin{split}
\re (u_k+u_{-k}) = \im (u_k-u_{-k}) = 0, \qquad & \mbox{if } 1\leq k < \frac{a}{2\pi}, \\ u_k = 0, \qquad & \mbox{if } k> \frac{a}{2\pi}.
\end{split}
\end{equation}
The parameter $u_0\in \C$ is instead free, so (\ref{p1}) and (\ref{p2}) imply that
\[
\dim \ker T = 2 \left\lfloor \frac{a}{2\pi} \right\rfloor + 2 = 2 \left\lceil \frac{a}{2\pi} \right\rceil, \qquad \mbox{if } a>0.
\]
This concludes the computation of the dimension of the kernel of $T$. 

We know from Lemma \ref{semi} that $T$ has closed image. Since the $L^2$ product is a continuous non-degenerate bilinear symmetric form on $H^1$ and on $H^{1/2}$, the dimension of the cokernel of $T$ equals the dimension of the space of pairs
\[
(v,g) \in H^1(\R^-\times \T,\C) \times H^{1/2}(\T,\R)
\]
which are $L^2$-orthogonal to the image of $T$, i.e.\ they satisfy
\begin{equation}
\label{adj}
\re \int_{\R^- \times \T} (\partial_s u + i\partial_t u - a u) \overline{v}\, ds\, dt + \int_{\T} \re \partial_t u(0,t) g(t)\, dt = 0,
\end{equation}
for every $u\in H^2(\R^-\times \T,\C)$. Let $(v,g)$ be such a pair. By letting $u$ vary among all smooth compactly supported functions on $\R^-\times \T$ such that $\re u(0,t)=0$ for every $t\in \T$, the standard regularity results for weak solutions of the Cauchy-Riemann operator (see e.g.\ \cite[Appendix B]{ms04}) imply that $v$ is smooth up to the boundary and satisfy
\begin{equation}
\label{r1}
\partial_s v - i \partial_t v + a v = 0, \qquad \forall (s,t)\in 
\R^- \times \T.
\end{equation}
Moreover, (\ref{adj}) becomes
\begin{equation}
\label{adj2}
\re \int_{\T} u(0,t) \overline{v(0,t)} \, dt + \int_{\T} \re\partial_t u(0,t) g(t) \, dt =0.
\end{equation}
By taking $u(0,\cdot)$ with vanishing real part, (\ref{adj2}) implies that
\begin{equation}
\label{r2}
\im v(0,t) = 0 , \qquad \forall t\in \T.
\end{equation}
On the other hand, taking $u(0,\cdot)$ with vanishing imaginary part, we get that $g$ is smooth and satisfies
\begin{equation}
\label{r3}
g'(t) = \re v(0,t) , \qquad \forall t\in \T.
\end{equation}
Therefore, the dimension of the cokernel of $T$ equals the dimension of the space of smooth solutions $(v,g)$ of (\ref{r1}), (\ref{r2}) and (\ref{r3}) such that $v$ belongs also to $H^1(\R^-\times \T)$. The general solution $v$  of (\ref{r1}) has the form
\[
v(s,t) = \sum_{k\in \Z} e^{-(2\pi k + a)s} e^{2\pi i k t} v_k,
\]
and it belongs to $H^1(\R^-\times \T)$ if and only if
\begin{equation}
\label{c1}
v_k = 0 , \qquad \forall k > - \frac{a}{2\pi}.
\end{equation}
Therefore,
\[
\im v(0,t) = \sum_{k< - \frac{a}{2\pi}} \bigl( \im v_k \cos (2\pi k t) + \re v_k \sin (2\pi k t) \bigr), \qquad \forall t\in \T,
\]
and if $a>0$ the condition (\ref{r2}) implies that $v$ is identically zero. In this case, (\ref{r3}) says that $g$ is constant, so
\[
\dim \coker T = 1 , \qquad \mbox{if } a>0.
\]
If $a<0$, (\ref{r2}) can be rewritten as
\[
\begin{array}{r} \displaystyle{
\im v_0 + \sum_{1\leq k < - \frac{a}{2\pi}} \bigl( \im (v_k+v_{-k})  \cos (2\pi k t) + \re (v_k-v_{-k})  \sin (2\pi k t) \bigr)}\\ \displaystyle{+ \sum_{k> - \frac{a}{2\pi}} \bigl( \im v_{-k} \cos (2\pi k t) - \re v_{-k} \sin (2\pi k t) \bigr) = 0,} \end{array} \qquad \forall t\in \T,
\]
which is equivalent to
\begin{equation}
\label{c2}
\begin{split}
\im v_0 = 0, \qquad & \\
\im (v_k+v_{-k}) = \re (v_k-v_{-k}) = 0, \qquad & \mbox{if } 1\leq k < - \frac{a}{2\pi}, \\ v_k = 0, \qquad & \mbox{if } k< \frac{a}{2\pi}.
\end{split}
\end{equation}
Given $v$, the equation (\ref{r3}) can be solved if and only if
\begin{equation}
\label{r4}
\int_{\T} \re v(0,t) \, dt = 0,
\end{equation}
or equivalently
\begin{equation}
\label{c3}
\re v_0 = 0,
\end{equation}
and in this case it has a one-dimensional space of solutions $g$.  
Together with (\ref{c1}), (\ref{c2}) and (\ref{c3})  imply that the space of smooth solutions $v\in H^1(\R^-\times \T)$ of (\ref{r1}), (\ref{r2}) and (\ref{r4}) has dimension $2 \lfloor -a/2\pi \rfloor$. We deduce that 
\[
\dim \coker T =  2 \left \lfloor - \frac{a}{2\pi} \right\rfloor + 1,
\]
thus concluding the computation of the dimension of the cokernel of $T$.

Finally, if $a>0$ then
\[
\ind T = \dim \ker T - \dim \coker T = 2 \left\lceil \frac{a}{2\pi} \right\rceil - 1 =  2 \left\lfloor \frac{a}{2\pi} \right\rfloor + 1,
\]
and if $a<0$ then
\[
\ind T = \dim \ker T - \dim \coker T = 0 -   \Bigl( 2 \left\lfloor - \frac{a}{2\pi} \right\rfloor + 1 \Bigr) = 2  \left\lceil \frac{a}{2\pi} \right\rceil - 1 =  2 \left\lfloor \frac{a}{2\pi} \right\rfloor + 1,
\]
as claimed.
\end{proof}

\begin{lem}
\label{compu2}
Assume that $n=1$, $J= J_0$, $\alpha= 0$, and let 
\[
A(s,t) \equiv \left( \begin{array}{cc} 1 & 0 \\ 0 & -1 \end{array} \right), \qquad \forall (s,t)\in \R^-\times \T.
\]
Then 
\[
\dim \ker T = \dim \coker T =1.
\]
 In particular,
\[
\ind T = 0 = \mu_{CZ}(W).
\]
\end{lem}

\begin{proof}
In this case, the solution of (\ref{linham}) is
\[
W(t) =  \left( \begin{array}{cc}  e^t & 0 \\ 0 & e^{-t} \end{array} \right),
\]
a non-degenerate symplectic path with zero Conley-Zehnder index. Let $u\in H^2 (\R^-\times \T,\R^2)$ be an element of the kernel of $T$. If we identify $\R^2$ to $\C$ by the standard complex structure $J_0=i$, $A$ is the conjugacy operator, hence $u$ satisfies
\begin{eqnarray*}
\partial_s u + i \partial_t u - i \overline{u} = 0, \qquad & \forall (s,t)\in \R^-\times \T, \\
\re \partial_t u(0,t) = 0, \qquad & \forall t\in \T.
\end{eqnarray*}
If we set
\[
w(s,t) := i \partial_t u(s,t) \quad \forall s\leq 0, \qquad w(s,t) := \overline{w(-s,t)} \quad \forall s>0,
\]
the fact that $w(0,t)$ is real implies that $w$ belongs to $H^1(\R\times \T,\C)$. Moreover, a direct computation shows that
\[
\partial_s w + i \partial_t w + i \overline{w} = 0, \qquad \forall (s,t)\in \R\times \T .
\]
Since the translation invariant operator
\[
H^1(\R\times \T,\C) \rightarrow L^2(\R\times \T,\C), \quad w \mapsto \partial_s w + i \partial_t w + i \overline{w},
\]
is an isomorphism, we deduce that $w$ is identically zero. Therefore, $u(s,t)=y(s)$, where $y\in H^2(\R^-,\C)$ solves the ordinary differential equation
\[
y'(s)=i\overline{y(s)}, \qquad \forall s\in \R^-.
\]
Such an equation has a one-dimensional space of solutions in $H^2(\R,\C)$, namely
\[
y(s) = \lambda (1+i) e^s, \qquad \mbox{for } \lambda \in \R,
\]
hence
\[
\dim \ker T = 1.
\]

Consider a pair
\[
(v,g) \in H^1(\R^-\times \T,\C) \times H^{1/2}(\T,\R),
\]
which is $L^2$-orthogonal to the image of $T$, that is satisfies
\begin{equation}
\label{adj3}
\re \int_{\R^- \times \T} (\partial_s u + i\partial_t u - i \overline{u}) \overline{v}\, ds\, dt + \int_{\T} \re \partial_t u(0,t) g(t)\, dt = 0,
\end{equation}
for every $u\in H^2(\R^-\times \T,\C)$. Arguing as in the proof of Lemma \ref{compu1}, we see that (\ref{adj3}) is equivalent to the fact that $(v,g)$ is a smooth solution of
\begin{eqnarray*}
\partial_s v - i \partial_t v + i \overline{v} = 0, \qquad & \forall (s,t)\in \R^-\times \T, \\
\im v(0,t) = 0, \qquad & \forall t\in \T, \\
g'(t) = \re v(0,t), \qquad & \forall t\in \T,
\end{eqnarray*}
with $v\in H^2(\R^-\times \T)$.
If we extend $v$ to the whole cylinder by Schwarz reflection,
\[
v(s,t) := \overline{v(-s,t)}, \qquad \forall (s,t)\in \R^+\times \T,
\]
we see that $v$ belongs to the kernel of the translation invariant operator
\[
H^2(\R\times \T,\C) \rightarrow H^1(\R\times \T,\C), \quad v \mapsto \partial_s v - i \partial_t v + i \overline{v}.
\]
Since this operator is an isomorphism, $v$ is identically zero. We conclude that $(v,g)$ satisfies (\ref{adj3}) if and only if $v=0$ and $g$ is constant, hence
\[
\dim \coker T = 1,
\]
as claimed.
\end{proof}

\paragraph{Proof of Theorem \ref{fredholm}.} By Lemma \ref{semi}, the operator $T$ is semi-Fredholm and we just have to prove that its index equals the Conley-Zehnder index of $W$. 
Since the semi-Fredholm index is locally constant, we can reduce the computation of the index of $T$ to particular classes of $J,A,\alpha$, by considering homotopies of these data which keep the semi-Fredholm property. Since the space of $\omega_0$-compatible complex structures on $\R^{2n}$ and the space $\mathrm{L}(\R^{2n},\R^n)$ are contractible, we can assume that $J=J_0$ and $\alpha=0$. The admissible homotopies of $A$ are those which keep the solution of (\ref{linham})
non-degenerate. We recall that, as already proven by I.\ M.\ Gel'fand and V.\ B.\ Lidski\v{\i} in \cite{gl58}, the Conley-Zehnder index labels the ``regions of stability of canonical systems of differential equations with periodic coefficients'': more precisely, two continuous loops
\[
A_0,A_1 : \T \rightarrow \mathrm{sp}(\R^{2n},\omega_0)
\]
with non-degenerate fundamental solutions
\[
W_j : [0,1] \rightarrow \mathrm{Sp}(\R^{2n},\omega_0), \quad W_j'(t) = A(t) W_j(t), \quad W_j(0)=I, \quad j=0,1,
\]
are homotopic via a path $[0,1]\rightarrow \mathrm{sp}(\R^{2n},\omega_0)$, $s\mapsto A_r(\cdot)$, with non-degenerate fundamental solutions for every $r\in [0,1]$ if and only if $\mu_{CZ}(W_0)=\mu_{CZ}(W_1)$. Let us consider the class $\mathscr{A}$ of endomorphisms $A_0\in \mathrm{sp}(\R^{2n},\omega_0)$ which preserve the standard symplectic splitting 
\[
\R^{2n} = \R^2 \oplus \dots \oplus \R^2,
\]
and which coincide with one of the endomorphisms considered in Lemmas \ref{compu1} and \ref{compu2} on each of these symplectic two-dimensional spaces. The Conley-Zehnder index of the fundamental solution of the linear Hamiltonian system on $\R^2$ given by an element of $\mathrm{sp}(\R^{2n},\omega_0)$ as in Lemma \ref{compu1} can be any odd number, while in the case of Lemma \ref{compu2} we obtain Conley-Zehnder index zero. Therefore, in the case $n\geq 2$
we can find an element $A_0\in \mathscr{A}$ such that the Conley-Zehnder index of the corresponding fundamental solution equals $\mu_{CZ}(W)$. By the above mentioned result of Gel'fand and Lidski\v{\i},
 we can find a homotopy $A_r$ from the constant map $A_0$ to the map $A$, which keeps $A_r(-\infty,t)$ in $\mathrm{sp}(\R^{2n},\omega_0)$ and such that the  solution $W_r$ of 
\[
W_r'(t) = A_r(-\infty,t) W_r(t), \quad W_r(0)=I,
\]
is non-degenerate. The corresponding path of operators $T_r$ is semi-Fredholm for every $r\in [0,1]$, and $T_0$ is the direct sum of the operators considered in Lemmas \ref{compu1} and \ref{compu2}. By these lemmas, together with the additivity of the Fredholm index and of the Conley-Zehnder index with respect to direct sums, we obtain that
\[
\ind T = \ind T_0 = \mu_{CZ}(W_0) = \mu_{CZ}(W),
\]
as claimed. When $n=1$, the argument is not complete because the Conley-Zehnder index of the fundamental solutions associated to elements in $\mathscr{A}$ is either zero or odd. In this case, we can consider the operator $T\oplus T$, which has index $2\,\ind T$ and which is of the form considered above with $n=2$. By what we have already proven and by
the additivity of the Conley-Zehnder index, $T\oplus T$ has index $2\,\mu_{CZ}(W)$, hence $T$ has index $\mu_{CZ}(W)$. This concludes the proof of Theorem \ref{fredholm}.

\hfill $\Box$

\section{The reduced unstable manifold}
\label{secrum}

Let $M$ be a smooth compact $n$-dimensional manifold without boundary. We shall also assume that $M$ is orientable: this assumptions is by no means necessary, but it simplifies the exposition (see Remark \ref{nonorient} below for information on how to deal with the non-orientable case). Let $T^*M$ be the cotangent bundle of $M$, whose elements are denoted  as $(q,p)$, with $q\in M$ and $p\in T^*_qM$, and let
\[
\pi: T^*M \rightarrow M, \quad (q,p) \mapsto q,
\]
be the bundle projection.
The cotangent bundle $T^*M$ is endowed with the standard Liouville form $\lambda=p\, dq$ and the standard symplectic structure $\omega=d\lambda = dp\wedge dq$. 
Let $H\in C^{\infty}(\T\times T^*M)$ be a Hamiltonian and let $X_H$ be the corresponding time-periodic Hamiltonian vector field on $T^*M$, which is defined by
\[
\omega (X_H (t,x), \xi) = - d_xH (t,x)[\xi], \qquad \forall x\in T^*M, \; \forall \xi \in T_xT^*M.
\]
Let $x \in C^{\infty}(\T,T^*M)$ be a 1-periodic orbit of $X_H$, which we assume to be {\em non-degenerate}, meaning that 1 is not an eigenvalue of the differential of the time-one map of $X_H$ at $x(0)$. 

The {\em Conley-Zehnder index $\mu_{CZ}(x)$ of $x$} is defined as the Conley-Zehnder index of the symplectic path $W:[0,1] \rightarrow \mathrm{Sp}(T^*\R^n,\omega_0)$ which is obtained by reading the differential of the Hamiltonian flow along $x$ on $T^*\R^n$ by means of a vertical-preserving trivialization of $x^*(TT^*M)$. More precisely, let
\[
\Theta(t) : (T^* \R^n,\omega_0) \rightarrow (T_{x(t)} T^*M, \omega_{x(t)}), \quad t\in \T,
\]
be a smooth family of symplectic isomorphisms such that 
\begin{equation}
\label{vert}
\Theta(t) \bigl( (0) \times (\R^n)^* \bigr) = T_{x(t)}^v T^*M := \ker D\pi(x(t)), \qquad \forall t\in \T,
\end{equation}
and let
\[
W(t) := \Theta(t)^{-1} D_x \phi_{X_H} (t,x(0)) \Theta(0) : T^* \R^n \rightarrow T^* \R^n,
\] 
where $\phi_{X_H}$ denotes the flow of the Hamiltonian vector field $X_H$. The condition (\ref{vert}) is the vertical-preserving requirement and a trivialization satisfying (\ref{vert}) exists since the vector bundle $(\pi\circ x)^*(TM)$ is trivial, because $M$ is orientable. Then $\mu_{CZ}(x):= \mu_{CZ}(W)$ does not depend on the choice of the vertical-preserving trivialization $\Theta$ (see \cite[Lemma 1.3]{as06}). 

We fix also a $\omega$-compatible smooth almost complex structure $J$ on $T^*M$, which we allow to depend smoothly on $t\in \T$. 
\begin{defn} The {\em reduced unstable manifold} of the non-degenerate 1-periodic orbit $x$ is the set $\mathscr{U}(x)$ consisting of all maps
\[
u\in C^{\infty} (\R^-\times \T,T^*M)
\]
such that:
\begin{enumerate}
\item $u$ solves the equation
\begin{equation}
\label{cr}
\partial_s u  + J(t,u) \bigl( \partial_t u - X_H(t,u) \bigr) = 0, 
\end{equation}
on $\R^- \times \T$;
\item the energy of $u$, that is the quantity
\[
E(u) := \int_{\R^-\times \T} |\partial_s u|_J^2\, ds\, dt,
\]
is finite;
\item $u(s,t)$ converges to $x(t)$ for $s\rightarrow -\infty$, uniformly in $t\in \T$.
\item the restriction of $u$ to the boundary of the half-cylinder satisfies the equation 
\[
\partial_t (\pi\circ u) (0,t) = d_p H\bigl(t,u(0,t)\bigr),
\]
where $d_p H(t,x)$ denotes the fiber-wise differential of $H(t,\cdot)$ at $x$, which is an element of $(T_{\pi(x)}^*M)^* = T_{\pi(x)} M$.
\end{enumerate}
\end{defn}

The name {\em unstable manifold} is motivated by the fact that (\ref{cr}) is the $L^2$-negative gradient flow equation of the Hamiltonian action functional
\begin{equation}
\label{action}
\mathbb{A}: C^{\infty}(\T,T^*M) \rightarrow \R, \qquad \mathbb{A}(x) = \int_{\T} x^*(\lambda) - \int_{T} H(t,x(t))\, dt,
\end{equation}
of which the periodic orbit $x$ is a critical point. The adjective {\em reduced} refers to the fact that we are adding to the conditions (i), (ii), (iii), which would define an infinite dimensional object, the condition (iv), which says that the loop $u(0,\cdot):\T \rightarrow T^*M$ solves ``half of the Hamiltonian equations'' associated to $H$. The set $\mathscr{U}(x)$ always contains the stationary solution
\[
u_0(s,t) := x(t), \quad \forall (s,t)\in \R^- \times \T.
\]

Consider a smooth map $\varphi:\T \times \R^{2n} \rightarrow T^*M$ such that
\[
\varphi(t,0) = x(t), \qquad \forall t\in \T,
\]
and $z \mapsto \varphi(t,z)$ is a smooth embedding of $\R^{2n}$ onto an open neighborhood of $x(t)$, for every $t\in \T$. If $u$ is an element of $\mathscr{U}(x)$, condition (iii) implies that there exists $s_0\leq 0$ such that
\begin{equation}
\label{vicinoxbar}
u(s,t) = \varphi(t,\hat{u}(s,t)), \qquad \forall (s,t)\in ]-\infty,s_0[ \times \T,
\end{equation}
for a unique $\hat{u}\in C^{\infty}( ]-\infty,s_0[ \times \T, \R^{2n} )$. The fact that $x$ is non-degenerate implies that the convergence of $u(s,t)$ to $x(t)$ for $s\rightarrow -\infty$ is exponential, together with the derivatives of every order: more precisely, there exists $\lambda>0$ and a sequence $(C_k)_{k\in \N}$ of positive numbers such that
\begin{equation}
\label{expo}
\bigl| \partial_s^h \partial_t^k \hat{u}(s,t) \bigr| \leq C_{h+k} e^{\lambda (s-s_0)}, \qquad \forall  (s,t)\in ]-\infty,s_0[ \times \T.
\end{equation}
See e.g.\ \cite[Proposition 1.21]{sal99}. 

\paragraph{The functional setting.} Let $\mathscr{X}$ be the set of all maps $u:\R^-\times \T\rightarrow T^*M$ which are in $H^2(]s_1,0[\times \T,T^*M)$ for every $s_1<0$ and for which there exists $s_0\leq 0$ such that (\ref{vicinoxbar}) holds for a (unique) $\hat{u}\in H^2( ( ]-\infty,s_0[ \times \T, \R^{2n} )$. The set $\mathscr{X}$ does not depend on the choice of the map $\varphi$.
Since maps on a two-dimensional domain of Sobolev class $H^2$ are continuous and since left-composition by smooth maps keeps the $H^2$ regularity, $\mathscr{X}$ carries the structure of a smooth Hilbert manifold modeled on $H^2(\R^-\times \T,\R^{2n})$. Indeed, any smooth map
\[
\psi: \R^- \times \T \times \R^{2n} \rightarrow T^*M
\]
such that $\psi(s,t,0)=x(t)$ for every $s$ small enough and every $t\in \T$, and such that $z \mapsto \psi(s,t,z)$ is an open embedding of $\R^{2n}$ into $T^*M$ for every $(s,t)\in \R^- \times \T$, defines the inverse of a chart by
\[
H^2(\R^-\times \T,\R^{2n}) \rightarrow \mathscr{X}, \quad \hat{u} \mapsto \psi(\cdot,\cdot,\hat{u}).
\]

\begin{rem}
\label{redatlas}
It is often convenient to choose the maps $\psi$ to be compatible with the cotangent bundle structure, that is to be of the form
\begin{equation}
\label{redphi}
\psi: \R^- \times \T \times T^* \R^n \rightarrow T^*M, \quad
\psi(s,t,(q,p)) = \bigl(f(s,t,q),(D_{q} f(s,t,q)^*)^{-1}[p]\bigr).
\end{equation}
Here
$f:\R^-\times \T\times \R^n$ is a smooth map such that 
\[
f(s,t,0)=\pi\circ x(t) \quad \mbox{and} \quad D_q f(s,t,0)^* x(t) = \hat{p}(t), \quad \mbox{for some } \hat{p}\in C^{\infty}\bigl(\T,(\R^n)^*\bigr), 
\]
for every $s$ small enough and such that $q\mapsto f(s,t,q)$ is an open embedding of $\R^n$ into $M$ for every $(s,t)\in \R^- \times \T$. In particular, $z\mapsto \psi(s,t,z)$ is a symplectic embedding for every $(s,t)\in \R^-\times \T$.
\end{rem}

Let $\mathscr{E}$ be the smooth Hilbert bundle over $\mathscr{X}$ whose fiber at $u\in \mathscr{X}$ is the space $\mathscr{E}_u$ of sections $\zeta$ of the vector bundle
\[
u^*(TT^*M) \rightarrow \R^- \times \T,
\]
which are in $H^1(]s_1,0[\times \T)$ for every $s_1<0$ and such that, if $\hat{u}\in H^2(]-\infty,s_0[\times \T,\R^{2n})$ satisfies (\ref{vicinoxbar}), then
\[
\zeta(s,t) = D_x \varphi(t, \hat{u}(s,t))[\hat\zeta(s,t)],
\]
for some $\hat\zeta \in H^1((]-\infty,s_0[\times \T,\R^{2n})$.

Finally, let $\mathscr{F}$ be the smooth Hilbert bundle over $\mathscr{X}$ whose fiber at $u\in \mathscr{X}$ is the space $\mathscr{F}_u$ of $H^{1/2}$-sections of the vector bundle
\[
(\pi\circ u(0,\cdot))^*(TM)\rightarrow \T.
\]
Let $\mathscr{T}$ be the section of the Hilbert bundle $\mathscr{E}\oplus \mathscr{F} \rightarrow \mathscr{X}$ which is defined by
\[
\mathscr{T} (u) := \bigl( \partial_s u + J(t,u) (\partial_t u - X_H(t,u)), \partial_t (\pi\circ u) (0,t) - d_p H(t,u(0,t)) \bigr).
\]
A standard argument shows that the section $\mathscr{T}$ is smooth. By (\ref{expo}), every element $u$ of $\mathscr{U}(x)$ belongs to $\mathscr{X}$, hence it is a zero of the section $\mathscr{T}$. The converse is also true, as we show in the following regularity result:

\begin{prop}
\label{regprop}
Let $u\in \mathscr{X}$ be a zero of the section $\mathscr{T}$. Then $u$ belongs to $\mathscr{U}(x)$.
\end{prop}

\begin{proof}
All the elements of $\mathscr{X}$ have finite energy and converge uniformly to $x$ for $s\rightarrow -\infty$, so $u$ satisfies (ii) and (iii). A standard elliptic inner regularity argument implies that $u$ is smooth in $]-\infty,0[ \times \T$ and that it is a classical solution of (\ref{cr}), so also (i) holds. There remains to show that $u$ is smooth up to the boundary, so that (iv) holds in the classical sense.

Consider a map $\psi$ as in Remark \ref{redatlas} such that $u$ belongs to the domain of the corresponding chart on $\mathscr{X}$, and define the map $\hat{u}:\R^- \times \T \rightarrow T^* \R^n$ by
\[
u(s,t) = \psi\bigl(s,t,\hat{u}(s,t) \bigr).
\]
Then the map $\hat{u}$ belongs to 
\[
\Bigl( (0,\hat{p}) + H^2(\R^-\times \T,T^* \R^n) \Bigr) \cap C^{\infty}(]-\infty,0[\times \T,T^* \R^n)
\]
and solves the equations
\begin{eqnarray}
\label{crloc}
\partial_s \hat{u}  + \hat{J}(s,t,\hat{u}) \bigl( \partial_t \hat{u} - X_{\hat{H}}(s,t,\hat{u}) \bigr) = 0,  \\
\label{bdryloc}
\partial_t (\pi\circ \hat{u}) (0,t) = d_p \hat{H}\bigl(0,t,\hat{u}(0,t)\bigr). 
\end{eqnarray}
Here 
\[
\hat{J}\in C^{\infty}( \R^- \times \T \times \R^{2n} , \mathrm{L}(T^*\R^n)) \quad \mbox{and} \quad 
\hat{H}\in C^{\infty}(\R^- \times \T \times T^*\R^n) 
\]
are the $\omega_0$-compatible almost complex structure and the Hamiltonian on $T^*\R^n$ which are obtained as pull-backs of $J$ and $H$ by $\psi$. We must prove that $\hat{u}$ is smooth up to the boundary. 
Being the trace of $\hat{u}$, the loop $t\rightarrow \hat{u}(0,t)$ is in $H^{3/2}(\T,T^*\R^n)$, so (\ref{bdryloc}) implies that $\pi\circ \hat{u}(0,\cdot)$ is in $H^{5/2}(\T,\R^n)$. Therefore, we can find a map $v\in H^3(]-1,0[\times \T,T^* \R^n)$ such that 
\[
v(0,t) = (\pi\circ u(0,t),0), \qquad \forall t\in \T.
\]
The map $w:= \hat{u}-v$ solves the problem
\begin{eqnarray}
\label{crw}
\partial_s  w + \hat{J}(s,t,\hat{u}) \partial_t w &=&  - \partial_s v - \hat{J}(s,t,\hat{u}) (\partial_t v - X_{\hat{H}}(s,t,\hat{u})), \\
\label{bdryw}
w(0,t) &\in & (0)\times (\R^n)^*.
\end{eqnarray}
The right-hand side of (\ref{crw}) belongs to $H^2(]-1,0[\times \T,T^* \R^n)$. Moreover, the linear subspace $(0)\times (\R^n)^*$ is Lagrangian for $\omega_0$, hence totally real for all the $\omega_0$-compatible complex structures $\hat{J}(s,t,\hat{u}(s,t))$. Therefore, $w$ solves a linear inhomogeneous first order Cauchy-Riemann equation with $H^2$ coefficients, $H^2$ right-hand side and totally real boundary conditions. This implies that $w$ belongs to $H^3(\R^-\times \T,T^* \R^n)$,
and so does $\hat{u}$.
Bootstrapping this argument, we obtain that $\hat{u}$ is smooth up to the boundary.
\end{proof}

\begin{prop}
If $u\in \mathscr{X}$ is a zero of the section $\mathscr{T}$, then the fiber-wise differential of $\mathscr{T}$ at $u$ is a Fredholm operator of index $\mu_{CZ}(x)$.
\end{prop}

\begin{proof}
By Proposition \ref{regprop}, $u$ is an element of the reduced unstable manifold $\mathscr{U}(x)$. We can use the inverse $\psi$ of a chart which belongs to the atlas described in Remark \ref{redatlas} to trivialize the section $\mathscr{T}$ locally near $u$. In such a way, we identify $u$ with a zero $\hat{u}\in H^2(\R^-\times \T,\R^{2n})$ of a a map
\[
\widehat{\mathscr{T}} : H^2(\R^- \times \T,\R^{2n}) \rightarrow H^1(\R^- \times \T,\R^{2n}) \times H^{1/2}(\T,\R^n),
\]
of the form
\[
\widehat{\mathscr{T}} (v) = \bigl( \partial_s v + \hat{J}(\cdot,\cdot,v) (\partial_t v - X_{\hat{H}} (\cdot,\cdot,v)), \partial_t \pi v (0,\cdot) - \partial_p \hat{H}(0,\cdot,v(0,\cdot)) \bigr),
\]
where $\pi:\R^{2n}\rightarrow \R^n$ denotes the projection onto the first factor.
Here
\[
\hat{J}\in C^{\infty}( \R^- \times \T \times \R^{2n} , \mathrm{L}(\R^{2n})) \quad \mbox{and} \quad 
\hat{H}\in C^{\infty}(\R^- \times \T \times \R^{2n}) 
\]
do not depend on $s$ for $s$ small enough, and $\hat{J}(s,t,x)$ is $\omega_0$-compatible for every $(s,t,x)\in \R^-\times \T \times \R^{2n}$. The fiber-wise derivative of $\mathscr{T}$ at $u$ is conjugated to the differential of $\widehat{\mathscr{T}}$ at $\hat{u}$, which has the form
\[
D\widehat{\mathscr{T}}(\hat{u})[v] = \bigl( \delbar v - J A v, \partial_t \pi v (0,\cdot) - \alpha v(0,\cdot) \bigr),
\]
where $J,A\in C^{\infty}(\R^-\times \T,\mathrm{L}(\R^{2n}))$ and $\alpha\in C^{\infty}(\T,\mathrm{L}(\R^{2n},\R^n))$ are defined as
\begin{eqnarray*}
J(s,t) & := & \hat{J}\bigl(s,t,\hat{u}(s,t) \bigr), \\
A(s,t) z & := & \hat{J}(s,t,\hat{u}) D_x \hat{J}(s,t,\hat{u}) [z] \bigl( \partial_t \hat{u} - X_{\hat{H}} (s,t,\hat{u}) \bigr) + D_x X_{\hat{H}} (s,t,\hat{u})[z], \\
\alpha(t) z & := & D_x \partial_p \hat{H} (0,t,\hat{u}(0,t))[z],
\end{eqnarray*}
for every $z\in \R^{2n}$. By using the fact that $\hat{u}$ is a zero of $\widehat{\mathscr{T}}$ and by differentiating the identity $\hat{J}(s,t,z)^2=-I$ with respect to $z\in \R^{2n}$, we can simplify the expression for $A$ and get
\[
A(s,t) z = D_x \hat{J}(s,t,\hat{u}) [z] \partial_s \hat{u} + D_x X_{\hat{H}} (s,t,\hat{u})[z], \qquad \forall z\in \R^{2n}.
\]
By (\ref{expo}), $\hat{u}(s,t)\rightarrow 0$ for $s\rightarrow -\infty$, exponentially with all its detivatives, so the smooth maps $J$ and $A$ extend continuously to $[-\infty,0]\times \T$ by setting
\[
J(-\infty,t) := \hat{J}\bigl(s,t,(0,\hat{p}(t))\bigr), \qquad A(-\infty,t) := D_x X_{\hat{H}}\bigl(s,t,(0,\hat{p}(t))\bigr),
\]
for $s$ small enough since, as we have noticed, $\hat{J}$ and $\hat{H}$ do not depend on $s$ for $s$ small enough. Moreoever, $J$ and $A$ are globally Lipschitz and conditions (\ref{regJ}) and (\ref{regA}) hold.

Since $H(t,\psi(s,t,z))=\hat{H}(s,t,z)$, and $x(t)=\psi(s,t,(0,\hat{p}(t)))$ for $s$ small enough, the fundamental solution $W$ of linear Hamiltonian system
\[
W'(t) = A(-\infty,t) W(t), \qquad W(0)=I,
\]
is conjugated to the differential of the Hamiltonian flow of $H$ along $x$ by the 1-periodic path of linear symplectic isomorphisms $t\mapsto D_x \psi(s,t,(0,\hat{p}(t)))$, which are vertical-preserving (again, for $s$ small enough). Therefore, the Conley-Zehnder index of $W$ equals the Conley-Zehnder index of $x$. By Theorem \ref{fredholm}, we conclude that $D\widehat{\mathscr{T}}(\hat{u})$ is a Fredholm operator of index $\mu_{CZ}(x)$, as claimed.
\end{proof}

\paragraph{Transversality.} Let $g$ be a Riemannian metric on $M$ and let $J_g$ be the corresponding Levi-Civita almost complex structure on $T^*M$, which in the horizontal-vertical splitting of $TT^*M$ induced by $g$,
\[
TT^*M = T^hT^*M \oplus T^v T^*M,
\]
takes the form 
\[
J_g = \left( \begin{array}{cc} 0 & -I \\ I & 0 \end{array} \right),
\]
where we are using the identifications
\[
T_x^h T^*M \cong T_x M \cong T_x^* M, \qquad T_x^v T^*M \cong T_x^* M,  
\]
given by the metric $g$. This almost complex structure is $\omega$-compatible.
We denote by $\mathscr{J}(g)$ the set of  $\omega$-compatible time-periodic smooth almost complex structures $J$ such that $\|J-J_g\|_{\infty}<+\infty$. The set $\mathscr{J}(g)$ has a natural structure of complete metric space (see \cite[Section 1.6]{as06}). As usual, a subset of $\mathscr{J}(g)$ is said to be generic if its complement is contained in a countable union of closed sets with empty interior. The proof of the next result is standard (see \cite{fhs95}). 

\begin{prop}
\label{trans}
For a generic choice of $J$ in $\mathscr{J}(g)$, the section $\mathscr{T}$ is transverse to the zero-section, except possibly at the stationary solution $u_0(s,t)=x(t)$. 
\end{prop}

Since the stationary solution $u_0$ belongs to $\mathscr{U}(x)$ also when the Fredholm index $\mu_{CZ}(x)$ of the fiberwise differential of the section $\mathscr{T}$, is negative, in general we cannot expect transversality at $u_0$. However, we shall prove that transversality at $u_0$ holds if
\begin{equation}
\label{pos}
d_{pp} H(t,x(t)) > 0 , \quad \forall t\in \T.
\end{equation}
Under this assumption, the Legendre transform
\begin{equation}
\label{legendre}
\mathscr{L} : \T \times T^*M \rightarrow \T \times TM, \qquad \mathscr{L}(t,q,p) \mapsto \bigl(t,q,d_p H(t,q,p)\bigr),
\end{equation}
maps a neighborhood of $\set{(t,x(t))}{t\in \T}$ diffeomorphically onto a neighborhood of 
\begin{equation}
\label{tqq}
\set{(t,\pi\circ x(t),(\pi\circ x)'(t))}{t\in \T}, 
\end{equation}
and the identity
\begin{equation}
\label{legtra}
L(t,q,v) + H(t,q,p) = \langle p, v \rangle, \qquad \mbox{for } (t,q,v) = \mathscr{L}(t,q,p),
\end{equation}
defines a smooth Lagrangian $L$ on a neighborhood of the set (\ref{tqq}). Here $\langle \cdot, \cdot \rangle$ denotes the duality pairing. The Legendre transform is involutive, meaning that
\begin{equation}
\label{inv}
\mathscr{L}^{-1} (t,q,v) = \bigl(t,q,d_v L(t,q,v) \bigr),
\end{equation}
for every $(t,q,v)$ in a neighborhood of the set (\ref{tqq}).
The Lagrangian action functional
\begin{equation}
\label{laf}
\mathbb{S}: C^{\infty}(\T,M) \rightarrow \R, \qquad \mathbb{S}(q) := \int_{\T} L\bigl(t,q(t),q'(t)\bigr)\, dt,
\end{equation}
is well-defined in a $C^1$-neighborhood of $\pi\circ x$. Then $\pi\circ x$ is a critical point of $\mathbb{S}$, and it is well-known that the Morse index $\ind(\pi\circ x;\mathbb{S})$ of $\pi\circ x$ with respect to $\mathbb{S}$, that is the dimension of any maximal subspace of the space of smooth sections of $(\pi\circ x)^*(TM)$ on which $d^2 \mathbb{S}(\pi\circ x)$ is negative-definite, coincides with the Conley-Zehnder index $\mu_{CZ}(x)$ of $x$, which is therefore non-negative (see e.g.\ \cite[Corollary 4.2]{aps08} for a proof of this fact which uses the sign conventions adopted also here). The following ``automatic transversality result'' is based on the properties of the Legendre transform.  

\begin{lem}
\label{auttrans}
Assume that the Hamiltonian $H$ satisfies (\ref{pos}) and let $u_0(s,t)=x(t)$ be the stationary solution. Then:
\begin{enumerate}
\item the fiberwise differential 
\[
D^f \mathscr{T}(u_0): T_{u_0} \mathscr{X} \rightarrow \mathscr{E}_{u_0} \oplus \mathscr{F}_{u_0}
\]
is surjective;
\item if $u\in T_{u_0} \mathscr{X}\setminus (0)$ belongs to the kernel of $D^f \mathscr{T}(u_0)$ and $\xi$ is the smooth section of the vector bundle $(\pi\circ x)^*(TM)$ defined by
\[
\xi(t) := D\pi (x(t)) [ u(0,t)],
\]
then $d^2 \mathbb{S}(\pi\circ x)[\xi]^2< 0$.
\end{enumerate}
\end{lem}

\begin{proof}
By using a smooth map $f:\T \times \R^n \rightarrow M$ such that $f(t,0)=\pi\circ x(t)$ and $f(t,\cdot)$ is a diffeomorphism and the corresponding chart as in Remark \ref{redatlas} (but without the dependence on $s$), we can assume that $H\in C^{\infty}(\T\times \R^{2n})$ and that $x(t)=(\hat{q}(t),\hat{p}(t))$, where $\hat{q}(t)=0$ and $\hat{p}$ solves
\[
\hat{p}'(t)= - \partial_q H(t,0,\hat{p}(t)).
\]
The kernel of $D^f \mathscr{T}(x)$ consists of the smooth maps $u\in H^2(\R^-\times \T,\R^{2n})$ which solve the system
\begin{eqnarray}
\label{ddt1}
\partial_s u + J(t,x) ( \partial_t u - D_x X_H (t,x) [u]) = 0,  \qquad \mbox{on } \R^-\times \T, 
\\ \label{ddt2}
\xi'(t) = \partial_{qp} H(t,x) \xi (t) +\partial_{pp} H(t,x) \eta(t),
\ \qquad \mbox{on } \T,
\end{eqnarray}
where $u(0,t) = (\xi(t),\eta(t)) = \zeta(t)$.
If $u\in \ker D^f \mathscr{T}(x)$ is non-zero, then the equation (\ref{ddt1}) implies that
\begin{equation}
\label{neg}
d^2 \mathbb{A}(x)[\zeta]^2 < 0.
\end{equation}
In fact, (\ref{ddt1}) can be rewritten abstractly as
\[
\partial_s u + \nabla^2 \mathbb{A}(x) u = 0,
\]
where $\nabla^2 \mathbb{A}(x)$ denotes the Hessian of $\mathbb{A}$ with respect to the $L^2$-metric
\[
\langle v,w \rangle_{L^2_J} := \int_{\T}g_J( v(t),w(t)) \, dt, \qquad \forall v,w\in C^{\infty}(\T,\R^{2n}).
\]
Then the function
\[
\varphi: \R^- \rightarrow \R, \qquad \varphi(s) = \| u(s,\cdot) \|_{L^2_J}^2,
\]
satisfies
\begin{eqnarray}
\label{d1}
\varphi'(s) &=& - 2 \langle \nabla^2 \mathbb{A}(x) u(s,\cdot), u(s,\cdot) \rangle_{L^2_J} = -2 \, d^2 \mathbb{A}(x) [u(s,\cdot)]^2, \\
\label{d2}
\varphi''(s) &=& 4 \bigl\| \nabla^2 \mathbb{A}(x) u(s,\cdot) \bigr\|_{L_J^2}^2.
\end{eqnarray}
By (\ref{d2}), $\varphi$ is convex and, being infinitesimal for $s\rightarrow -\infty$ and not identically zero, it satisfies $\varphi'(0)>0$, so (\ref{neg}) follows from (\ref{d1}). 

On the other hand, we claim that the equation (\ref{ddt2}) implies that
\begin{equation}
\label{uguali}
d^2 \mathbb{A}(x) [\zeta]^2 = d^2 \mathbb{S}(\hat{q}) [\xi]^2.
\end{equation}
In fact, is $\epsilon>0$ is small enough the map
\[
X : \set{q\in C^{\infty}(\T,\R^n)}{\|q\|_{C^1}< \epsilon} \rightarrow C^{\infty}(\T,\R^{2n}), \quad X(q)(t):= \bigl(q(t),\partial_v L(t,q(t),q'(t)) \bigr),
\]
is well defined and, by the involutive property of the Legendre transform, it maps $\hat{q}=0$ into $x$. By (\ref{legtra}) and (\ref{inv}), we have
\[
\mathbb{A}(X(q)) = \int_{\T} \Bigl( \partial_v L(t,q,q')\cdot q' - H(t,q,\partial_v L(t,q,q')) \Bigr)\, dt = \int_{\T} L(t,q,q')\, dt = \mathbb{S}(q),
\]
for every $q\in C^{\infty}(\T,\R^n)$ with $\|q\|_{C^1}<\epsilon$. By differentiating this identity twice at $\hat{q}$, taking into account the fact that $\hat{q}$ is a critical point of $\mathbb{S}$ and that $x=X(\hat{q})$ is a critical point of $\mathbb{A}$, we obtain the identity
\begin{equation}
\label{ided2}
d^2 \mathbb{S}(\hat{q})[\xi]^2 = d^2 \mathbb{A}(x) \bigl[ DX(\hat{q})[\xi] \bigr]^2 = d^2 \mathbb{A}(x) \bigl[ (\xi,\partial_{qv} L(\cdot,\hat{q},\hat{q}') \xi + \partial_{vv} L(\cdot,\hat{q},\hat{q}') \xi' )\bigr]^2,
\end{equation}
for every $\xi\in C^{\infty}(\T,\R^n)$. By differentiating the identity 
\[
v = \partial_p H\bigr(t,q,\partial_v L(t,q,v) \bigr)
\]
with respect to $v$ and with respect to $q$, we obtain the formulas
\[
\partial_{pp} H(t,\hat{q},\hat{p}) \partial_{vv} L (t,\hat{q},\hat{q}') = I, \qquad
\partial_{qp} H(t,\hat{q},\hat{p}) = - \partial_{pp} H(t,\hat{q},\hat{p}) \partial_{qv} L (t,\hat{q},\hat{q}').
\]
These identities imply that $(\xi,\eta)$ satisfies (\ref{ddt2}) if and only if 
\[
\eta = \partial_{qv} L(\cdot,\hat{q},\hat{q}') \xi + \partial_{vv} L(\cdot,\hat{q},\hat{q}') \xi' ,
\]
so (\ref{ided2}) implies (\ref{uguali}), as claimed.

Claim (ii) immediately follows from (\ref{neg}) and (\ref{uguali}). 

The second differential of $\mathbb{S}$ at $\hat{q}$ extends to a continuous non-degenerate symmetric bilinear form on $H^1(\T,\R^n)$. Let $\mathscr{W}$ be a closed linear subspace of $H^1(\T,\R^n)$ on which $d^2 \mathbb{S}(\hat{q})$ is positive-definite and which has codimension $\ind (\hat{q};\mathbb{S})$.
Consider the continuous linear mapping
\[
Q : H^2(\R^-\times \T,\R^{2n}) \rightarrow H^1(\T,\R^n), \quad u \mapsto \pi u(0,\cdot),
\]
where as usual $\pi$ is the projection onto the first factor of $\R^n\times \R^n$. By (\ref{neg}) and (\ref{uguali}) we have
\[
d^2 \mathbb{S}(\hat{q}) [ Qu ]^2 < 0, \qquad \forall u\in \ker D^f \mathscr{T}(u_0) \setminus (0).
\]
It follows that $Q$ is injective on the kernel of $D^f \mathscr{T}(u_0)$ and 
\[
\bigl( Q \ker D^f \mathscr{T}(u_0) \bigr) \cap \mathscr{W} = (0).
\]
Therefore,
\[
\dim \ker D^f \mathscr{T}(u_0) = \dim \bigl( Q \ker D^f \mathscr{T}(u_0) \bigr) \leq \codim \, \mathscr{W} = \ind (\hat{q};\mathbb{S}) = \mu_{CZ}(x).
\]
Since $D^f \mathscr{T}(u_0)$ has Fredholm index $\mu_{CZ}(x)$,
the above inequality implies that $D^f \mathscr{T}(u_0)$ is surjective, concluding the proof of (i).
\end{proof} 

Proposition \ref{trans} and Lemma \ref{auttrans} (i) have the following immediate consequence:

\begin{cor}
\label{isaman}
Let $x$ be a non-degenerate 1-periodic orbit of $X_H$ at which (\ref{pos}) holds. Then for a generic choice of $J$ in $\mathscr{J}(g)$, $\mathscr{U}(x)$ is a non-empty smooth submanifold of $\mathscr{X}$ of dimension $\mu_{CZ}(x)$.
\end{cor}

\begin{rem}
\label{nonorient} We have used the orientability assumption on $M$ in order to trivialize the vector bundle $(\pi\circ x)^*(TM)$, so to have vertical-preserving trivializations of $x^*(TT^*M)$ and maps $\psi$ as in Remark \ref{redatlas}. When $M$ is not orientable and $\pi\circ x$ is an orientation-reversing loop, $x^*(TT^*M)$ admits no vertical-preserving trivialization over $\T$. However, one can consider a trivialization $V$ of $(\pi\circ x)^*(TM)$ over $[0,1]$ and use the induced trivialization $U(t)=V(t)\oplus (V(t)^*)^{-1}$ of $x^*(TT^*M)\rightarrow [0,1]$ to define the Conley-Zehnder index (in this case a Maslov index) and to read the fiberwise differential of the section $\mathscr{T}$. Actually, the more general framework of Hamiltonian systems with non-local conormal boundary conditions allows a setting under which the orientable and non-orientable loops are treated in the same way, together with more general Lagrangian intersection problems. See \cite{aps08} for a description of this setting and \cite{web02} for another approach.
\end{rem}

\paragraph{Compactness.} Under a standard assumption on $H$, the elements of the reduced unstable manifold $\mathscr{U}(x)$ satisfy uniform action and energy bounds:

\begin{lem}
\label{energy}
Assume that $H$ satisfies 
\begin{equation}
\label{actint}
a:= \inf_{(t,q,p)\in \T \times T^*M} \Bigl( \langle p, d_p H(t,q,p) \rangle - H(t,q,p) \Bigr) > -\infty.
\end{equation}
 Then every $u\in  \mathscr{U}(x)$ satisfies the action and energy bounds
\[
a \leq \mathbb{A}(u(s,\cdot)) \leq \mathbb{A}(x), \qquad
E(u) = \int_{\R^- \times \T} |\partial_s u|_J^2\, ds\, dt \leq \mathbb{A}(x) - a. 
\]
\end{lem} 

\begin{proof}
Since $u$ satisfies (i), (ii) and (iii), we have
\[
E(u) = \int_{\R^- \times \T} |\partial_s u|_J^2\, ds\, dt = \mathbb{A}(x) - \mathbb{A}(u(0,\cdot)).
\]
By (iv) and (\ref{actint}) we find
\[
\mathbb{A}(u(0,\cdot)) = \int_{\T} \Bigl( \langle p, d_p H(t,q,p) \rangle - H(t,q,p) \Bigr)\, dt \geq a,
\]
where $u(0,t)=(q(t),p(t))$. The action and energy bounds follow.
\end{proof}

Since $T^*M$ is not compact, the next step in order to have the standard compactness statement for the set $\mathscr{U}(x)$ is proving a uniform $L^{\infty}$ bound for its element. Such a bound holds under additional assumptions on the Hamiltonian $H$ and on the $\omega$-compatible almost complex structure $J$. We treat two different sets of conditions: the first one ($H$ quadratic at infinity and $J$ close to a Levi-Civita almost complex structure) is the condition which is considered also in \cite{as06}, the second one ($H$ radial and $J$ of contact type outside a compact set) is the condition which is more common in symplectic homology (see e.g.\ \cite{sei06b}).

\begin{defn}
\label{qai}
We say that $H\in C^{\infty}(\T\times T^*M)$ is quadratic at infinity if there exist numbers $h_0>0$, $h_1\in \R$ and $h_2>0$ such that
\begin{eqnarray}
\label{H1}
\langle p, d_p H(t,q,p) \rangle - H(t,q,p) \geq h_0 |p|^2 - h_1, \\
\label{H2}
|\nabla_p H(t,q,p) | \leq h_2 ( 1 + |p|), \qquad |\nabla_q H(t,q,p) | \leq h_2 (1 + |p|^2 ),
\end{eqnarray}
for every $(t,q,p)\in \T\times T^*M$. Here the norm $|\cdot|$ and the horizontal and vertical components of the gradient $\nabla_q$ and $\nabla_p$ are induced by some Riemannian metric on $M$, but these conditions do not depend on the choice of such a metric, up to the replacement of the numbers $h_0,h_1,h_2$.
\end{defn}

The name ``quadratic at infinity'' refers to the fact that (\ref{H1}) and (\ref{H2}) imply that $H$ satisfies the bounds
\[
\frac{1}{2} h_0 |p|^2 - h_3 \leq H(t,q,p) \leq h_4 |p|^2 + h_5,
\]
for suitable positive numbers $h_3,h_4,h_5$. Typical examples of Hamiltonians which are quadratic at infinity are the physical Hamiltonians of the form
\begin{equation}
\label{phys}
H(t,q,p) = \frac{1}{2} | p - \theta(t,q) |^2 + V(t,q),
\end{equation}
where $|\cdot|$ is the norm on $T^*M$ induced by some Riemannian metric on $M$, $\theta$ is a smooth time-periodic one-form on $M$ and $V$ is a smooth function on $\T\times M$. When $H$ is quadratic at infinity, for every $A\in \R$ the set of 1-periodic orbits $x$ with action $\mathbb{A}(x)\leq A$ is compact in $C^{\infty}(\T,T^*M)$ (see \cite[Lemma 1.10]{as06}).  

The proof of the next compactness result builds on the estimates proved in \cite{as06}.

\begin{prop}
\label{komp1}
There exists a positive number $j_0$ such that if $\|J-J_g\|_{\infty} \leq j_0$, where $J_g$ is the Levi-Civita almost complex structure induced by a Riemannian metric $g$ on $M$, and $H$ is quadratic at infinity, then there exists a compact subset of $T^*M$ which contains the image of all the elements of $\mathscr{U}(x)$.
\end{prop}

\begin{proof}
Let $u\in \mathscr{U}(x)$ and write $u(s,t)=(q(s,t),p(s,t))$. By (\ref{H1}), the assumption of Lemma \ref{energy} is fulfilled, hence the energy of $u$ is uniformly bounded, and so is the action of $u(s,\cdot)$, for every $s\in \R^-$.

By \cite[Lemma 1.12]{as06}, $p$ is uniformly bounded in $H^1(]-1,0[ \times \T)$. By the continuity of the trace operator, $p(0,\cdot)$ is uniformly bounded in $H^{1/2}(\T)$, and a fortiori in $L^r(\T)$, for every $r<+\infty$. By the boundary condition (iv), together with the first estimate in (\ref{H2}), we deduce that $\partial_t q(0,\cdot)$ is uniformly bounded in $L^r(\T)$, for every $r<+\infty$. Therefore, $q(0,\cdot)$ is uniformly bounded in $W^{1,r}(\T)$, so by \cite[Theorem 1.14 (iii)]{as06}, the image of $u$ is uniformly bounded, provided that $\|J-J_g\|_{\infty}$ is small enough. 
\end{proof}

Although it is not needed in this paper, we now prove the analogous compactness statement in the setting which is more common in symplectic homology. Let $W\subset T^*M$ be a compact domain with smooth boundary, star-shaped with respect to the zero-section. For instance, $W$ could be the unit cotangent disk bundle induced by a metric on $M$. Then the restriction $\alpha:=\lambda|_{\partial W}$ of the standard Liouville form $\lambda=p\, dq$ to $\partial W$ is a contact form on $\partial W$, $(W,\lambda)$ is a Liouville domain, and $T^* M$ is naturally identified with its completion 
\[
\widehat{W} := W \cup_{\partial W} \bigl( [1,+\infty[ \times \partial W \bigr),
\]
(see e.g.\ \cite{sei06b} for the relevant definitions). Indeed, the standard Liouville vector field $Y=p \, \partial_p$ is transverse to $\partial W$ and its flow $\phi$ induces the diffeomorphism
\[
]0,+\infty[ \times \partial W \rightarrow T^*M \setminus \mathbb{O}_{T^*M} , \quad (\rho,x) \mapsto \phi_{\log \rho} (x),
\]
where $\mathbb{O}_{T^*M}$ denotes the the zero-section of $T^*M$. We denote by $\rho: T^*M \rightarrow \R^+$ the corresponding radial variable, extended as $\rho=0$ on the zero-section $\mathbb{O}_{T^*M}$, 
and we recall that the pull-back by the above diffeomorphism of the Liouville vector field $Y$ is $\rho \partial_{\rho}$. A Hamiltonian $H\in \C^{\infty}(\T\times T^*M)$ is said to be {\em radial outside a compact set} if there exists $\rho_0>0$ such that
\begin{equation}
\label{radial}
H(t,x) = h(\rho(x)), \qquad \mbox{on } \T \times \rho^{-1}([\rho_0,+\infty[),
\end{equation}
for some smooth function $h:[\rho_0,+\infty[ \rightarrow \R$. In this case, if $x=(q,p) \in T^*M$ and $\rho(x)\geq \rho_0$ then
\begin{equation}
\label{pdH} 
\langle p, d_p H(t,x) \rangle = dH(t,x) [Y(x)] = dh(\rho(x)) \bigl[ \rho(x) \partial_{\rho} \bigr] = \rho(x) h'(\rho(x)).
\end{equation}
By the above identity, the assumption of Lemma \ref{energy} is fulfilled whenever the function
\[
\rho \mapsto \rho h'(\rho) - h(\rho)
\]
is bounded from below on $[\rho_0,+\infty[$; this happens, for instance, when $h$ has the form
\[
h(\rho) = \lambda \rho^{\mu} + \eta,
\]
with $\lambda>0$, $\mu\geq 1$, $\eta\in \R$.
The time-periodic $\omega$-compatible almost complex structure $J$ is said to be {\em of contact type outside a compact set} if
\begin{equation}
\label{contact}
d\rho \circ J = \lambda, \qquad \mbox{on } \T \times \rho^{-1}([\rho_0,+\infty[),
\end{equation}
for $\rho_0>0$ large enough. Equivalently, $J$ preserves the symplectic	splitting
\[
T_{(\rho,x)} T^*M = \ker \alpha \oplus ( \R R \oplus \R Y),
\]
on the set $\rho^{-1}([\rho_0,+\infty[)$, 
where $R$ is the Reeb vector field of $(\partial W,\alpha)$, and
\[
J R = Y, \quad J Y = - R.
\]
The next result shows that when $H$ is radial and $J$ is of contact type outside a compact set, the maps in $\mathscr{U}(x)$ satisfy a uniform $L^{\infty}$ bound.

\begin{prop}
\label{komp2}
Assume that the Hamiltonian $H\in C^{\infty}(\T\times T^*M)$ satisfies (\ref{radial}) for some $h\in C^{\infty}([\rho_0,+\infty[)$ such that
\begin{equation}
\label{sotto}
\inf_{\rho\geq \rho_0} \bigl( \rho h'(\rho) - h(\rho) \bigr) > -\infty,
\end{equation}
and that the time-periodic $\omega$-compatible almost complex structure $J$ satisfies (\ref{contact}). Let $\rho_1\geq \rho_0$ be such that
\[
x(\T) \subset \rho^{-1}([0,\rho_1[).
\]
Then the image of every $u\in \mathscr{U}(x)$ is contained in the compact set $\rho^{-1}([0,\rho_1])$.
\end{prop}

\begin{proof}
If we consider the real valued function
\[
r: \R^- \times \T \rightarrow \R, \qquad r(s,t) := \rho(u(s,t)),
\]
we must prove that $r\leq \rho_1$. Arguing by contradiction, we assume the open subset of $\R^- \times \T$
\[
\Omega := r^{-1} (]\rho_1,+\infty[) = 
u^{-1} \bigl( \rho^{-1}(]\rho_1,+\infty[) \bigr)
\]
is not empty. Since $u(s,\cdot)$ converges uniformly to $x$ for $s\rightarrow -\infty$ and $\rho< \rho_1$ on $x(\T)$, the set $\Omega$ is bounded, hence $r$ achieves its maximum - necessarily larger than $\rho_1$ - on $
\Omega$. We denote by $\partial \Omega$ the relative boundary of $\Omega$ in $\R^- \times \T$ and by $\Gamma:= \Omega \cap (\T \times \{0\})$ the intersection of $\Omega$ with the boundary of the half-cylinder. Therefore, the boundary of $\Omega$ as a subset of the complete cylinder $\R \times \T$ is $\partial \Omega \cup \Gamma$.

By (\ref{radial}) and (\ref{contact}), a standard computation (see e.g.\ \cite[Lemma 2.4]{as09}) shows that $r$ satisfies
\begin{eqnarray}
\label{elliptic}
\Delta r - r h''(r) \partial_s r = | \partial_s u |_J^2 \geq 0, \\
\label{bdry}
\partial_s r = -\lambda (\partial_t u) + rh'(r), 
\end{eqnarray}
on the set $\Omega$. By (\ref{elliptic}) and the weak maximum principle, $r$ achieves its maximum on $\partial \Omega\cup \Gamma$. Since this maximum is larger than $\rho_1$ and $r=\rho_1$ on $\partial \Omega$, the maximum is achieved at a point $(0,t)$ of $\Gamma$ and by the Hopf lemma,
\[
\partial_s r (0,t) >0.
\]
On the other hand, (\ref{bdry}), the fact that $u$ satisfies the boundary condition (iv), and (\ref{pdH}) imply that at $(0,t)$ the function $r$ satisfies
\[
\partial_s r (0,t) = - \langle p, \partial_t q  \rangle
+ r h'(r) = - \langle p, d_p H(t,u) \rangle - r h'(r) = - r h'(r) + r h'(r) = 0,
\] 
where $u=(q,p)$.
This contradiction shows that $\Omega$ must be empty, concluding the proof.
\end{proof}

Propositions \ref{komp1} and \ref{komp2} have the following consequence, whose proof is standard and hence only sketched here: 

\begin{cor}
\label{komp}
Assume that either $H$ is quadratic at infinity and $J$ is $C^0$-close enough to a Levi-Civita almost complex structure, or that $H$ is radial, satisfies (\ref{sotto}), and $J$ is of contact type outside a compact set.
Then the space $\mathscr{U}(x)$ is relatively compact in $C^{\infty}_{\mathrm{loc}} (\R^-\times \T,T^*M)$.
\end{cor}
 
\begin{proof}
The first step is to show that for every $s_0<0$, $\nabla u$ is bounded on $[s_0,0]\times \T$, uniformly in $u\in \mathscr{U}(x)$.
Assume, by contradiction, that there are sequences $(u_n)\subset \mathscr{U}(x)$ and $(z_n)\subset [s_0,0] \times \T$ such that $|\nabla u_n(z_n)|_{J}\rightarrow +\infty$. Up to a subsequence, we may assume that $(z_n)$ converges to some $z=(\bar{s},\bar{t}) \in [s_0,0] \times \T$. If $\bar{s} < 0$, a standard blow-up analysis at $z$ implies the existence of a non-constant pseudo-holomorphic sphere in $(T^*M,J(\bar{t},\cdot))$, which cannot exist because $\omega$ is exact. Assume that $\bar{s}=0$ and set $(q_n(t),p_n(t)) = u_n(0,t)$. By the boundary condition (iv) and the fact that the image of $u_n$ is uniformly bounded, $q_n'$ is uniformly bounded in $L^{\infty}$, so up to a subsequence $(q_n)$ converges uniformly to some $q\in C(\T,M)$. Then the standard blow-up analysis at $z$ implies the existence of a non-constant pseudo-holomorphic disk in $(T^*M,J(\bar{t},\cdot))$ with boundary in $T_{q(\bar{t})}^* M$. Such a disk cannot exist, because the Liouville form $\lambda$ vanishes on $T_{q(\bar{t})}^* M$. This contradiction shows that $\nabla u$ is uniformly bounded on $[s_0,0]\times \T$. Uniform bounds for all the higher derivatives on compact subsets of $\R^-\times \T$ follow from elliptic bootstrap.
\end{proof}

\section{The linear setting for orientations}
\label{linorsec}

In this section we recall some basic facts about determinant bundles and orientation of linear Cauchy-Riemann operators on cylinders and half-cylinders.

\paragraph{The determinant bundle.}
Let $E$ and $F$ be real Banach spaces. We recall that the determinant bundle $\Det(\mathrm{Fr}(E,F))$ over the space of Fredholm operators $\mathrm{Fr}(E,F)$ is the line bundle whose fiber at $T\in \mathrm{Fr}(E,F)$ is the line
\[
\Det(T) := \Lambda^{\max}(\ker T) \otimes (\Lambda^{\max}(\coker
T))^* .
\]
Here $\Lambda^{\max}$ denotes the top exterior power functor on the category of finite dimensional real vector spaces and linear maps among them.
See \cite{qui85} or \cite[Section 7]{ama09} for the construction of the analytic bundle
structure of $\Det(\mathrm{Fr}(E,F))$. 

Two linear isomorphisms
\[
R:E' \cong E \qquad \mbox{and} \qquad L: F \cong F'
\]
induce canonical smooth line bundle isomorphisms
\[
\Det(\mathrm{Fr}(E,F)) \rightarrow \Det (\mathrm{Fr}(E',F)) \qquad \mbox{and} \qquad
\Det(\mathrm{Fr}(E,F)) \rightarrow \Det (\mathrm{Fr}(E,F'))
\]
which lift the diffeomorphisms
\[
\mathrm{Fr} (E,F) \rightarrow \mathrm{Fr} (E',F), \; T \mapsto TR, \qquad \mbox{and} \qquad \mathrm{Fr} (E,F) \rightarrow \mathrm{Fr} (E,F'), \; T \mapsto LT.
\]
Actually, these are special cases of the fact that the composition of Fredholm operators lifts to the determinant bundles (see \cite[Section 7]{ama09}).

\paragraph{A class of linear Cauchy-Riemann operators on cylinders.}
Endow the extended real line $\overline{\R}=[-\infty,+\infty]$ with the differentiable structure which is induced by the bijection
\[
\Bigl[-\frac{\pi}{2},\frac{\pi}{2}\Bigr] \rightarrow \overline{\R}, \qquad s \mapsto \tan s, \quad \pm \frac{\pi}{2} \mapsto \pm \infty.
\]
Such a differentiable structure is inherited by the extended half lines $\overline{\R}^+=[0,+\infty]$ and $\overline{\R}^-=[-\infty,0]$.

Let $\mathscr{A}$ be the space of loops
\[
a: \T \longrightarrow \mathrm{sp}(\R^{2n},\omega_0)
\]
into the Lie-algebra of the symplectic group which are non-degenerate, meaning that the solution $W:[0,1]\rightarrow \mathrm{Sp}(\R^{2n},\omega_0)$ of
\[
W'(t) = a(t) W(t), \qquad W(0) = I,
\]
is a non-degenerate symplectic path, i.e.\ $W(1)$ does not have the eigenvalue 1.
We denote by
\[
\Sigma_{\partial} := \set{D_A}{A\in C^{\infty}\bigl(\overline{\R}\times \T, \mathrm{L}(\R^{2n}) \bigl) \mbox{ with } A(\pm \infty,\cdot) \in \mathscr{A} }
\]
the set of operators
\[
D_A : H^2(\R\times \T,\R^{2n}) \rightarrow H^1(\R\times \T,\R^{2n}), \qquad D_A u = \partial_s u + J_0 \partial_t u - J_0 A u,
\]
with non-degenerate asymptotics. It is well-known that $\Sigma_{\partial}$ consists of Fredholm operators (see e.g.\ \cite[Theorem 4.1]{sz92} for the same operators from $H^1$ to $L^2$; the proof in the case of higher Sobolev regularity is similar and uses the smoothness of $A$). The subscript $\partial$ refers to the fact that this is the set of linear operators that one gets when linearizing the Floer equation on cylinders, which is used in the definition of the boundary operator of the Floer complex.

The restriction of the determinant bundle over the space of Fredholm operators from $H^2(\R\times \T,\R^{2n})$ into $H^1(\R\times \T,\R^{2n})$ to $\Sigma_{\partial}$ is not orientable (that is, not trivial) (see \cite[Theorem 2]{fh93}), but if we fix the asymptotics we get an orientable line bundle: given $a_-$ and $a_+$ in $\mathscr{A}$, the space
\[
\Sigma_{\partial} (a_-,a_+) := \set{D_A \in \Sigma_{\partial}}{A(\pm \infty,\cdot) = a_{\pm}}
\]
is convex, hence the determinant bundle  restricts to an orientable line bundle, which we denote by
\[
\Det\bigl(\Sigma_{\partial}(a_-,a_+)\bigr) \rightarrow \Sigma_{\partial}(a_-,a_+).
\]
By a linear glueing construction (see \cite[Section 3]{fh93}), two orientations $o(a_0,a_1)$ and $o(a_1,a_2)$ of $\Det(\Sigma_{\partial}(a_0,a_1))$ and $\Det(\Sigma_{\partial}(a_1,a_2))$ induce canonically an orientation of $\Det(\Sigma_{\partial}(a_0,a_2))$, which we denote by
\[
o(a_0,a_1) \# o(a_1,a_2).
\]
Gluing of orientations is anti-symmetric: if $\widehat{o}$ denotes the opposite orientation of $o$, then
\[
\widehat{o}_1 \# o_2 = o_1 \# \widehat{o}_2 = \widehat{o_1 \# o_2}.
\]
The space $\Sigma_{\partial}(a_0,a_0)$ contains the translation invariant operator $D_{a_0}$, which is an isomorphism (see e.g. by \cite[proof of Theorem 4.1]{sz92}). Therefore
\[
\Det (D_{a_0}) = \R \otimes \R^*
\]
carries a canonical orientation. If $o(a_0,a_0)$ is the orientation of $\Sigma_{\partial}(a_0,a_0)$ which agrees the canonical one at $D_{a_0}$, then $o(a_0,a_0)$ acts as a unit for the gluing of orientations:
\begin{equation}
\label{unit0}
o(a,a_0) \# o(a_0,a_0) = o(a,a_0), \qquad o(a_0,a_0) \# o(a_0,a) = o(a_0,a),
\end{equation}
for every $a\in \mathscr{A}$ and every orientation $o(a,a_0)$, resp.\ $o(a_0,a)$, of $\Det(\Sigma_{\partial}(a,a_0))$, resp.\ $\Det(\Sigma_{\partial}(a_0,a))$.

This construction is associative, that is
\[
\bigl( o(a_0,a_1) \# o(a_1,a_2) \bigr) \# o(a_2,a_3) = o(a_0,a_1) \# \bigl( o(a_1,a_2) \# o(a_2,a_3) \bigr).
\]

Conjugation of an operator $D_A\in \Sigma_{\partial}(a_-,a_+)$ by some $U$ in $C^{\infty}(\overline{\R} \times \T,\mathrm{U}(n))$ produces an operator $U D_A U^{-1}$ which is of the form $D_B$, where
\begin{equation}
\label{laB}
B:= U A U^{-1}- J_0 \partial_s U U^{-1} + \partial_t U U^{-1}.
\end{equation}
In particular, if
\[
U(\pm \infty, t) = I , \qquad \forall t\in \T,
\]
then $U D_A U^{-1}$ belongs to the same space $\Sigma_{\partial}(a_-,a_+)$. The following result is proved in \cite[Lemma 13]{fh93}:

\begin{lem}
\label{conju1}
Let $a_-,a_+\in \mathscr{A}$ and let $U\in C^{\infty}(\overline{\R} \times \T,\mathrm{U}(n))$ be such that 
\[
U(\pm \infty, t) = I , \qquad \forall t\in \T.
\]
Then the canonical lift to the determinant bundle of the map
\[
\Sigma_{\partial}(a_-,a_+) \rightarrow \Sigma_{\partial}(a_-,a_+), \quad D_A \mapsto U D_A U^{-1},
\]
is orientation-preserving.
\end{lem}

\paragraph{A class of linear Cauchy-Riemann operators on half-cylinders.} Consider the space of linear operators
\[
\Sigma_{\mathscr{U}} := \set{T_{A,\alpha}}{ A\in C^{\infty}(\overline{\R}^-\times \T,\mathrm{L}(\R^{2n})) \mbox{ with } A(-\infty,\cdot)\in \mathscr{A}, \; \alpha \in C^{\infty}(\T,\mathrm{L}(\R^{2n},\R^n)) }
\]
of the form
\begin{eqnarray*}
T_{A,\alpha} &:& H^2(\R^-\times \T,\R^{2n}) \rightarrow H^1(\R^-\times \T,\R^{2n}) \times H^{1/2}(\T,\R^n), \\ u & \mapsto & \bigl( \partial_s u + J_0 \partial_t u - J_0 A u, \partial_t \pi u(0,\cdot) - \alpha u(0,\cdot) \bigr),
\end{eqnarray*}
where $\pi:\R^{2n}=\R^n \times \R^n\rightarrow \R^n$ is the projection onto the first factor. By Theorem \ref{fredholm}, each element of $\Sigma_{\mathscr{U}}$ is a Fredholm operator. As it is shown in Section \ref{secrum}, this is the set of operators which arise when linearizing the problem whose solution is the reduced unstable manifold $\mathscr{U}(x)$ of a Hamiltonian periodic orbit $x$.  

Let $a\in \mathscr{A}$. The determinant bundle over the space of Fredholm operators from $H^2(\R^-\times \T,\R^{2n})$ into $H^1(\R^-\times \T,\R^{2n}) \times H^{1/2}(\T,\R^n)$ restricts to an orientable line bundle on the convex space
\[
\Sigma_{\mathscr{U}} (a) := \set{T_{A,\alpha}\in \Sigma_{\mathscr{U}}}{A(-\infty,\cdot)=a},
\]
which we denote by
\[
\Det\bigl(  \Sigma_{\mathscr{U}} (a) \bigr) \rightarrow \Sigma_{\mathscr{U}} (a).
\]
Let $a_0$ and $a_1$ be elements of $\mathscr{A}$.
A linear gluing construction, which is completely analogous to the one which is described in \cite[Section 3]{fh93}, associates to an orientation $o(a_0,a_1)$ of $\Sigma_{\partial} (a_0,a_1)$ and an orientation $o(a_1)$ of $\Det(\Sigma_{\mathscr{U}}(a_1))$ an orientation of $\Det(\Sigma_{\mathscr{U}}(a_0))$, which we denote by
\[
o(a_0,a_1) \# o(a_1).
\]
Again, if $o(a_0,a_0)$ agrees with the canonical orientation of $\Det (D_{a_0}) = \R \otimes \R^*$, then
\begin{equation}
\label{unit}
o(a_0,a_0) \# o(a_0) = o(a_0).
\end{equation}
Associativity now reads as
\begin{equation}
\label{ass}
\bigl( o(a_0,a_1) \# o(a_1,a_2) \bigr) \# o(a_2) =
o(a_0,a_1) \# \bigl( o(a_1,a_2) \# o(a_2) \bigr) .
\end{equation}

The subgroup of $\mathrm{U}(n)$ of automorphisms which preserve the vertical space $(0) \times (\R^n)^*=i\R^n$ in the identification $T^* \R^n \cong \C^n$, $(q,p) \mapsto q+ip$, is precisely $\mathrm{O}(n)$, the orthogonal group of $\R^n$, whose elements are extended to $\C^n$ by complex linearity. We shall be interested in the subgroup of $\mathrm{O}(n)$ of orientation-preserving automorphisms, that is $\mathrm{SO}(n)$. 
 
Let $T_{A,\alpha} \in \Sigma_{\mathscr{U}}(a)$. Conjugation of $T_{A,\alpha}$ by some
$U$ in $C^{\infty}(\overline{\R}^- \times \T,\mathrm{U}(n))$ such that
\[
U(0,t)\in \mathrm{O}(n) \qquad \forall t\in \T,
\]
produces the operator
$(U \times U(0,\cdot)) T_{A,\alpha} U^{-1}$, which has the form $T_{B,\beta}$, where $B$ is as in (\ref{laB}) and
\[
\beta := \pi \partial_t U(0,\cdot) U^{-1}(0,\cdot) + U(0,\cdot) \alpha U^{-1}(0,\cdot).
\]
Therefore, if we also assume 
\[
U(- \infty, t) = I  \qquad \forall t\in \T,
\]
then $(U \times U(0,\cdot)) T_{A,\alpha} U^{-1}$ belongs to the same space $\Sigma_{\mathscr{U}}(a)$. The proof of the following result is analogous to the proof of \cite[Proposition 2.1]{as13}: 

\begin{lem}
\label{conju2}
Let $a\in \mathscr{A}$ and let $U\in C^{\infty}(\overline{\R}^- \times \T,\mathrm{U}(n))$ be such that
\[
U(- \infty, t) = I  \qquad \mbox{and} \qquad U(0,t)\in \mathrm{SO}(n), \qquad \forall t\in \T.
\]
Then the canonical lift to the determinant bundle of the map
\[
\Sigma_{\mathscr{U}}(a) \rightarrow \Sigma_{\mathscr{U}}(a), \quad T_{A,\alpha} \mapsto \bigl( U \times U(0,\cdot)\bigr) T_{A,\alpha} U^{-1},
\]
is orientation-preserving if and only if the lift of the loop $U(0,\cdot)$ to the two-fold covering $\mathrm{Spin}(n)$ of $\mathrm{SO}(n)$ is closed.
\end{lem}

Since $\mathrm{SO}(1)=\{1\}$, $\pi_1(\mathrm{SO}(2))=\Z$ and $\pi_1(\mathrm{SO}(n))=\Z_2$ for $n\geq 3$, the above lemma says that the above map is orientation preserving if an only if either
\begin{enumerate}
\item $n=1$, or
\item $n=2$ and the homotopy class of the loop $U(0,\cdot)$ in $\mathrm{SO}(2)$ is an even multiple of the generator of $\pi_1(\mathrm{SO}(2))=\Z$, or
\item $n\geq 3$ and the loop $U(0,\cdot)$ is contractible in $\mathrm{SO}(n)$.
\end{enumerate}

\section{The Floer and the Morse complexes}
\label{defcompsec}

\paragraph{The Floer complex of $H$.} We start by recalling the definition of the Floer complex with integer coefficients. Our main reference for orientations is \cite{fh93}, but we slightly deviate from the approach which is adopted there: instead than fixing a priori a coherent orientation, we shall follow the approach of P.\ Seidel and M.\ Abouzaid (see e.g.~ \cite{sei08b} and \cite[Section 3.3]{abo11}) of considering orientations in the definition of the generators. The two approaches are completely equivalent, but the latter has the advantage of reducing the amount of arbitrary data which one must choose.

Let $H\in C^{\infty}(\T\times T^*M)$ be quadratic at infinity in the sense of Definition \ref{qai}. The set of 1-periodic orbits of the corresponding 1-periodic Hamiltonian vector field $X_H$ is denoted by $\mathscr{P}(H)$. We assume that each element of $\mathscr{P}(H)$ is non-degenerate. The symbol $\mathscr{P}_k(H)$ denotes the set of $x\in \mathscr{P}_k(H)$ with Conley-Zehnder index $k$.

Let $J$ be a smooth $\omega$-compatible almost complex structure depending 1-periodically on time. We assume $J$ to be generic, so that for each pair $(x,y)\in \mathscr{P}(H)\times \mathscr{P}(H)$ the space of Floer cylinders
\begin{eqnarray*}
\mathscr{M}_{\partial^F}(x,y) :=& \Bigl\{ & u\in C^{\infty}(\R\times \T,T^*M) \, \Big| \, \partial_s u + J(t,u) (\partial_t u - X_{H_t}(u)) = 0, \\ 
&& \lim_{s\rightarrow -\infty} u(s,\cdot) = x \mbox{ and } \lim_{s\rightarrow +\infty} u(s,\cdot) = y \mbox{ in } C^{\infty}(\T,T^*M)\Bigr\}
\end{eqnarray*}
is the set of zeroes of a smooth Fredholm section 
\[
\mathscr{D} : \mathscr{B} \rightarrow \mathscr{H}
\]
of a Hilbert bundle $\mathscr{H}\rightarrow \mathscr{B}$ which is transverse to the zero-section. As base $\mathscr{B}$ of this bundle one can take a Hilbert manifold of cylinders $u:\R\times \T \rightarrow T^*M$ of Sobolev class  $H^2_{\mathrm{loc}}$ which converge to $x$ and $y$ for $s\rightarrow -\infty$ and $s\rightarrow +\infty$ in a suitable fashion. The fibers $\mathscr{H}_u$ consist of sections of $u^*(TT^*M)$ of Sobolev class $H^1_{\mathrm{loc}}$ which tend to zero for $s\rightarrow \pm \infty$ in a suitable way (see Section \ref{secrum}). In particular, each set $\mathscr{M}_{\partial^F}(x,y)$ is a smooth finite dimensional manifold.

For every $x\in \mathscr{P}(H)$ we fix once and for all a unitary trivialization of the complex bundle $(x^*(TT^*M),u^*(J))$ 
\[
\Theta_x: \T \times \C^n \rightarrow u^*(TT^*M), \qquad (t,\xi) \mapsto \Theta_x(t) \xi,
\]
which is vertical-preserving, meaning that it maps the subspace $i\R^n \cong (0) \times (\R^n)^*$ into the vertical subbundle $T^v T^*M = \ker D\pi$. The existence of vertical-preserving trivialization is guaranteed by the orientability of $M$ (see \cite[Lemma 1.2]{as06}). We also assume these trivializations to be orientation-preserving, meaning that $\Theta_x(t)|_{i\R^n}$ is orientation preserving (since the manifold $M$ is oriented, so is the vector bundle $T^v T^*M$). We also fix once and for all an element $a_0$ of $\mathscr{A}$.

The trivialization $\Theta_x$ determines an element $a_x$ of $\mathscr{A}$. Indeed, the differential of the flow of $X_H$ along $x$ can be read through the trivialization $\Theta_x$ and becomes a symplectic path $W_x:[0,1]\rightarrow \mathrm{Sp} (2n)$, which solves the linear Cauchy problem
\[
W_x'(t) = a_x(t) W_x(t), \qquad W_x(0) = I,
\]
where $a_x: \T \rightarrow \mathrm{sp}(2n)$ belongs to $\mathscr{A}$ because $x$ is non-degenerate. 

Let $F_k(H)$ be the free abelian group which is obtained from
\[
\bigoplus_{\substack{o_x \; \mbox{\small orientation of } \\ \mathrm{Det} (\Sigma_{\partial}(a_x,a_0)), \; \mbox{\small for }x\in \mathscr{P}_k(H)}} o_x,
\]
by taking the quotient by the identification
\[
\widehat{o}_x = - o_x, \qquad \forall x\in \mathscr{P}_k(H),
\]
where $\widehat{o}_x$ denotes the orientation of $\mathrm{Det} (\Sigma_{\partial}(a_x,a_0))$ which is opposite to $o_x$. A choice of an orientation of $\mathrm{Det} (\Sigma_{\partial}(a_x,a_0))$ for each $x\in \mathscr{P}_k(H)$ defines an isomorphism between $F_k(H)$ and
\[
\bigoplus_{x\in \mathscr{P}_k(H)} \Z.
\]
However, we prefer not to make such a choice and to consider ``oriented periodic orbits'' as generators of the Floer complex. 

Let $u\in \mathscr{M}_{\partial^F}(x,y)$ and denote by the same symbol the continuous extension $u: \overline{\R} \times \T \rightarrow T^*M$  which is obtained by setting $u(-\infty,t):=x(t)$ and $u(+\infty,t):=y(t)$. Since the convergence of $u$ to its asympltotic periodic orbits is exponential together with all the higher derivatives, this extension  is smooth with respect to the differentiable structure which was introduced at the beginning of Section \ref{linorsec}. Since the unitary trivializations $\Theta_x$ and $\Theta_y$ are vertical- and orientation-preserving, we can find a (not necessarily vertical-preserving) unitary trivialiazation 
\[
\Theta_u : \overline{\R} \times \T \times \C^n \rightarrow u^*(TT^*M), \qquad (s,t,\xi) \rightarrow \Theta_u (s,t)\xi,
\]
which agrees with the trivializations $\Theta_x$ and $\Theta_y$ for $s=-\infty$ and $s=+\infty$. This is proved in \cite[Lemma 1.7]{as06} and follows from the fact that the inclusion of the subgroup $\mathrm{SO}(n)$ of vertical- and orientation-preserving unitary automorphisms of $\C^n$ into $\mathrm{U}(n)$ induces the zero homomorphism between fundamental groups.

Let $x,y\in \mathscr{P}(H)$ have index difference $\mu_{CZ}(x)-\mu_{CZ}(y)=1$. In this case, the quotient $\mathscr{M}_{\partial^F}(x,y)/\R$ given by  the action of $\R$ by translations is discrete and compact, hence finite. An orientation $o_x$ of $\Det(\Sigma_{\partial}(a_x,a_0))$ and an element $[u]$ of $\mathscr{M}_{\partial^F}(x,y)/\R$ determine an orientation
\[
[u]_* (o_x)
\]
of $\Det(\Sigma_{\partial}(a_y,a_0))$ by the following procedure. Consider a unitary trivialization $\Theta_u$ of $u^*(TT^*M)$ as above. Thanks to $\Theta_u$, the fiberwise differential of the section $\mathscr{D}$ at $u$, that we denote by $D_f \mathscr{D}(u)$, is conjugated to a Fredholm operator $D_{A}$ in the space $\Sigma_{\partial}(a_x,a_y)$. This operator depends on the choice of the trivialization $\Theta_u$, but the space $\Sigma_{\partial}(a_x,a_y)$ to which it belongs does not. The operator $D_A$ is onto and has a one-dimensional kernel, which is spanned by the image of $\partial_s u$ by $\Theta_u^{-1}$. This kernel has the canonical orientation given by declaring the vector $\Theta_u^{-1} \partial_s u$ to be positive. Therefore, we get an orientation of 
\[
\Det(D_A) = \Lambda^{\max}(\Ker D_A) \otimes \R^*,
\]
which induces by continuation an orientation $o_{[u]}$ of $\Det(\Sigma_{\partial}(a_x,a_y))$. By Lemma \ref{conju1}, $o_{[u]}$ does not depend on the choice of the trivialization with prescribed asymptotics $\Theta_u$. Moreover, it clearly depends only on the equivalence class of $u$ in $\mathscr{M}_{\partial^F}(x,y)/\R$. This justifies the notation $o_{[u]}$. Now we define $[u]_* (o_x)$ to be the orientation of $\Det(\Sigma_{\partial}(a_y,a_0))$ such that the formula
\[
o_{[u]} \# [u]_*(o_x) = o_x
\]
holds.

Then we can define the boundary operator $\partial^F$ of the Floer complex generator-wise as
\[
\partial^F : F_k(H) \longrightarrow F_{k-1}(H), \qquad \partial^F o_x := \sum_{y\in \mathscr{P}_{k-1}(H)} \sum_{[u]\in \mathscr{M}_{\partial^F}(x,y)/\R} [u]_*(o_x).
\]
The well-posedness of this defintion uses the fact that 
\[
[u]_* (\widehat{o}_x) = \widehat{[u]_*(o_x)},
\]
which follows from the antisymmetry of the gluing of orientations. The fact that $\partial^F$ is indeed a boundary operator follows from a standard non-linear gluing construction.

The chain complex $\{F_*(H),\partial^F\}$ is called the {\em Floer complex of $H$}. A different choice of the almost complex structure $J$ or of the orientation data $\{\Theta_x\}_{x\in \mathscr{P}(H)}$, and $a_0\in \mathscr{A}$ produces an isomorphic chain complex, while a different choice of the Hamiltonian $H$ determines a chain isotopic complex. The homology of the Floer complex is thus independent on the pair $(H,J)$ and it is called the {\em Floer homology} of $T^*M$. 

\begin{rem}
The choice of using the orientations of the determinant bundle over $\Sigma_{\partial}(a_x,a_0)$ as generators of the Floer complex is motivated only by the fact that here we are always considering the cylinder (or half-cylinders) as domain of the linear Cauchy-Riemann-like operators we are looking at. By allowing more general Riemann surfaces one has other choices, that in other contexts might be more natural. For instance, one could consider linear operators on the open disc which in the variables $s=1/(1-\rho)$, $t=\varphi$ (where $(\rho,\varphi)$ are polar coordinates on the disc)  have the form $D_A$ near the boundary. One can then use the gluing of orientations in the more general framework of \cite{bm04} and obtain a definition of the Floer complex which is completely equivalent to the one we are using here. See \cite[Appendix 2]{rit10} for the latter approach in the context of symplectic homology and of topological quantum field theory.
\end{rem}

\begin{rem}
This way of defining the Floer complex with integer coefficients works whenever one can fix unitary trivializations of the tangent bundle along the periodic orbits in such a way that pairs of such trivializations along two orbits $x$ and $y$ can be extended to unitary trivializations over each cylinder joining $x$ and $y$. On cotangent bundles, vertical preserving trivializations have this property. More generally, such a choice is possible on symplectic manifolds whose first Chern class vanishes on tori. 
\end{rem}

\paragraph{The Morse complex of the Lagrangian action functional with local coefficients.} Here we want to define the Morse complex of the Lagrangian action functional on the loop space of $M$ with a particular system of local coefficients which takes the second Stiefel-Whitney class of the oriented manifold $M$ into account. See Appendix \ref{app} for the main definitions and general facts about systems of local coefficients, singular homology and Morse homology with local coefficients. 

Besides what we have already assumed above, let us assume that $H$ is {\em fiberwise uniformly convex}, meaning that
\begin{equation}
\label{uncon}
d_{pp} H(t,q,p)[\xi]^2 \geq h_0 |\xi|^2, \qquad  \forall (t,q,p)\in \T\times T^*M, \; \forall \xi \in T_q^*M,
\end{equation}
for some positive number $h_0$. Hamiltonians of the form (\ref{phys}) satisfy this condition. Under this assumption, the Legendre transform  $\mathscr{L}$ (see (\ref{legendre})) is a global diffeomorphism and the Fenchel formula
\begin{equation}
\label{fench}
L(t,q,v) := \max_{p\in T_q^* M} \Bigl( \langle p,v \rangle - H(t,q,p) \Bigr), \qquad \forall (t,q,v) \in \T\times TM,
\end{equation}
defines a smooth Lagrangian $L$ which satisfies the Legendre duality identity (\ref{legtra}), is quadratic at infinity  and fiberwise uniformly convex (as in Definition \ref{qai} and in (\ref{uncon}), but replacing the cotangent by the tangent bundle). A loop $q\in C^{\infty}(\T,M)$ is a critical point of the Lagrangian action functional $\mathbb{S}$ (see (\ref{laf})) if and only if it satisfies the Euler-Lagrange equations associated to $L$ with 1-periodic boundary conditions, and if and only if the loop $x = (q,d_v L(\cdot,q,q'))$ is an element of $\mathscr{P}(H)$. Furthermore, the fact that each $x\in \mathscr{P}(H)$ is non-degenerate is equivalent to the fact that each critical point $q$ of $\mathbb{S}$ is  non-degenerate. The Morse index $\ind (q;\mathbb{S})$ of $q$ coincides with the Conley-Zehnder index $\mu_{CZ}(x)$ of $x$.

The functional $\mathbb{S}$ has a continuously differentiable extension to the Hilbert manifold $H^1(\T,M)$ consisting of loops in $M$ of Sobolev class $H^1$. Such an extension is twice Gateaux differentiable but in general not twice Fr\'echet differentiable, unless the Lagrangian $L$ is exactly fiberwise quadratic, e.g.\ it comes from the Hamiltonian (\ref{phys}). Nevertheless, $\mathbb{S}$ always admits a smooth negative pseudo-gradient Morse vector field, that is a smooth vector field $G_{\mathbb{S}}$ on $H^1(\T,M)$ such that
\[
d\mathbb{S}(q) [G_{\mathbb{S}}(q)] < 0, \quad \forall q\in H^1(\T,M) \setminus \crit \, \mathbb{S},
\]
while each critical point $q$ of $\mathbb{S}$ is a hyperbolic singular point for $G_{\mathbb{S}}$, such that the second differential $d^2 \mathbb{S}(q)$ is positive-definite (resp.\ negative-definite) on the stable (resp.\ unstable) subspace of $T_q H^1(\T,M)$ (see \cite{as09b}). Each $q\in \crit\, \mathbb{S}$ has a stable manifold $W^s(q;G_{\mathbb{S}})$ and an unstable manifold $W^u(q;G_{\mathbb{S}})$ with respect to the flow of $G_{\mathbb{S}}$. These are smoothly embedded submanifolds of $H^1(\T,M)$ with
\[
\dim W^u(q;G_{\mathbb{S}}) = \ind (q;\mathbb{S}) = \codim W^s(q;G_{\mathbb{S}}).
\]
Up to a generic perturbation of $G_{\mathbb{S}}$, we may assume that $G_{\mathbb{S}}$ is Morse-Smale, meaning that the stable and unstable manifold of any two critical points meet transversally. The vector field $G_{\mathbb{S}}$ is positively complete and the pair $(\mathbb{S},G_{\mathbb{S}})$ satisfies the Palais-Smale condition, meaning that each sequence $(q_n)\subset H^1(\T,M)$ such that $\mathbb{S}(q_n)$ is bounded and $d\mathbb{S}(q_n)[G_{\mathbb{S}}(q_n)]$ is infinitesimal has a converging subsequence (see again \cite{as09b}). 

The Hilbert manifold $H^1(\T,M)$ on which the Lagrangian action functional $\mathbb{S}$ is defined is homotopically equivalent to $\Lambda M:= C^0(\T,M)$, the free loop space of $M$, and carries the system of local coefficients $\mathcal{G}$ which is induced by the following $\Z$-representation $\rho$ of the fundamental group of $\Lambda M$: every continuous loop $\gamma:\T \rightarrow \Lambda M$ defines a continuous map
\[
\tilde{\gamma} : \T^2 \rightarrow M, \qquad (s,t) \mapsto \gamma(s)(t),
\]
and the homomorphism
\[
\rho: \pi_1(\Lambda M) \rightarrow \mathrm{Aut}(\Z)
\]
is defined as
\[
\rho [\gamma] := \left\{ \begin{array}{cl} \mathrm{id} & \mbox{if } w_2(\tilde{\gamma}^* TM) = 0, \\ -\mathrm{id} & \mbox{if } w_2(\tilde{\gamma}^* TM) \neq 0. \end{array} \right.
\]
Here $w_2$ denotes the second Stiefel-Whitney class. Therefore, the local system of groups is trivial if and only if the second Stiefel-Whitney class of the oriented manifold $M$ vanishes on tori, and in particular if $M$ is spin. An example of a manifold which is not Spin but for which the latter condition holds is the non-trivial $S^2$-bundle over the orientable closed surface of genus 2 (we are grateful to B.\ Martelli for suggesting us this example).

More concretely, the system of local coefficients 
\[
\mathcal{G}= \bigl\{ \{G_q\}_{q\in \Lambda M}, \{h[\gamma]\}_{\gamma\in C^0([0,1],\Lambda M}) \bigr\} 
\]
is defined in the following way. Fix an element $\bar{q}$ in each connected component of $\Lambda M$ and, for every $q$ in the connected component of $\bar{q}$, a continuous path $\beta_q : [0,1] \rightarrow \Lambda M$ such that $\beta_q(0)=q$ and $\beta_q(1)=\bar{q}$. We also assume that $\beta_{\bar{q}}$ is the constant path at $\bar{q}$. Fix an orientation-preserving trivialization $\theta_{\bar{q}}$ of $\bar{q}^* (TM)$ for each privileged loop $\bar{q}$ and use parallel transport along the path $\beta_q$ to obtain a trivialization $\theta_q$ of $q^* (TM)$ for each $q\in \Lambda M$. Given a continuous path $\gamma: [0,1] \rightarrow \Lambda M$ such that $\gamma(0)=q_0$ and $\gamma(1)=q_1$, we define the isomorphism
\[
h[\gamma] : G_{q_1} := \Z \longrightarrow \Z =: G_{q_0}
\]
in the following way: if we compare the trivialization $\theta_{q_0}$ of $q_0^* (TM)$ with the trivialization of $q_0^* (TM)$ which is obtained by transporting  the trivialization $\theta_{q_1}$ along $\gamma$, we get a closed loop in $\mathrm{GL}^+(n,\R) \sim \mathrm{SO}(n)$, and we define $h[\gamma]$ to be the identity if this loop has a closed lift into the two-fold covering $\mathrm{Spin}(n) \rightarrow \mathrm{SO}(n)$, minus the identity otherwise. In order to simplify the notation, we set
\begin{equation}
\label{signum}
\sgn(\gamma) := \left\{ \begin{array}{cl} 1 & \mbox{if } h[\gamma] = \mathrm{id}, \\ -1 & \mbox{if } h[\gamma] = -\mathrm{id}. \end{array} \right.
\end{equation}

Denote by $\crit_k(\mathbb{S})$ the set of critical points of $\mathbb{S}$ of Morse index $k$, and let $M_k(\mathbb{S};\mathcal{G})$ be the free abelian group which is obtained from group
\[ 
\bigoplus_{\substack{o_q \; \mathrm{orientation} \\ \mathrm{of}\; W^u(q;G_{\mathbb{S}}), \; \mathrm{for} \; q\in \crit_k(\mathbb{S})}} o_q 
\] 
by taking the quotient by the identification
\[
\widehat{o}_q = - o_q, \qquad \forall q\in \mathrm{crit}_k(\mathbb{S}),
\]
where $\widehat{o}_q$ denotes the orientation of $W^u(q;G_{\mathbb{S}})$ which is opposite to $o_q$. 
Let $q^-,q^+$ be critical points of $\mathbb{S}$ with $\mathrm{ind}(q^-) - \mathrm{ind}(q^+)=1$. Then $W^u(q^-;G_{\mathbb{S}})\cap W^s(q^+;G_{\mathbb{S}})$ is a one-dimensional manifold and consists of finitely many curves joining $q^-$ to $q^+$: denote by
\[
\mathscr{M}_{\partial^M}(q^-,q^+) \subset C^0([0,1],\Lambda M)
\]
the finite set consisting of a choice of a parametrization for each of these curves (where $\gamma\in \mathscr{M}_{\partial^M}(q^-,q^+)$ implies that $\gamma(0)=q^-$ and $\gamma(1)=q^+$). 

An orientation $o_{q^-}$ of $W^u(q^-;G_{\mathbb{S}})$ and a path $\gamma\in \mathscr{M}_{\partial^M}(q^-,q^+)$ induce an orientation $\gamma_*(o_{q^-})$ of $W^u(q^+;G_{\mathbb{S}})$. Indeed, since
\[
T_{q^+} H^1(\T,M) = T_{q^+} W^u(q^+) \oplus T_{q^+} W^s(q^+),
\]
an orientation of $W^u(q^+;G_{\mathbb{S}})$ can be identified with a co-orientation of $W^s(q^+;G_{\mathbb{S}})$ (that is, an orientation of the normal bundle of $W^s(q^+;G_{\mathbb{S}})$ in $H^1(\T,M)$) and $\gamma_*(o_{q^-})$ is defined as the co-orientation of $W^s(q^-;G_{\mathbb{S}})$ which, together with the orientation $o_{q^-}$ of $W^u(q^-;G_{\mathbb{S}})$, determines the orientation of the intersection $W^u(q^-;G_{\mathbb{S}})\cap W^s(q^+;G_{\mathbb{S}})$ which at $\gamma((0,1))$ agrees with the direction of the flow of $G_{\mathbb{S}}$.

The homomorphism
\[
\partial_k^M : M_k(\mathbb{S};\mathcal{G}) \rightarrow M_{k-1}(\mathbb{S};\mathcal{G})
\]
is defined generator-wise as
\[
\partial_k^M o_{q^-} = \sum_{q^+\in \mathrm{crit}_{k-1}(\mathbb{S})} \sum_{\gamma\in \mathscr{M}_{\partial^M}(q^-,q^+)} \sgn (\gamma) \, \gamma_*(o_{q^-}) .
\]
It satisfies $\partial_k^M \circ \partial_{k+1}^M=0$, so $\{ M_*(\mathbb{S};\mathcal{G}),\partial_*^M\}$ is a chain complex of abelian groups, which is called the {\em Morse complex of $(\mathbb{S},G_{\mathbb{S}})$ with local coefficients in $\mathcal{G}$}. The choice of a different pseudo-gradient vector field $G_{\mathbb{S}}$ for $\mathbb{S}$ produces an isomorphic chain complex. In particular, the homology
\[
HM_*(\mathbb{S};\mathcal{G}) := H_* \bigl( M_*(\mathbb{S};\mathcal{G}),\partial_*^M \bigr)
\]
does not depend on $G_{\mathbb{S}}$ and is called the {\em Morse homology of $\mathbb{S}$  with local coefficients in $\mathcal{G}$}.

This homology is isomorphic to the singular homology of $H^1(\T,M)\simeq \Lambda M$ with coefficients in $\mathcal{G}$:
\[
HM_*(\mathbb{S};\mathcal{G}) \cong H_*(\Lambda M;\mathcal{G}).
\]
When the second Stiefel-Whitney class of $M$ vanishes on tori, $\mathcal{G}$ is the trivial system of integer coefficients, and the above Morse homology and singular homology are the usual ones. 

\section{The chain isomorphism $\Psi$}
\label{defisosec}

\paragraph{The chain isomorphism $\Psi$.} We can now define the chain isomorphism
\[
\Psi : F_*(H) \rightarrow M_*(\mathbb{S};\mathcal{G})
\]
from the Floer complex of the fiberwise uniformly convex and quadratic at infinity Hamiltonian $H$ to the Morse complex of the action functional $\mathbb{S}$, which is given by the Legendre-dual Lagrangian $L$, with coefficients in the local system $\mathcal{G}$. We start by noticing that, since $L$ is the Legendre-dual of $H$, the boundary condition (iv) on the elements of the reduced unstable manifold of $x\in \mathscr{P}(H)$ implies that if $u\in \mathscr{U}(x)$ and $u(0,t)=(q(t),p(t))$ then by (\ref{legtra})
\begin{equation}
\label{fide}
\begin{split}
\mathbb{A}(u(0,\cdot)) = \int_{\T} \Bigl( \langle p(t), q'(t) \rangle - H(t,q(t),p(t)) \Bigr)\, dt \\= \int_{\T} L(t,q(t),q'(t)) \, dt = \mathbb{S}(q) = \mathbb{S}(\pi\circ u(0,\cdot)). \end{split}
\end{equation}

Let $J\in \mathscr{J}(g)$ be close enough to the Levi-Civita almost complex structure $J_g$ and generic, so that transversality holds both for the spaces of full cylinders which define the Floer complex and for the reduced unstable manifolds $\mathscr{U}(x)$, for every $x\in \mathscr{P}(H)$ (see Lemma \ref{trans} (i)).

In order to make the definition of $\Psi$ simpler, it is useful to choose the fixed vertical- and orientation-preserving unitary trivializations $\Theta_x$ of $x^*(TT^*M)$, for $x\in \mathscr{P}(H)$, to be compatible with the fixed trivializations $\theta_q$ of $q^*(TM)$, for $q\in \Lambda M$: if $x\in \mathscr{P}(H)$ and $q:= \pi\circ x$, we choose the trivialization $\Theta_x$ in such a way that its restriction
\[
\Theta_x|_{\T\times \R^n} : \T \times \R^n \rightarrow u^*((T^v T^*M)^{\perp}) \cong q^*(TM)
\]
is homotopic to $\theta_q$. This choice has the following consequence. Given $u\in \mathscr{M}_{\partial^F}(x,y)$, $x,y\in \mathscr{P}(H)$, let $\gamma_u \in C^0([0,1],\Lambda M)$ be a reparametrization on $[0,1]$ of the path $s\mapsto \pi\circ u(s,\cdot)$ such that $\gamma_u(0)=\pi\circ x$ and $\gamma_u(1)=\pi\circ y$.  Let $\Theta_u$ be a vertical preserving unitary trivialization of $u^*(TT^*M)$ over $\overline{\R}\times \T$, which agrees with $\Theta_x$ on $\{-\infty\}\times \T$. Then the restriction of the trivialization $\Theta_u$ to $\{+\infty\} \times \T$ can be compared with the trivialization $\Theta_y$, producing a loop which takes values in $\mathrm{SO}(n)$,
\[
\Theta_y^{-1} \circ \Theta_u (+\infty,\cdot) : \T \rightarrow \mathrm{SO}(n),
\]
because both $ \Theta_u$ and $\Theta_y$ are vertical- and orientation-preserving. From the definition (\ref{signum}) of $\sgn(\gamma)$ and from the compatibility between the trivializations $\{\Theta_x\}_{x\in \mathscr{P}(H)}$ and $\{\theta_q\}_{q\in C^0([0,1],\Lambda M)}$ we immediately deduce the following:

\begin{lem}
\label{twist}
The loop $\Theta_y^{-1} \circ \Theta_u (+\infty,\cdot)$ lifts to a closed loop in $\mathrm{Spin}(n)$ if and only if $\sgn(\gamma_u)=1$.  
\end{lem}

Let us show how an orientation $o_x$ of $\Det(\Sigma_{\partial}(a_x,a_0))$ induces an orientation of the reduced unstable manifold $\mathscr{U}(x)$. Let us fix once and for all an orientation of $\Det(\Sigma_{\mathscr{U}}(a_0))$. An orientation $o^{\mathscr{U}}_x$ of $\Det(\Sigma_{\mathscr{U}}(a_x))$ can be defined by requiring that
$o_x \# o_x^{\mathscr{U}}$
agrees with the fixed orientation of $\Det(\Sigma_{\mathscr{U}}(a_0))$. 
Given $u\in \mathscr{U}(x)$, extend the trivialization $\Theta_x$ to a vertical-preserving unitary trivialization $\Theta_u$ of $u^*(TT^*M)$. Conjugation by $\Theta_u$ identifies the fiber-wise differential $D_f \mathscr{T}(u)$ of the section $\mathscr{T}$ whose zeros are the elements of $\mathscr{U}(x)$ (see Section \ref{secrum}) at $u$ with an operator $T_{A,\alpha}$ in $\Sigma_{\mathscr{U}}(a_x)$. Therefore, $o_x^{\mathscr{U}}$ induces an orientation of the line 
\[
\Det(D_f \mathscr{T}(u))= \Lambda^{\max} (\ker D_f \mathscr{T}(u)) \otimes \R^*= \Lambda^{\max} (T_u \mathscr{U}(x)) \otimes \R^*.
\]
Therefore, we get an orientation of $T_u \mathscr{U}(x)$.
If we choose a different vertical-preserving extension $\Theta_u'$ of the trivialization $\Theta_x$, we easily see that $\Theta_u^{-1} \Theta_u' (0,\cdot)$ is a contractible loop in $\mathrm{SO}(n)$, so Lemma \ref{conju2} ensures that the trivialization $\Theta_u'$ induces the same orientation of $T_u \mathscr{U}(x)$ (actually, here a more special and much easier to prove version of Lemma \ref{conju2} would suffice). Since the construction depends continuously on $u$, we obtain an orientation of the finite dimensional manifold $\mathscr{U}(x)$, as claimed. Conversely, an orientation of $\mathscr{U}(x)$ determines an orientation of $\Det(\Sigma_{\partial}(a_x,a_0))$. Therefore, in the case of a fiber-wise uniformly convex Hamiltonian the generators of the Floer complex are the oriented  
reduced unstable manifolds, exactly as in Morse theory.

Given $x\in \mathscr{P}(H)$, we consider the smooth map
\[
\mathscr{Q}_x: \mathscr{U}(x) \rightarrow H^1(\T,M),\qquad \mathscr{Q}_x(u) = \pi\circ u(0,\cdot).
\]
We can choose the Morse-Smale negative pseudo-gradient vector field $G_{\mathbb{S}}$ for $\mathbb{S}$ in such a way that for every pair $(x,q) \in \mathscr{P}(H) \times \crit \, \mathbb{S}$ the map $\mathscr{Q}_x$ is transverse to the submanifold $W^s(q;G_{\mathbb{S}})$. Notice that, since 
\[
d^2 \mathbb{S}(q)[\xi]^2 \geq 0, \qquad \forall q\in \crit\, \mathbb{S}, \; \forall \xi\in T_q W^s(q;G_{\mathbb{S}}),
\]
Lemma \ref{trans} (ii) implies that if $x\in \mathscr{P}(H)$ is the Hamiltonian orbit which corresponds to $q$, then $\mathscr{Q}_x$ is automatically transverse to $W^s(q;G_{\mathbb{S}})$ at the stationary element $u_0(s,t)=x(t)$. In fact, by this lemma the vector space 
\[
D\mathscr{Q}_x(u_0) T_{u_0} \mathscr{U}(x)
\]
has dimension $\mu_{CZ}(x) = \mathrm{ind}\, (q;\mathbb{S})$ and is contained in the negative cone of $d^2 \mathbb{S}(q)$, hence
\begin{equation}
\label{tra}
T_q H^1(\T,M) = D\mathscr{Q}_x(u_0) T_{u_0} \mathscr{U}(x) \oplus T_q W^s(q;G_{\mathbb{S}}).
\end{equation}

Let $x\in \mathscr{P}(H)$, $q\in \crit \, \mathbb{S}$, and denote by
\[
\mathscr{M}_{\Psi}(x,q):= \mathscr{Q}^{-1}_x\bigl(W^s(q;G_{\mathbb{S}})\bigr)
\]
 the set of elements $u$ in the reduced unstable manifold $\mathscr{U}(x)$ such that the loop $\pi\circ x$ belongs to the stable manifold $W^s(q;G_{\mathbb{S}})$. Since $\mathscr{U}(x)$ has dimension $\mu_{CZ}(x)$ and $W^s(q;G_{\mathbb{S}})$ has codimension $\ind(q;\mathbb{S})$, the above transversality requirements imply that $\mathscr{M}_{\Psi}(x,q)$ is empty whenever $\mu_{CZ}(x)< \ind (q;\mathbb{S})$ and it is 
 a (possibly empty) submanifold of $\mathscr{U}(x)$ of dimension
 \[
 \dim \mathscr{M}_{\Psi}(x,q) = \mu_{CZ}(x) - \ind (q;\mathbb{S}),
 \] 
 when $\mu_{CZ}(x) \geq \ind (q;\mathbb{S})$. Moreover, an orientation $o_x$ of $\Det(\Sigma_{\partial}(a_x,a_0))$ and an orientation $o_q$ of the unstable manifold $W^u(q;G_{\mathbb{S}})$ determine an orientation of $\mathscr{M}_{\Psi}(x,q)$: indeed, $o_x$ determines an orientation of $\mathscr{U}(x)$ - as we have seen above - and $o_q$ determines a co-orientation of $W^s(q;G_{\mathbb{S}})$, so the claim just follows from the fact that the transverse inverse image of a co-oriented submanifold by a map whose domain is oriented inherits a canonical orientation.
 
If $u$ belongs to $\mathscr{M}_{\Psi}(x,q)$, then $\pi\circ x$ and $q$ belong to the same free homotopy class and the inequality
\[
\label{actineq}
\mathbb{A}(x) \geq \mathbb{A}(u(0,\cdot)) = \mathbb{S}(\pi\circ u(0,\cdot)) \geq \mathbb{S}(q),
\]
holds, where the middle equality follows from (\ref{fide}). Morever, the first inequality is an equality if and only if $u$ is the stationary solution $u_0(s,t)=x(t)$, while the second inequality is an equality if and only if $\pi\circ u(0,\cdot)=q$.
We deduce that
\begin{eqnarray}
\label{cai1}
&\mathscr{M}_{\Psi}(x,q) = \emptyset \quad & \mbox{if } \mathbb{S}(q)> \mathbb{A}(x) \mbox{, or } \mathbb{S}(q)= \mathbb{A}(x) \mbox{ and } q\neq \pi\circ x, \\
\label{cai2}
&\mathscr{M}_{\Psi}(x,q) = \{u_0\} \quad & \mbox{if } q= \pi\circ x.
\end{eqnarray}
 
Moreover, Corollary \ref{komp} implies that $\mathscr{M}_{\Psi}(x,q)$ is relatively compact in the $C^{\infty}_{\mathrm{loc}}(\R^- \times \T)$ topology. In the particular case $\mu_{CZ}(x)=\ind (q;\mathbb{S})$, a standard argument involving compactness and transversality implies that $\mathscr{M}_{\Psi}(x,q)$ consists of finitely many elements. The choice of orientations $o_x$ and $o_q$ of $\Det(\Sigma_{\partial}(a_x,a_0))$ and $W^u(q;G_{\mathbb{S}})$ determine an orientation of $\mathscr{M}_{\Psi}(x,q)$ - as we have seen above - which in this case is nothing else by the choice of a sign $\pm 1$ for every $u\in \mathscr{M}_{\Psi}(x,q)$, and for such an $u$ we define
\[
u_*(o_x)
\]
to be the orientation of $W^u(q;G_{\mathbb{S}})$ for which this sign is $+1$. 

Given $u\in \mathscr{M}_{\Psi}(x,q)$ we define $\gamma_u : [0,1] \rightarrow \Lambda M$ to be a path joining $\pi\circ x$ to $q$, such that $\gamma_u|_{[0,1/2]}$ is a reparametrization on $[0,1/2]$ of the path $s\mapsto \pi\circ u(s,\cdot)$, $s\in ]-\infty,0]$, and $\gamma_u|_{[1/2,1]}$ is a reparametrization on $[1/2,1]$ of the path $s\mapsto \phi_s^{G_{\mathbb{S}}}(\pi\circ u(0,\cdot))$, $s\in [0,+\infty[$, where $\phi_s^{G_{\mathbb{S}}}$ denotes the flow of the vector field $G_{\mathbb{S}}$. The homomorphism $\Psi$ is defined generator-wise by
\begin{equation}
\label{defpsi}
\Psi: F_*(H) \rightarrow M_*(\mathbb{S};\mathcal{G}), \qquad \Psi o_x := \sum_{\substack{q\in \crit\, \mathbb{S} \\ \ind(q;\mathbb{S}) = \mu_{CZ}(x)}} \sum_{u\in \mathscr{M}_{\Psi}(x,q)} \sgn(\gamma_u) u_*(o_x).
\end{equation}
The above sum ranges over a finite set because of (\ref{cai1}).

\begin{lem}
$\Psi$ is a chain map.
\end{lem}

\begin{proof}
The chain map identity $\Psi\circ \partial^F = \partial^M \circ \Psi$ is implied by the following fact: for every $x_0\in \mathscr{P}(H)$ and $q_0\in \crit\, \mathbb{S}$ with $\ind(q_0) = \mu_{CZ}(x_0) -1$, there exists an involution on the set
\begin{equation}
\label{invol}
\Bigl( \bigcup_{\substack{x\in \mathscr{P}(H)\\ \mu_{CZ}(x) = \mu_{CZ}(x_0)-1}} \mathscr{M}_{\partial^F} (x_0,x) \times \mathscr{M}_{\Psi}(x,q_0) \Bigr) \cup \Bigl( \bigcup_{\substack{q\in \crit\, \mathbb{S}\\ \ind (q) = \mu_{CZ}(x_0)}} \mathscr{M}_{\Psi} (x_0,q) \times \mathscr{M}_{\partial^M}(q,q_0) \Bigr)
\end{equation}
which has no fixed points and satisfies, for $o_{x_0}$ an orientation of $\Det(\Sigma_{\partial}(a_{x_0},a_0))$:
\begin{enumerate}

\item if $(u_1,v_1)\in \mathscr{M}_{\partial^F}(x_0,x_1) \times \mathscr{M}_{\Psi} (x_1,q_0)$ is mapped into $(u_2,v_2)\in \mathscr{M}_{\partial^F}(x_0,x_2) \times \mathscr{M}_{\Psi} (x_2,q_0)$, then
\[
\sgn(\gamma_{v_1}) {v_1}_* ( [u_1]_*(o_{x_0})) = - \sgn(\gamma_{v_2}) {v_2}_* ( [u_2]_*(o_{x_0}));
\]

\item if $(u,v_1)\in \mathscr{M}_{\partial^F}(x_0,x) \times \mathscr{M}_{\Psi} (x,q_0)$ is mapped into $(v_2,\gamma)\in \mathscr{M}_{\Psi}(x_0,q) \times \mathscr{M}_{\partial^M} (q,q_0)$, then
\[
\sgn(\gamma_{v_1}) {v_1}_* ( [u]_*(o_{x_0})) = \sgn(\gamma_{v_2}) \sgn (\gamma) \gamma_* ( {v_2}_*(o_{x_0}));
\]

\item if $(v_1,\gamma_1)\in \mathscr{M}_{\Psi}(x_0,q_1) \times \mathscr{M}_{\partial^M} (q_1,q_0)$ is mapped into $(v_2,\gamma_2)\in \mathscr{M}_{\Psi}(x_0,q_2) \times \mathscr{M}_{\partial^M} (q_2,q_0)$, then
\[
\sgn(\gamma_{v_1}) \sgn(\gamma_1) {\gamma_1}_* ( {v_1}_*(o_{x_0})) = - \sgn(\gamma_{v_2}) \sgn (\gamma_2) {\gamma_2}_* ( {v_2}_*(o_{x_0})).
\]

\end{enumerate}

The fixed-point-free involution is given by a standard cobordism argument: the space $\mathscr{M}_{\Psi}(x_0,q_0)$ is a one-dimensional manifold, and each of its non-compact connected component is an open curve $W$, which has two distinct elements of (\ref{invol}) as boundary points; moreover, each element of (\ref{invol}) is a boundary point of exactly one connected component $W$ of $\mathscr{M}_{\Psi}(x_0,q_0)$ (see e.g.\ \cite{flo89b}). There remains to prove (i), (ii) and (iii).

Consider a connected component $W$ of $\mathscr{M}_{\Psi}(x_0,q_0)$ whose boundary points are $(u_1,v_1)$ and $(u_2,v_2)$ as in (i). First assume that $u_1^*(TT^*M)$ and $u_2^*(TT^*M)$ admit vertical-preserving unitary trivializations $\Theta_{u_1}$ and $\Theta_{u_2}$ which agree with $\Theta_{x_0}$ on $\{-\infty\} \times \T$ and with $\Theta_{x_1}$ and $\Theta_{x_2}$ on $\{+\infty\} \times \T$. In particular, by Lemma \ref{twist},
\[
\sgn(\gamma_{u_1}) = \sgn(\gamma_{u_2}) = 1.
\]
Let $\Theta_{v_1}$ and $\Theta_{v_2}$ be vertical-preserving unitary trivializations of $v_1^*(TT^*M)$ and $v_2(TT^*M)$ which agree with $\Theta_{x_1}$ and $\Theta_{x_2}$ at $\{-\infty\}\times \T$. Then $\Theta_{u_j}$ and $\Theta_{v_j}$ can be glued together and then slightly perturbed in order to produce a vertical-preserving trivialization $\Theta_{w_j}$ of $w^*_j(TT^*M)$ which agrees with $\Theta_{x_0}$ on $\{-\infty\} \times \T$, where $w_j$ is an element of $W$ close to the boundary point $(u_j,v_j)$. The trivializations $\Theta_{w_j}$ are admissible for the definition of the orientation of $W$. Then a standard argument using gluing of orientations (see e.g.\ \cite[Section 5]{fh93} and\cite{flo89b}) implies that 
\[
{v_1}_* ( [u_1]_*(o_{x_0})) = - {v_2}_* ( [u_2]_*(o_{x_0})).
\]
Consider now the general case, in which it might be impossible to choose $\Theta_{u_1}$ and/or $\Theta_{u_2}$ to be vertical-preserving and with the prescribed asymptotics. Changing the trivialization which is obtained by gluing $\Theta_{u_j}$ and $\Theta_{v_j}$ into a vertical-preserving one involves multiplication by a map $U:(-\infty,0]\times \T \rightarrow \mathrm{U}(n)$ such that $U(0,t)\in \mathrm{SO}(n)$ and $U(-\infty,t)=I$, for every $t\in \T$. By Lemmata \ref{conju2} and \ref{twist}, the above formula must be modified as follows
\begin{equation}
\label{in}
\sgn(\gamma_{u_1}) {v_1}_* ( [u_1]_*(o_{x_0})) = - \sgn(\gamma_{u_2}) {v_2}_* ( [u_2]_*(o_{x_0})).
\end{equation}
Since the paths $\gamma_{u_1} * \gamma_{v_1}$ and $\gamma_{u_2} * \gamma_{v_2}$ are homotopic with fixed end-points, there holds
\[
\sgn(\gamma_{u_1}) \sgn(\gamma_{v_1}) = \sgn(\gamma_{u_2}) \sgn(\gamma_{v_2}),
\]
and the identity (\ref{in}) is equivalent to the identity of Claim (i).

Consider now the case of a connected component $W$ of $\mathscr{M}_{\Psi}(x_0,q_0)$ whose boundary points are $(u,v_1)$ and $(v_2,\gamma)$ as in (ii). Arguing as above, if $u^*(TT^*M)$ admits a vertical-preserving unitary trivialization  which agrees with $\Theta_{x_0}$ on $\{-\infty\} \times \T$ and with $\Theta_{x}$ on $\{+\infty\} \times \T$, then 
\[
 {v_1}_* ( [u]_*(o_{x_0})) = \gamma_* ( {v_2}_*(o_{x_0})).
\]
In general, Lemmata \ref{conju2} and \ref{twist} imply that
\[
\sgn(\gamma_{u_1}) {v_1}_* ( [u]_*(o_{x_0})) = \gamma_* ( {v_2}_*(o_{x_0})).
\]
The identity of Claim (ii) now follows from the fact that the paths $\gamma_{u} * \gamma_{v_1}$ and $\gamma_{v_2} * \gamma$ are homotopic with fixed end-points, which implies that
\[
\sgn(\gamma_{u}) \sgn(\gamma_{v_1}) = \sgn(\gamma_{v_2}) \sgn(\gamma).
\]

Finally, consider a connected component $W$ of $\mathscr{M}_{\Psi}(x_0,q_0)$ whose boundary points are $(v_1,\gamma_1)$ and $(v_2,\gamma_2)$ as in (iii). In this case, we always have
\[
{\gamma_1}_* ( {v_1}_*(o_{x_0})) = -  {\gamma_2}_* ( {v_2}_*(o_{x_0})).
\]
The above identity is equivalent to the one of Claim (iii), because of the identity 
\[
\sgn(\gamma_{v_1}) \sgn(\gamma_1) = \sgn(\gamma_{v_2}) \sgn(\gamma_2),
\]
which follows from the fact that the paths $\gamma_{v_1} * \gamma_1$ and $\gamma_{v_2} * \gamma_2$ are homotopic with fixed end-points.
\end{proof}

By (\ref{cai1}) and (\ref{cai2}), the chain map $\Psi$ has the form
\[
\Psi o_x = {u_x}_*(o_x) + \sum_{\substack{q\in \mathrm{crit}_k\, \mathbb{S} \\ \mathbb{S}(q)< \mathbb{A}(x)}} n_{\Psi}(o_x,o_q)\, o_q, \qquad \forall x\in \mathscr{P}_k(H),
\]
where $u_x\in \mathscr{U}(x)$ denotes the stationary solution $u_x(s,t)=x(t)$, $o_q$ denotes an arbitrary orientation of $W^u(q;G_{\mathbb{S}})$, and $n_{\Psi}(o_x,o_q)$ is an integer. Therefore, $\Psi$ is an isomorphism, since it differs from the diagonal isomorphism 
\[
F_*(H) \rightarrow M_*(\mathbb{S},\mathcal{G}), \qquad o_x \mapsto {u_x}_*( o_x), 
\qquad \forall x\in \mathscr{P}(H),
\]
(which needs not be a chain map) by a homorphism which is strictly triangular with respect to the filtrations 
\begin{equation}
\label{actfil}
\bigl\{ F_*^{\leq A}(H) \bigr\}_{A\in \R} \quad  \mbox{and} \quad \bigl\{ M_*^{\leq A}(\mathbb{S},\mathcal{G})\bigr\}_{A\in \R},
\end{equation}
which are obtained by considering the subgroups generated by all the elements $x\in \mathscr{P}(H)$ or $q\in \crit \, \mathbb{S}$ whose action does not exceed the real number $A$. We summarize what we have proven so far into the following statement, which is part (i) of the theorem in the introduction:

\begin{thm}
Assume that $L$ is the Lagrangian which is Legendre-dual to the fiberwise uniformly convex and quadratic at infinity Hamiltonian $H\in C^{\infty}(\T\times \T^* M)$, all of whose 1-periodic orbits are non-degenerate. Then the formula (\ref{defpsi}) defines a chain isomorphism $\Psi$ from the Floer complex of $H$ to the Morse complex of the Lagrangian action functional $\mathbb{S}$ associated to the Legendre-dual Lagrangian $L\in C^{\infty}(\T\times TM)$ with coefficients in the local system $\mathcal{G}$. Such an isomorphism preserves the action filtrations (\ref{actfil}) and the splittings of $F_*(H)$ and $M_*(\mathbb{S},\mathcal{G})$ determined by the free homotopy classes of the generators.
\end{thm}

\begin{rem}
\label{twistrem}
As mentioned in the introduction, it is also possible to define a twisted Floer complex which is isomorphic to $M_*(\mathbb{S})$, the standard Morse complex of $\mathbb{S}$. In particular, the homology of the former is isomorphic to the standard singular homology of the free loop space of $M$, also when $M$ is not spin. The boundary operator of the twisted Floer complex takes the form
\[
\widehat{\partial}^F : F_k(H) \longrightarrow F_{k-1}(H), \qquad \widehat{\partial}^F o_x := \sum_{y\in \mathscr{P}_{k-1}(H)} \sum_{[u]\in \mathscr{M}_{\partial^F}(x,y)/\R} \sgn(\gamma_u) [u]_*(o_x),
\]
while the chain isomorphism is
\[
\widehat{\Psi}: F_*(H) \rightarrow M_*(\mathbb{S}), \qquad \widehat{\Psi} o_x := \sum_{\substack{q\in \crit\, \mathbb{S} \\ \ind(q;\mathbb{S}) = \mu_{CZ}(x)}} \sum_{u\in \mathscr{M}_{\Psi}(x,q)} u_*(o_x).
\]
\end{rem}

\paragraph{The chain map $\Psi$ is a homotopy inverse of $\Phi$.} Let us recall the definition of the chain isomorphism
\[
\Phi : M_*(\mathbb{S};\mathcal{G}) \rightarrow F_*(H)
\]
from \cite{as06} (suitably modified, in order to take the local system $\mathcal{G}$ into account). Set $\R^+:= [0,+\infty[$. Given $q\in \crit\, \mathbb{S}$ and $x\in \mathscr{P}(H)$, we denote by $\mathscr{M}_{\Phi}(q,x)$ the space of solutions $u: \R^+ \times \T \rightarrow T^*M$ of (\ref{cr}) such that $u(s,t)$ converges to $x(t)$ for $s\rightarrow +\infty$, uniformly in $t\in \T$, and such that
\[
\pi\circ u(0,\cdot) \in W^u(q;G_{\mathbb{S}}).
\]
For a generic choice of $J$ and $G_{\mathbb{S}}$, this is a smooth manifold of dimension $\ind(q;\mathbb{S})-\mu_{CZ}(x)$. Moreover, the elements $u$ of $\mathscr{M}_{\Phi}(q,x)$ satisfy the action estimates
\begin{equation}
\label{act2}
\mathbb{S}(q) \geq \mathbb{S}\bigl(\pi\circ u(0,\cdot)\bigr) \geq \mathbb{A}(u(0,\cdot)) \geq \mathbb{A}(x),
\end{equation}
where the middle inequality follows from the Fenchel formula (\ref{fench}) (see \cite[Lemma 2.3]{as06}). The above estimate implies the energy bounds which lead to compactness for $\mathscr{M}_{\Phi}(x,q)$. 

When $\ind(q;\mathbb{S})=\mu_{CZ}(x)$, $\mathscr{M}_{\Phi}(x,q)$ is a finite set. 
If $o_q$ and $o_x$ are orientations of $W^u(q;G_{\mathbb{S}})$ and $\Det(\Sigma_{\partial}(a_x,a_0))$, respectively, the sign $\epsilon^+(u)$  is defined as in \cite[Section 3]{as13}. We define $u_*(o_q)$ to be the orientation of $\Det(\Sigma_{\partial}(a_x,a_0))$ for which $\epsilon^+(u)=1$.

The chain map $\Phi$ is defined by the formula
\[
\Phi o_q := \sum_{\substack{x\in \mathscr{P}(H)\\ \mu_{CZ}(x) = \ind (x;\mathbb{S})}} 
\sum_{u\in \mathscr{M}_{\Phi}(q,x)} \sgn(\gamma_u) u_*(o_q), \qquad \forall q\in \crit\, \mathbb{S},
\]
where $\gamma_u : [0,1] \rightarrow \Lambda M$ is a path connecting $q$ to $\pi\circ x$ which reparametrizes the path $\phi_s^{G_{\mathbb{S}}}(\pi\circ u(0,\cdot))$, $s\in ]-\infty,0]$, on $[0,1/2]$ and the path $\pi\circ u$ on $[1/2,1]$.
By (\ref{act2}) and by automatic transversality at the stationary solutions, $\Phi$ is an isomorphism and it preserves the action filtrations (\ref{actfil}). It also preserves the splitting of the Morse and the Floer complexes determined by the free homotopy classes of the generators.

The next result is precisely statement (ii) of the theorem in the introduction.

\begin{thm}
\label{homeq}
The chain isomorphisms $\Phi$ and $\Psi$ are homotopy inverses one of the other through chain homoopies which preserve the action filtrations (\ref{actfil}) and the splitting of the Floer and the Morse complexes determined by the free homotopy classes of the generators.
\end{thm}

\begin{proof}
Since $\Phi$ and $\Psi$ are isomorphisms, it is enough to show that $\Psi\Phi$ is chain homotopic to the identity on $M_*(\mathbb{S};\mathcal{G})$ through a chain homotopy with the above properties, i.e.\ that there exists a homomorphism
\[
P : M_*(\mathbb{S};\mathcal{G}) \rightarrow M_{*+1}(\mathbb{S};\mathcal{G})
\]
such that
\begin{equation}
\label{he}
\Psi\Phi - I = P \partial^M + \partial^M P,
\end{equation}
and for every $q^-\in \crit \, \mathbb{S}$, $P o_{q^-}$ is a linear combination of generators $o_{q^+}$, for $q^+\in \crit\, \mathbb{S}$ with $\mathbb{S}(q^+)\leq \mathbb{S}(q^-)$ and $[q^+]=[q^-]$ in $[\T,M]$. 
The definition of $P$ and the proof of (\ref{he}) is based on a standard argument, which we now sketch.
Given $q^-,q^+\in \crit\, \mathbb{S}$, let $\mathscr{M}_P(q^-,q^+)$ be the set of pairs $(\sigma,u)$ where $\sigma$ is a positive number and $u:[0,\sigma]\times \T\rightarrow T^*M$ solves (\ref{cr}) together with the boundary conditions
\begin{eqnarray}
\label{bdry-}
& \pi\circ u(0,\cdot) \in W^u(q^-;G_{\mathbb{S}}), \\
\label{bdry+}
& \pi\circ u(\sigma,\cdot) \in W^s(q^+;G_{\mathbb{S}}), \qquad \partial_t (\pi\circ u)(\sigma,t) =d_p H(t,u(\sigma,t)), \quad \forall t\in \T.
\end{eqnarray}
For generic $J$ and $G_{\mathbb{S}}$, $\mathscr{M}_P(q^-,q^+)$ is a smooth manifold of dimension
\[
\dim \mathscr{M}_P(q^-,q^+) = \ind(q^-;\mathbb{S}) - \ind(q^+;\mathbb{S})+1.
\]
The choice of orientations $o_{q^-}$ and $o_{q^+}$ of $W^u(q^-;G_{\mathbb{S}})$ and $W^u(q^-;G_{\mathbb{S}})$ determines an orientation of $\mathscr{M}_P(q^-,q^+)$.
The action estimates
\begin{equation}
\label{act2bis}
\mathbb{S}(q^-) \geq \mathbb{S}(\pi\circ u(0,\cdot)) \geq \mathbb{A}(u(0,\cdot)) \geq \mathbb{A}(u(\sigma,\cdot)) = \mathbb{S}(\pi\circ u(\sigma,\cdot)) \geq S(q^+)
\end{equation}
imply uniform energy bounds for the elements of $\mathscr{M}_P(q^-,q^+)$, from which one can deduce the compactness of the sequences $(\sigma_n,u_n)$ in  $\mathscr{M}_P(q^-,q^+)$ for which $\sigma_n$ is bounded and bounded away from zero. Compactness and transversality allow to prove that when $\ind(q^+;\mathbb{S}) = \ind(q^-;\mathbb{S}) + 1$, $\mathscr{M}_P(q^-,q^+)$ consists of finitely many points. The orientations $o_{q^-}$ and $o_{q^+}$ allow to count these points algebraically and to obtain the integer $n_P(o_{q^-},o_{q^+})$. The homomorphism $P: M_*(\mathbb{S}) \rightarrow M_{*+1}(\mathbb{S})$ is defined by
\[
P o_{q^-} := \sum_{\substack{q^+ \in \crit \, \mathbb{S}\\ \ind(q^+;\mathbb{S}) = \ind(q^-;\mathbb{S}) + 1}} n_P(o_{q^-},o_{q^+})\, o_{q^+}, \qquad \forall q^- \in \crit \, \mathbb{S}.
\]
By definition and by (\ref{act2bis}), the above sum involves only generators $q^+$ such that $[q^+]=[q^-]$ and $\mathbb{S}(q^+)\leq \mathbb{S}(q^-)$.
The identity (\ref{he}) can be proven by a standard cobordism argument, by analyzing the limiting behavior of the non-compact sequences in the 1-dimensional manifold $\mathscr{M}_P(q^-,q^+)$, for $\ind(q^+;\mathbb{S}) = \ind(q^-;\mathbb{S})$. In fact, non-compact sequences of the form $(\sigma_n,u_n)$ with $\sigma_n\rightarrow +\infty$ contribute to the coefficient of $o_{q^+}$ in either $\Psi\Phi o_{q^-}$, or $P\partial^M o_{q^-}$, or $\partial^M P o_{q^-}$. Sequences $(\sigma_n,u_n)$ such that $\sigma_n\rightarrow 0$ produce instead a negative pseudo-gradient flow line connecting $q^-$ to $q^+$, and since these critical points have the same Morse index, such a flow line exists if and only if $q^-=q^+$ and in this case it is constant. This implies that the non-compact sequences of the latter form contribute to the identity operator on $M_*(\mathbb{S};\mathcal{G})$ (see \cite[Proposition 4.10]{as10} for more details on how to deal with the non-compact sequences of the latter form). 
\end{proof}

\begin{rem}
\label{invchom0}
Also the chain homotopy between the composition $\Phi\Psi$ and the identity mapping on $F_*(H)$ can be defined by a counting process. One would be tempted to choose, for every $x,y\in F_*(H)$, the space of triplets $(\sigma,u,v)$ where $\sigma$ is a positive number, $u$ is an element of $\mathscr{U}(x)$, and $v:\R^+\times \T\rightarrow T^*M$ is a solution of (\ref{cr}) which is asymptotic to $y$ for $s\rightarrow +\infty$ and is such that 
\[
\pi\circ v(0,\cdot)=\phi_{G_{\mathbb{S}}}(\sigma,\pi\circ u(0,\cdot)).
\]
However, this choice would define a chain homotopy between $\Phi\Psi$ and a chain isomorphism on $F_*(H)$ which needs not be the identity. Indeed, the latter chain map would count the set of pairs $(u,v)$ where $u\in \mathscr{U}(x)$, $v:\R^+\times \T\rightarrow T^*M$ is a solution of (\ref{cr}) which is asymptotic to $y$ for $s\rightarrow +\infty$ and such that $\pi\circ v(0,\cdot) = \pi\circ u(0,\cdot)$. In the case $\mu_{CZ}(y)=\mu_{CZ}(x)$, nothing prevents the existence of non-stationary solutions for the latter problem, provided that $\mathbb{S}(y)< \mathbb{S}(x)$. Therefore, the above chain map would just be another isomorphism which preserves the action filtrations, and one would need a second chain homotopy to show that this isomorphism is homotopic to the identity operator on $F_*(H)$. This asymmetry finds its explanation in the next section and in Remark \ref{invchom} below, where we show how to correctly define a chain homotopy between $\Phi\Psi$ and the identity mapping on $F_*(H)$.
\end{rem}

\section{Interpretation of the chain maps $\Phi$ and $\Psi$} 
\label{heursec}

Let $H\in C^{\infty}(\T,T^*M)$ be a uniformly convex Hamiltonian which is quadratic at infinity and whose Hamiltonian vector field has only non-degenerate 1-periodic orbits. Let $L\in C^{\infty}(\T,TM)$ be its Legendre-dual Lagrangian. In order to explain the heuristic ideas which lie behind the choice of the spaces which define the chain maps $\Phi$ and $\Psi$, we start by showing that $\mathbb{A}$ can be seen as a continuously differentiable functional on a suitable Hilbert manifold, on which it admits a negative pseudo-gradient vector field with a well-defined Morse complex. Since our aim is to motivate the choice of the spaces $\mathscr{M}_{\Phi}$ and $\mathscr{M}_{\Psi}$, throughout this section we ignore orientations and we consider the periodic orbits, respectively the critical points, as the generators of the Floer complex, respectively the Morse complex.

\paragraph{A negative pseudo-gradient vector field for $\mathbb{A}$.} Let $\Lambda(T^*M)$ be the Hilbert vector bundle over the Hilbert manifold $H^1(\T,M)$ whose fiber at $q\in H^1(\T,M)$ is the Hilbert space of $L^2$-sections of the vector bundle $q^*(T^*M)$. The Hamiltonian action, which can be written as
\[
\mathbb{A}(q,p) = \int_{\T} \Bigl( \langle p, q' \rangle - H(t,q,p) \Bigr)\, dt,
\]
extends to a continuous G\^{a}teaux-differentiable functional on $\Lambda(T^*M)$. We claim that such a functional admits a Morse-Smale negative pseudo-gradient vector field $G_{\mathbb{A}}$, which has a well-defined Morse complex, naturally identified with the Morse complex of the Lagrangian action functional $\mathbb{S}$.

Indeed, let us consider the Hilbert vector bundle $\Lambda(TM)$ over $H^1(\T,M)$ whose fiber at $q\in H^1(\T,M)$ is the Hilbert space of $L^2$-sections of $q^*(TM)$. The Legendre transform induces the following fiber-preserving diffeomorphism 
\[
\mathcal{L} : \Lambda(TM) \rightarrow \Lambda(T^*M), \quad \mathcal{L}(q,v) := \bigl(q, d_v L(\cdot,q,q'+v)\bigr).
\]
By the Legendre duality formula (\ref{legtra}), the composition $\mathbb{A}\circ \mathcal{L}$ has the form
\begin{eqnarray*}
\mathbb{A}\bigl(\mathcal{L}(q,v)\bigr) &=& \int_{\T} \Bigl( \langle d_v L(t,q,q'+v), q' \rangle - H\bigl(t,q,d_v L(t,q,q'+v)\bigr) \Bigr) \, dt \\
&=& \int_{\T} \Bigl( \langle d_v L(t,q,q'+v), q' \rangle - \langle d_v L(t,q,q'+v), q'+v \rangle + L(t,q,q'+v) \Bigr) \, dt \\ &=& \int_{\T} \Bigl( L(t,q,q'+v) - \langle d_v L(t,q,q'+v),v \rangle \Bigr)\, dt.
\end{eqnarray*}
By the Taylor formula with integral remainder, we have
\[
L(t,q,q') = L(t,q,q'+v) + \langle d_v L(t,q,q'+v), -v \rangle + \int_0^1 s \, d_{vv} L(t,q,q'+sv)[v]^2\, ds,
\]
so we can rewrite $\mathbb{A}(\mathcal{L}(q,v))$ as
\[
\mathbb{A}\bigl(\mathcal{L}(q,v)\bigr) = \int_{\T} L(t,q,q')\, dt - \int_{\T} \int_0^1 s\, d_{vv} L(t,q,q'+sv)[v]^2\, ds \, dt.
\]
Therefore, $\mathbb{A}\circ \mathcal{L}$ has the form
\begin{equation}
\label{compo}
\mathbb{A}\bigl(\mathcal{L}(q,v)\bigr) = \mathbb{S}(q) - \mathbb{U}(q,v), 
\end{equation}
where $\mathbb{U}$ is the function
\[
\mathbb{U} : \Lambda(TM) \rightarrow \R, \qquad \mathbb{U}(q,v) :=  \int_{\T} \int_0^1 s\, d_{vv} L(t,q,q'+sv)[v]^2\, ds \, dt.
\]

\begin{rem}
When $H$ is the physical Hamiltonian (\ref{phys}), the function $\mathbb{U}$ takes the simple form
\[
\mathbb{U}(q,v) = \frac{1}{2} \int_{\T} g_q(v,v)\, dt.
\]
This fact, in the particular case $\alpha=0$ and $V=0$, is used by M.\ Lipyanskiy in \cite{lip09}.
\end{rem}

Since $d_{vv} L>0$, the function $\mathbb{U}$ is non-negative and it assumes its global minimum 0 on the zero section of $\Lambda(TM)$. Since the diffeomorphism $\mathcal{L}$ is fiber preserving, the identity (\ref{compo}) implies that the fiber-wise differential of $\mathbb{U}$ at $(q,v)\in  \Lambda(TM)$ vanishes if and only if the fiberwise-differential of $\mathbb{A}$ at $\mathcal{L}(q,v)$ vanishes. The latter fact happens if and only if
\[
q' = d_p H\bigl(t,q, d_v L(t,q,q'+v)\bigr) = q' + v,
\]
hence if and only if $v=0$. Therefore, the restriction of $\mathbb{U}$ to each fiber
$\Pi^{-1}(q)$, where $\Pi$ is the bundle projection on the basis
\[
\Pi: \Lambda(TM) \rightarrow H^1(\T,M),
\]
is a non-negative function with a unique minimum at $0$ and no other critical point. In particular, all the critical points of $\mathbb{A}\circ \mathcal{L}$ lie on the zero section of $\Lambda(TM)$.    

Let $G_{\mathbb{S}}$ be a smooth complete Morse-Smale negative pseudo-gradient vector field for $\mathbb{S}$ on $H^1(\T,M)$. 
By (\ref{compo}), the function $\mathbb{A} \circ \mathcal{L}$ on $\Lambda(TM)$ admits a smooth pseudo-gradient vector field $G_{\mathbb{A}\circ \mathcal{L}}$ which is tangent to the zero section, on which it coincides with $G_{\mathbb{S}}$, such that the flows $\phi_{G_{\mathbb{S}}}$ and $\phi_{G_{\mathbb{A}\circ \mathcal{L}}}$ of $G_{\mathbb{S}}$ and $G_{\mathbb{A}\circ \mathcal{L}}$ are related by
\[
\Pi \circ \phi_{G_{\mathbb{A}\circ \mathcal{L}}} (s,x) = \phi_{G_{\mathbb{S}}} (s, \Pi (x) ), \qquad \forall x\in \Lambda(TM),\; \forall s\in \R,
\]
and such that $\phi_{G_{\mathbb{A}\circ \mathcal{L}}} (s,x)$ converges to the zero section for $s\rightarrow -\infty$ and escapes at infinity for $s\rightarrow +\infty$. It follows that for every critical point $(q,0)$ of $\mathbb{A}\circ \mathcal{L}$ there holds
\begin{eqnarray}
\label{wu}
W^u\bigl( (q,0); G_{\mathbb{A}\circ \mathcal{L}} \bigr) &=& \Pi^{-1} \bigl( W^u(q;G_{\mathbb{S}}) \bigr), \\
\label{ws}
W^s\bigl( (q,0); G_{\mathbb{A}\circ \mathcal{L}} \bigr) &=& W^s(q;G_{\mathbb{S}})\subset \mathbb{O}_{\Lambda(TM)},
\end{eqnarray}
where $\mathbb{O}_{\Lambda(TM)}= H^1 ( \T, \mathbb{O}_{TM})$ denotes the zero section of $\Lambda(TM)$. The stable and unstable manifolds of the singular points of $G_{\mathbb{A}\circ \mathcal{L}}$ are infinite-dimensional, but (\ref{wu}) and (\ref{ws}) show that their pairwise intersections are transverse and finite dimensional, and that they define a Morse complex which is precisely the Morse complex of $\mathbb{S}$, which is determined by the pseudo-gradient vector field $G_{\mathbb{S}}$. The push-forward of the vector field $G_{\mathbb{A}\circ \mathcal{L}}$ by the diffeomorphism $\mathcal{L}$, that is the vector field 
\[
G_{\mathbb{A}} := \mathcal{L}_* \bigl( G_{\mathbb{A}\circ \mathcal{L}}\bigr),
\]
is a Morse-Smale negative pseudo-gradient vector field for $\mathbb{A}$ on $\Lambda(T^*M)$, and has a well-defined Morse complex, which is naturally identified with the Morse complex of $\mathbb{S}$, as claimed. Since the diffeomorphism $\mathcal{L}$ is fiber-preserving, 
(\ref{wu}) implies that the unstable manifold of a critical point $x$ of $\mathbb{A}$ on $\Lambda(T^*M)$ is the set
\begin{equation}
\label{wuu}
W^u(x; G_{\mathbb{A}}) = \mathcal{L} \bigl( W^u ( \mathcal{L}^{-1}(x); G_{\mathbb{A}\circ \mathcal{L}} ) \bigr) = \Pi^{-1} \bigl(W^u(\pi\circ x ;G_{\mathbb{S}}) \bigr), 
\end{equation}
where $\Pi:\Lambda(T^*M) \rightarrow H^1(\T,M)$ denotes the projection onto the basis, $\Pi(x)=\pi\circ x$. By the form of $\mathcal{L}$ and by (\ref{ws}), the stable manifold of $x\in \crit \,\mathbb{A}$ is the set
\begin{equation}
\label{wss}
\begin{split}
W^s(x; G_{\mathbb{A}})  &=  \mathcal{L} \bigl( W^s ( \mathcal{L}^{-1}(x); G_{\mathbb{A}\circ \mathcal{L}} ) \bigr) \\
&= 
\set{(q,p)\in \Lambda(T^*M)}{q\in W^s(\pi\circ x;G_{\mathbb{S}}), \; p = d_v L(\cdot,q,q')} \\ &= \set{(q,p)\in \Lambda(T^*M)}{q\in W^s(\pi\circ x;G_{\mathbb{S}}), \; q' = d_p H(\cdot,q,p)}.
\end{split}
\end{equation}

\begin{rem}
\label{thom}
By the form (\ref{compo}) of the functional $\mathbb{A}\circ \mathcal{L}$, the identification between the Morse complex of $\mathbb{S}$ and the Morse complex of $\mathbb{A}\circ \mathcal{L}$ should be seen as the Morse theoretical interpretation of the equivalent of the Thom isomorphism for the infinite dimensional vector bundle $\Lambda(TM)\rightarrow H^1(\T,M)$. The picture for $\mathbb{A}$ is the same, but the zero-section of $\Lambda(TM)$ is replaced by the section $q\mapsto d_v L(\cdot,q,q')$ of $\Lambda(T^*M)$. This fact has been observed in \cite{lip09}. See also \cite{kra07}.
\end{rem} 

\paragraph{The chain maps $\Phi$ and $\Psi$.} We recall that, when $G_1$ and $G_2$ are two negative pseudo-gradient vector fields for the the same functional on a Hilbert manifold, a chain map from the Morse complex induced by $G_1$ to the one induced by $G_2$ can be often defined by counting the intersections 
\[
W^u(x;G_1)\cap W^s(y;G_2)
\]
of the unstable manifold of the critical point $x$ with respect to the first flow with the stable manifold of the critical point $y$ with respect to the second flow. In our case, the Hamiltonian action functional $\mathbb{A}$ has the negative pseudo-gradient vector field $G_{\mathbb{A}}$, which was constructed above and has a well-defined Morse complex $M_*(\mathbb{A})\cong M_*(\mathbb{S})$, and the $L^2$-negative gradient equation 
\begin{equation}
\label{cr2}
\partial_s u + J(t,u) \bigl( \partial_t u - X_H(t,u) \bigr) = 0,
\end{equation}
which produces the Floer complex $F_*(H)$. Although the latter equation does not determine a flow, one may interpret the image of the evaluation map $u\mapsto u(0,\cdot)$ on the set 
\begin{equation}
\label{fs}
\set{u:\R^+\times \T\rightarrow T^*M}{u \mbox{ solves (\ref{cr2}) and } u(s,\cdot) \rightarrow y \mbox{ for } s\rightarrow +\infty}
\end{equation}
as the stable manifold of the critical point $y$. Analogously, the image of the evaluation map $u\mapsto u(0,\cdot)$ on the set 
\begin{equation}
\label{fu}
\set{u: \R^-\times \T\rightarrow T^*M}{u \mbox{ solves (\ref{cr2}) and } u(s,\cdot) \rightarrow x \mbox{ for } s\rightarrow -\infty}
\end{equation}
should be interpreted as the unstable manifold of $x$. By this interpretation and by the general procedure recalled above, the  chain map $\Phi$ from $M_*(\mathbb{A})\cong M_*(\mathbb{S})$ to $F_*(H)$ should be defined by the following counting process: The coefficient of $y$ in $\Phi x$ should be obtained by counting  the elements of the set 
\[
\set{u}{u \mbox{ belongs to (\ref{fs}) and } u(0,\cdot) \in W^u (x; G_{\mathbb{A}}) }.
\]
By (\ref{wuu}), this set is precisely the set $\mathscr{M}_{\Phi}(\pi\circ x,y)$ which we have used in the definition of the chain map $\Phi$.

Similarly,  the coefficient of $y$ in $\Psi x$ should be the algebraic count of the elements of the set
\[
\set{u}{u \mbox{ belongs to (\ref{fu}) and } u(0,\cdot) \in W^s (y; G_{\mathbb{A}}) }.
\]
By (\ref{wss}), this is precisely the set $\mathscr{M}_{\Psi}(x, \pi\circ y) =
\mathscr{Q}_x^{-1}(W^s(\pi\circ y;G_{\mathbb{S}}))$ that we have used in the definition of the chain map $\Psi$.

\begin{rem} 
\label{invchom}
The above arguments suggest the correct way to construct a chain homotopy between the composition $\Phi\Psi$ and the identity on $F_*(H)$ (see Remark \ref{invchom0}). One should consider, for each pair $x,y\in \mathscr{P}(H)$, the space of triplets $(\sigma,u,v)$, where $\sigma$ is a positive number, $u$ belongs to (\ref{fu}), $v$ belongs to (\ref{fs}), and the coupling condition
 \[
 v(0,\cdot) = \phi_{G_{\mathbb{A}}} (\sigma, u(0,\cdot))
 \]
 holds. This space is different from the one suggested in Remark \ref{invchom0}, for instance because here $u$ needs not belong to the reduced unstable manifold $\mathscr{U}(x)$.
 \end{rem}
 
\begin{rem}
The functional $\mathbb{A}$ on $\Lambda(T^*M)$ fits into the Morse theory for functional whose critical points have infinite Morse index which was developed by the first author and P.\ Majer in \cite{ama05}. In particular, the relevant subbundle of the tangent bundle of the domain $\Lambda(T^*M)$ is the vertical subbundle $\ker D\Pi$. The construction of the negative pseudo-gradient vector field $G_{\mathbb{A}}$ that we have described here uses the fiberwise uniform convexity assumption on $H$ in an essential way. It seems natural to ask whether one can use the techniques of \cite{ama05} in order to define a Morse complex for more general Hamiltonians (for instance, assuming only quadraticity at infinity, that is the assumption used to define the Floer complex). It turns out that this is indeed possible, but the right function space is the Hilbert bundle over $H^s(\T,M)$ of $H^{1-s}$-sections of $q^*(T^*M)$, for $1/2<s<1$, instead of the case $s=1$ considered above. In fact with this choice, the dominant term in the action functional is the integral of the Liouville form and the integral of the Hamiltonian behaves as a compact perturbation.
 \end{rem}
   
\renewcommand{\thesection}{\Alph{section}}
\setcounter{section}{0}
\section{Appendix: The Morse complex with a local system of coefficients}
\label{app}

In this appendix we review the notion of bundle of groups $\mathcal{G}$ (or  system of local coefficients) over a topological space and the definition of singular homology with local coefficients as defined by N.\ S.\ Steenrod in \cite{ste43}. We follow the presentation of \cite[Section 5.3]{mcc01}. Then we review the construction of the Morse complex for a gradient-like flow on a Hilbert manifold which is endowed with a system of local coefficients. More details about the Morse complex in such an infinite dimensional setting are contained in \cite{ama06m}, while the use of local coefficients in Morse homology is described in \cite[Section 7.2]{oan08}.

\paragraph{Bundles of groups.} Let $B$ be a topological space. A {\em bundle of groups $\mathcal{G}$ over $B$} is a collection of groups $\set{G_b}{b\in B}$ together with a collection of isomorphisms
\[
h[\alpha] : G_{b_1} \rightarrow G_{b_0}
\]
for every continuous path $\alpha:[0,1] \rightarrow B$ joining $\alpha(0)=b_0$ and $\alpha(1)=b_1$, such that the following conditions hold:
\begin{enumerate}
\item if $\alpha_0$ is the constant path at $b\in B$, then $h[\alpha_0] = \mathrm{id}: G_b \rightarrow G_b$;
\item if $\alpha$ and $\beta$ are homotopic with fixed end-points, then $h[\alpha]=h[\beta]$;
\item if $\alpha$ and $\beta$ are paths satisfying $\alpha(1)=\beta(0)$ and 
\[
\alpha * \beta (t) := \left\{ \begin{array}{ll} \alpha(2t) & \mbox{if } 0\leq t \leq 1/2, \\\beta(2t-1) & \mbox{if } 1/2 \leq t \leq 1, \end{array} \right.
\]
denotes their product path, then $h[\alpha * \beta] = h[\alpha] \circ h[\beta]$.
\end{enumerate}
One refers to $\mathcal{G}$ also as a {\em system of local coefficients on $B$}.

Assume that $B$ is path connected. We fix a point $b_* \in B$ and a path $\beta_b$ such that $\beta_b(0)=b$ and $\beta_b(1)=b_*$ for every $b\in B$. Then any representation
\[
\rho: \pi_1(B,b_*) \rightarrow \mathrm{Aut}(G)
\]
of the fundamental group of $B$ into the group of automorphisms of the group $G$ gives rise to a bundle of groups over $B$. Indeed, we define $G_b:=G$ for every $b\in B$ and for every path $\alpha:[0,1]\rightarrow B$ with $b_0:=\alpha(0)$ and $b_1:= \alpha(1)$ we set
\[
h[\alpha] := \rho [ \beta_{b_0}^{-1} * \alpha * \beta_{b_1} ].   
\]
Conversely, any bundle of groups gives rise to a representation
\[
\rho': \pi_1 (B,b_0) \rightarrow \mathrm{Aut}(G_{b_0})
\]
for every $b_0\in B$.

\paragraph{Singular homology with local coefficients.} Fix a bundle of abelian groups $\mathcal{G}$ over a topological space $B$. Let 
\[
\Delta^p := \set{\sum_{j=0}^p x_j e_j}{x_j \geq 0 \mbox{ for every $j$ and } \sum_{j=0}^p x_j = 1}
\]
denote the standard $p$-symplex in $\R^{p+1}$, which is endowed with the standard basis $e_0,e_1, \dots,e_p$. The group $C_p(B;\mathcal{G})$ of {\em singular $p$-chains with coefficients in the bundle $\mathcal{G}$} is defined as the free abelian group generated by the elements
\[
g\otimes \sigma, \quad \mbox{where } \sigma:\Delta^p \rightarrow B \mbox{ and } g\in G_{\sigma(e_0)}.
\]
The $j$-th face of the singular simplex $\sigma:\Delta^p \rightarrow B$, $0\leq j \leq p$, is defined as usual as
\[
\partial_j \sigma := \sigma \circ \epsilon_j \mbox{ where } \epsilon_j:\Delta^{p-1} \hookrightarrow \Delta^p, \; (x_0,\dots, x_{p-1}) \mapsto (x_0,\dots, x_{j-1}, 0, x_j, \dots, x_{p-1} ).
\]
Notice that $\partial_j (e_0) = e_0$ for every $j=1,\dots,p$, while
$\partial_0 \sigma(e_0)=e_1$. If $\sigma: \Delta^p \rightarrow B$ is a singular simplex, we denote by
\[
\alpha_{\sigma} : [0,1] \rightarrow B, \quad t \mapsto \sigma \bigl(t e_0 + (1-t) e_1 \bigr),
\]
the path joining $\sigma(e_1)$ to $\sigma(e_0)$. The boundary homomorphism
\[
\partial_p : C_p(B;\mathcal{G}) \rightarrow C_{p-1}(B;\mathcal{G})
\]
is defined on the set of generators as
\[
\partial_p (g\otimes \sigma) := h[ \alpha_{\sigma}] (g) \otimes \partial_0 \sigma + \sum_{j=1}^p (-1)^j g \otimes \partial_j \sigma,
\]
for every simplex $\sigma:\Delta^p \rightarrow B$ and $g\in G_{\sigma(e_0)}$. One easily shows that $\partial_p \circ \partial_{p+1}=0$, so $\{C_*(B;\mathcal{G}),\partial_*\}$ is a chain complex of abelian groups. Its homology
\[
H_*(B;\mathcal{G}) := H\bigl(  C_*(B;\mathcal{G}),\partial_* \bigr)
\]
is said to be the {\em homology of $B$ with local coefficients in $\mathcal{G}$}.

\paragraph{The Morse complex with local coefficients.} Let $\mathcal{M}$ be a manifold modeled on a Hilbert space and let $f$ be a continuously differentiable real function on $\mathcal{M}$ which is bounded from below. We assume that $f$ admits a nice negative pseudo-gradient vector field $X$, that is a smooth tangent vector field on $\mathcal{M}$ which satisfies the following properties:

\begin{enumerate}

\item The flow $(t,x) \mapsto \phi(t,x)$ of $X$ exists for all $t\geq 0$. 

\item The set of critical points $\mathrm{crit}(f)$ coincides with the set of points $x\in \mathcal{M}$ where $X(x)=0$; $df(p)[X(p)]<0$ for every $p\in \mathcal{M}\setminus \mathrm{crit}(f)$. 

\item $X$ is Morse, meaning that the spectrum of the Jacobian $\nabla X(x): T_x \mathcal{M} \rightarrow T_x \mathcal{M}$ of $X$ at each $x\in \mathrm{crit}(f)$ 
does not meet the imaginary axis $i\R$. 

\item Every critical point $x\in \mathrm{crit}(f)$ has finite Morse index $\mathrm{ind}(x)$, where the Morse index $\mathrm{ind}(x)$ is defined to be the dimension of  the $\nabla X(x)$-invariant subspace of $T_x \mathcal{M}$ which is associated to the part of the spectrum of $\nabla X(x)$ having positive real part is finite-dimensional. In particular, the unstable (resp.\ stable) manifold of $x$
\[
\begin{split}
W^u(x) &:= \set{p\in \mathcal{M}}{\phi(t,p) \rightarrow x \mbox{ for } t\rightarrow -\infty} \\ \bigl( \mbox{resp. } W^s(x) &:= \set{p\in \mathcal{M}}{\phi(t,p) \rightarrow x \mbox{ for } t\rightarrow +\infty} \bigr)
\end{split} \]
has dimension (resp.\ codimension) $\mathrm{ind}(x)$ (thanks to (ii) these are embedded submanifolds).

\item $X$ is Morse-Smale, meaning that for every $x,y\in \crit(f)$ the submanifolds $W^u(x)$ and $W^s(y)$ intersect transversally.

\item The pair $(X,f)$ satisfies the Palais-Smale condition: every sequence $(p_n) \subset \mathcal{M}$ such that $(f(p_n))$ is bounded and $(df(p_n)[X(p_n)])$ is infinitesimal has a converging subsequence.

\end{enumerate}

Given $k\in \N$ let $\mathrm{crit}_k(f)$ be the set of critical points of $f$ of Morse index $k$. 
Let 
\[
\mathcal{G}= \bigl(\{G_p\}_{p\in \mathcal{M}, \{h[\alpha]}\bigr) 
\]
be a bundle of abelian groups over $\mathcal{M}$. Let $M_k(f;\mathcal{G})$ be the abelian group which is obtained from group
\[ 
\bigoplus_{\substack{x\in \mathrm{crit}_k(f) \\ o_x \; \mathrm{orientation} \; \mathrm{of}\; W^u(x)}} o_x \otimes G_x
\] 
by taking the quotient by the identification
\[
g_x \otimes \widehat{o}_x = - g_x\otimes o_x, \qquad \forall x\in \mathrm{crit}_k(f), \; g_x \in G_x, \; o_x \mbox{ orientation of } W^u(x),
\]
where $\widehat{o}_x$ denotes the orientation of $W^u(x)$ which is opposite to $o_x$. A choice of an orientation for the unstable manifold of each $x\in \mathrm{crit}_k(f)$ defines an isomorphism between $M_k(f;\mathcal{G})$ and 
\[
\bigoplus_{x\in \mathrm{crit}_k(f)} G_x.
\] 
Let $x,y$ be critical points of $f$ with $\mathrm{ind}(x) - \mathrm{ind}(y)=1$. Then $W^u(x)\cap W^s(y)$ is a one-dimensional manifold and consists of finitely many curves joining $x$ to $y$: denote by
\[
\Gamma(x,y) \subset C^0([0,1],\mathcal{M})
\]
the finite set consisting of a choice of a parametrization for each of these curves (where $\gamma\in \Gamma(x,y)$ implies that $\gamma(0)=x$ and $\gamma(1)=y$). 

An orientation $o_x$ of $W^u(x)$ and a path $\gamma\in \Gamma(x,y)$ induce an orientation $\gamma_*(o_x)$ of $W^u(y)$. Indeed, since
\[
T_y \mathcal{M} = T_y W^u(y) \oplus T_y W^s(y),
\]
an orientation of $W^u(y)$ can be identified with a co-orientation of $W^s(y)$ (that is, an orientation of the normal bundle of $W^s(y)$ in $\mathcal{M}$) and $\gamma_*(o_x)$ is defined as the co-orientation of $W^s(y)$ which, together with the orientation $o_x$ of $W^u(x)$, determines the orientation of the intersection $W^u(x)\cap W^s(y)$ which at $\gamma$ agrees with the direction of the flow of $X$.

The homomorphism
\[
\partial_k : M_k(f;\mathcal{G}) \rightarrow M_{k-1}(f;\mathcal{G})
\]
is defined generator-wise as
\[
\partial_k ( o_x \otimes g_x )= \sum_{y\in \mathrm{crit}_{k-1}(f)} \sum_{\gamma\in \Gamma(x,y)} \gamma_*(o_x) \otimes h[\gamma](g_x).
\]
It satisfies $\partial_k \circ \partial_{k+1}=0$, so $\{ M_*(f;\mathcal{G}),\partial_*\}$ is a chain complex of abelian groups, which is called the {\em Morse complex of $(f,X)$ with local coefficients in $\mathcal{G}$}. The choice of a different pseudo-gradient vector field for $f$ produces an isomorphic chain complex. In particular, the homology
\[
HM_*(f;\mathcal{G}) := H_* \bigl( M_*(f;\mathcal{G}),\partial_* \bigr)
\]
does not depend on $X$ and is called the {\em Morse homology of $f$  with local coefficients in $\mathcal{G}$}.

This homology is isomorphic to the singular homology of $\mathcal{M}$ with coefficients in $\mathcal{G}$:
\[
HM_*(f;\mathcal{G}) \cong H_*(B;\mathcal{G}).
\]


\providecommand{\bysame}{\leavevmode\hbox to3em{\hrulefill}\thinspace}
\providecommand{\MR}{\relax\ifhmode\unskip\space\fi MR }
\providecommand{\MRhref}[2]{%
  \href{http://www.ams.org/mathscinet-getitem?mr=#1}{#2}
}
\providecommand{\href}[2]{#2}

\end{document}